\documentclass[numbers=enddot,12pt,final,onecolumn,notitlepage]{scrartcl}%
\usepackage[headsepline,footsepline,manualmark]{scrlayer-scrpage}
\usepackage{amsfonts}
\usepackage{amssymb}
\usepackage{amsmath}
\usepackage{amsthm}
\usepackage{framed}
\usepackage{comment}
\usepackage{color}
\usepackage[breaklinks=true]{hyperref}
\usepackage[sc]{mathpazo}
\usepackage[T1]{fontenc}
\usepackage{needspace}
\usepackage{tabls}
\usepackage{hyperxmp}
\providecommand{\U}[1]{\protect\rule{.1in}{.1in}}
\theoremstyle{definition}
\newtheorem{theo}{Theorem}[section]
\newenvironment{theorem}[1][]
{\begin{theo}[#1]\begin{leftbar}}
{\end{leftbar}\end{theo}}
\newtheorem{lem}[theo]{Lemma}
\newenvironment{lemma}[1][]
{\begin{lem}[#1]\begin{leftbar}}
{\end{leftbar}\end{lem}}
\newtheorem{prop}[theo]{Proposition}
\newenvironment{proposition}[1][]
{\begin{prop}[#1]\begin{leftbar}}
{\end{leftbar}\end{prop}}
\newtheorem{defi}[theo]{Definition}
\newenvironment{definition}[1][]
{\begin{defi}[#1]\begin{leftbar}}
{\end{leftbar}\end{defi}}
\newtheorem{remk}[theo]{Remark}
\newenvironment{remark}[1][]
{\begin{remk}[#1]\begin{leftbar}}
{\end{leftbar}\end{remk}}
\newtheorem{coro}[theo]{Corollary}
\newenvironment{corollary}[1][]
{\begin{coro}[#1]\begin{leftbar}}
{\end{leftbar}\end{coro}}
\newtheorem{conv}[theo]{Convention}

\newtheorem{quest}[theo]{Question}
\newenvironment{question}[1][]
{\begin{quest}[#1]\begin{leftbar}}
{\end{leftbar}\end{quest}}
\newtheorem{warn}[theo]{Warning}

\newtheorem{conj}[theo]{Conjecture}
\newenvironment{conjecture}[1][]
{\begin{conj}[#1]\begin{leftbar}}
{\end{leftbar}\end{conj}}
\newtheorem{exam}[theo]{Example}
\newenvironment{example}[1][]
{\begin{exam}[#1]\begin{leftbar}}
{\end{leftbar}\end{exam}}
\newcommand{\silentsection}{\section}
\newenvironment{verlong}{}{}
\newenvironment{vershort}{}{}

\excludecomment{verlong}
\includecomment{vershort}
\excludecomment{noncompile}

\let\sumnonlimits\sum
\let\prodnonlimits\prod
\renewcommand{\sum}{\sumnonlimits\limits}
\renewcommand{\prod}{\prodnonlimits\limits}

\ihead{Commutators of somewhere-to-below shuffles, version September 20, 2023}
\ohead{page \thepage}
\begin{document}

\title{Commutator nilpotency for somewhere-to-below shuffles}
\author{Darij Grinberg}
\date{version 2.0,
September 20, 2023
}
\maketitle

\begin{abstract}
\textbf{Abstract.} Given a positive integer $n$, we consider the group algebra
of the symmetric group $S_{n}$. In this algebra, we define $n$ elements
$t_{1},t_{2},\ldots,t_{n}$ by the formula%
\[
t_{\ell}:=\operatorname*{cyc}\nolimits_{\ell}+\operatorname*{cyc}%
\nolimits_{\ell,\ell+1}+\operatorname*{cyc}\nolimits_{\ell,\ell+1,\ell
+2}+\cdots+\operatorname*{cyc}\nolimits_{\ell,\ell+1,\ldots,n},
\]
where \medskip$\operatorname*{cyc}\nolimits_{\ell,\ell+1,\ldots,k}$ denotes
the cycle that sends $\ell\mapsto\ell+1\mapsto\ell+2\mapsto\cdots\mapsto
k\mapsto\ell$. These $n$ elements are called the \emph{somewhere-to-below
shuffles} due to an interpretation as card-shuffling operators.

In this paper, we show that their commutators $\left[  t_{i},t_{j}\right]
=t_{i}t_{j}-t_{j}t_{i}$ are nilpotent, and specifically that
\[
\left[  t_{i},t_{j}\right]  ^{\left\lceil \left(  n-j\right)  /2\right\rceil
+1}=0\ \ \ \ \ \ \ \ \ \ \text{for any }i,j\in\left\{  1,2,\ldots,n\right\}
\]
and%
\[
\left[  t_{i},t_{j}\right]  ^{j-i+1}=0\ \ \ \ \ \ \ \ \ \ \text{for any }1\leq
i\leq j\leq n.
\]
We discuss some further identities and open questions. \medskip

 \medskip


\end{abstract}
\tableofcontents

\section{Introduction}

The \emph{somewhere-to-below shuffles} $t_{1},t_{2},\ldots,t_{n}$ (and their
linear combinations, called the \emph{one-sided cycle shuffles}) are certain
elements in the group algebra of a symmetric group $S_{n}$. They have been
introduced in \cite{s2b1} by Lafreni\`{e}re and the present author, and are a
novel generalization of the top-to-random shuffle (also known as the
\emph{Tsetlin library}). They are defined by the formula%
\[
t_{\ell}:=\operatorname*{cyc}\nolimits_{\ell}+\operatorname*{cyc}%
\nolimits_{\ell,\ell+1}+\operatorname*{cyc}\nolimits_{\ell,\ell+1,\ell
+2}+\cdots+\operatorname*{cyc}\nolimits_{\ell,\ell+1,\ldots,n}\in
\mathbf{k}\left[  S_{n}\right]  ,
\]
where $\operatorname*{cyc}\nolimits_{\ell,\ell+1,\ldots,k}$ denotes the cycle
that sends $\ell\mapsto\ell+1\mapsto\ell+2\mapsto\cdots\mapsto k\mapsto\ell$
(and leaves all remaining elements of $\left\{  1,2,\ldots,n\right\}  $ unchanged).

One of the main results of \cite{s2b1} was the construction of a basis
$\left(  a_{w}\right)  _{w\in S_{n}}$ of the group algebra in which
multiplication by these shuffles acts as an upper-triangular matrix (i.e., for
which $a_{w}t_{\ell}$ equals a linear combination of $a_{u}$'s with $u\leq w$
for a certain total order on $S_{n}$). Consequences of this fact (or, more
precisely, of a certain filtration that entails this fact) include an explicit
description of the eigenvalues of each one-sided cycle shuffle, as well as
analogous properties of some related shuffles.

Another consequence of the joint triangularizability of $t_{1},t_{2}%
,\ldots,t_{n}$ is the fact that the commutators $\left[  t_{i},t_{j}\right]
:=t_{i}t_{j}-t_{j}t_{i}$ are nilpotent (since the commutator of two
upper-triangular matrices is strictly upper-triangular and thus nilpotent).
Explicitly, this means that $\left[  t_{i},t_{j}\right]  ^{n!}=0$, since the
$t_{1},t_{2},\ldots,t_{n}$ act on a free module of rank $n!$. However,
experiments have suggested that the minimal $m\in\mathbb{N}$ satisfying
$\left[  t_{i},t_{j}\right]  ^{m}=0$ is far smaller than $n!$, and in fact is
bounded from above by $n$.

In the present paper, we shall prove this. Concretely, we will prove the
following results (the notation $\left[  m\right]  $ means the set $\left\{
1,2,\ldots,m\right\}  $):

\begin{itemize}
\item \textbf{Corollary \ref{cor.right-bound}.} We have $\left[  t_{i}%
,t_{j}\right]  ^{\left\lceil \left(  n-j\right)  /2\right\rceil +1}=0$ for any
$i,j\in\left[  n\right]  $.

\item \textbf{Theorem \ref{thm.right-bound}.} Let $j\in\left[  n\right]  $ and
$m\in\mathbb{N}$ be such that $2m\geq n-j+2$. Let $i_{1},i_{2},\ldots,i_{m}$
be $m$ elements of $\left[  j\right]  $ (not necessarily distinct). Then,
\[
\left[  t_{i_{1}},t_{j}\right]  \left[  t_{i_{2}},t_{j}\right]  \cdots\left[
t_{i_{m}},t_{j}\right]  =0.
\]

\item \textbf{Corollary \ref{cor.left-bound}.} We have $\left[  t_{i}%
,t_{j}\right]  ^{j-i+1}=0$ for any $1\leq i\leq j\leq n$.

\item \textbf{Theorem \ref{thm.left-bound}.} Let $j\in\left[  n\right]  $, and
let $m$ be a positive integer. Let $k_{1},k_{2},\ldots,k_{m}$ be $m$ elements
of $\left[  j\right]  $ (not necessarily distinct) satisfying $m\geq
j-k_{m}+1$. Then,
\[
\left[  t_{k_{1}},t_{j}\right]  \left[  t_{k_{2}},t_{j}\right]  \cdots\left[
t_{k_{m}},t_{j}\right]  =0.
\]

\end{itemize}

Along the way, we will also prove the following helpful facts:

\begin{itemize}
\item \textbf{Theorem \ref{thm.1+sjtitj}.} We have $\left(  1+s_{j}\right)
\left[  t_{i},t_{j}\right]  =0$ for any $1\leq i\leq j<n$, where $s_{j}$
denotes the transposition swapping $j$ with $j+1$.

\item \textbf{Theorem \ref{thm.ti+1ti}.} For any $i\in\left[  n-1\right]  $,
we have $t_{i+1}t_{i}=\left(  t_{i}-1\right)  t_{i}=t_{i}\left(
t_{i}-1\right)  $.

\item \textbf{Theorem \ref{thm.ti+2ti}.} For any $i\in\left[  n-2\right]  $,
we have $t_{i+2}\left(  t_{i}-1\right)  =\left(  t_{i}-1\right)  \left(
t_{i+1}-1\right)  $.

\item \textbf{Corollary \ref{cor.titi+12=0}.} For any $i\in\left[  n-1\right]
$, we have $\left[  t_{i},t_{i+1}\right]  =t_{i}\left(  t_{i+1}-\left(
t_{i}-1\right)  \right)  $ and $\left[  t_{i},t_{i+1}\right]  t_{i}=\left[
t_{i},t_{i+1}\right]  ^{2}=0$.
\end{itemize}

These results can be regarded as first steps towards understanding the
$\mathbf{k}$-subalgebra $\mathbf{k}\left[  t_{1},t_{2},\ldots,t_{n}\right]  $
of $\mathbf{k}\left[  S_{n}\right]  $ that is generated by the
somewhere-to-below shuffles. So far, very little is known about this
$\mathbf{k}$-subalgebra, except for its simultaneous triangularizability (a
consequence of \cite[Theorem 4.1]{s2b1}). One might ask for its dimension as a
$\mathbf{k}$-module (when $\mathbf{k}$ is a field). Here is some numerical
data for $\mathbf{k}=\mathbb{Q}$ and $n\leq8$:
\begin{equation}%
\begin{tabular}
[c]{|c||c|c|c|c|c|c|c|c|}\hline
$n$ & $1$ & $2$ & $3$ & $4$ & $5$ & $6$ & $7$ & $8$\\\hline
$\dim\left(  \mathbb{Q}\left[  t_{1},t_{2},\ldots,t_{n}\right]  \right)  $ &
$1$ & $2$ & $4$ & $9$ & $23$ & $66$ & $212$ & $761$\\\hline
\end{tabular}
\label{eq.dimQ}%
\end{equation}
As of September 14th, 2023, this sequence of dimensions is not in the OEIS. We
note that each single somewhere-to-below shuffle by itself is easily
understood using the well-known theory of the top-to-random
shuffle\footnote{Indeed, the first somewhere-to-below shuffle $t_{1}$ is known
as the \emph{top-to-random shuffle}, and has been discussed, e.g., in
\cite{Hendri72, Donnel91, Phatar91, Fill96, BiHaRo99, DiFiPi92, Palmes10,
Grinbe18}. More generally, for each $\ell\in\left[  n\right]  $, the
somewhere-to-below shuffle $t_{\ell}$ is exactly the top-to-random shuffle of
the symmetric group algebra $\mathbf{k}\left[  S_{n-\ell+1}\right]  $,
transported into $\mathbf{k}\left[  S_{n}\right]  $ using the embedding
$S_{n-\ell+1}\hookrightarrow S_{n}$ that renames the numbers $1,2,\ldots
,n-\ell+1$ as $\ell,\ell+1,\ldots,n$. Thus, we know (e.g.) that the minimal
polynomial of $t_{\ell}$ over a characteristic-$0$ field $\mathbf{k}$ is
$\prod_{i=0}^{n-\ell-1}\left(  x-i\right)  \cdot\left(  x-\left(
n-\ell+1\right)  \right)  $ (by \cite[Theorem 4.1]{DiFiPi92}).}, but this
approach says nothing about the interactions between two or more of the $n$
somewhere-to-below shuffles.

\paragraph{Acknowledgements}

The author would like to thank Sarah Brauner and Nadia Lafreni\`{e}re for
inspiring discussions. The SageMath CAS \cite{sagemath} was indispensable at
every stage of the research presented here.

\section{Notations and notions}

\subsection{Basic notations}

Let $\mathbf{k}$ be any commutative ring. (The reader can safely take
$\mathbf{k}=\mathbb{Z}$.)

Let $\mathbb{N}:=\left\{  0,1,2,\ldots\right\}  $ be the set of all
nonnegative integers.

For any integers $a$ and $b$, we set
\[
\left[  a,b\right]  :=\left\{  k\in\mathbb{Z}\ \mid\ a\leq k\leq b\right\}
=\left\{  a,a+1,\ldots,b\right\}  .
\]
This is an empty set if $a>b$. In general, $\left[  a,b\right]  $ is called an
\emph{integer interval}.

For each $n\in\mathbb{Z}$, let $\left[  n\right]  :=\left[  1,n\right]
=\left\{  1,2,\ldots,n\right\}  $.

Fix an integer $n\in\mathbb{N}$. Let $S_{n}$ be the $n$-th symmetric group,
i.e., the group of all permutations of $\left[  n\right]  $. We multiply
permutations in the \textquotedblleft continental\textquotedblright\ way: that
is, $\left(  \pi\sigma\right)  \left(  i\right)  =\pi\left(  \sigma\left(
i\right)  \right)  $ for all $\pi,\sigma\in S_{n}$ and $i\in\left[  n\right]
$.

For any $k$ distinct elements $i_{1},i_{2},\ldots,i_{k}$ of $\left[  n\right]
$, we let $\operatorname*{cyc}\nolimits_{i_{1},i_{2},\ldots,i_{k}}$ be the
permutation in $S_{n}$ that sends $i_{1},i_{2},\ldots,i_{k-1},i_{k}$ to
$i_{2},i_{3},\ldots,i_{k},i_{1}$, respectively while leaving all remaining
elements of $\left[  n\right]  $ unchanged. This permutation is known as a
\emph{cycle}. Note that $\operatorname*{cyc}\nolimits_{i}=\operatorname*{id}$
for any single $i\in\left[  n\right]  $.

For any $i\in\left[  n-1\right]  $, we let $s_{i}:=\operatorname*{cyc}%
\nolimits_{i,i+1}\in S_{n}$. This permutation $s_{i}$ is called a \emph{simple
transposition}, as it swaps $i$ with $i+1$ while leaving all other elements of
$\left[  n\right]  $ unchanged. It clearly satisfies
\begin{equation}
s_{i}^{2}=\operatorname*{id}. \label{eq.si.invol}%
\end{equation}
Furthermore, two simple transpositions $s_{i}$ and $s_{j}$ commute whenever
$\left\vert i-j\right\vert >1$. This latter fact is known as \emph{reflection
locality}.

\subsection{Some elements of $\mathbf{k}\left[  S_{n}\right]  $}

Consider the group algebra $\mathbf{k}\left[  S_{n}\right]  $. In this
algebra, define $n$ elements $t_{1},t_{2},\ldots,t_{n}$ by setting\footnote{We
view $S_{n}$ as a subset of $\mathbf{k}\left[  S_{n}\right]  $ in the obvious
way.}%
\begin{equation}
t_{\ell}:=\operatorname*{cyc}\nolimits_{\ell}+\operatorname*{cyc}%
\nolimits_{\ell,\ell+1}+\operatorname*{cyc}\nolimits_{\ell,\ell+1,\ell
+2}+\cdots+\operatorname*{cyc}\nolimits_{\ell,\ell+1,\ldots,n}\in
\mathbf{k}\left[  S_{n}\right]  \label{eq.def.tl.deftl}%
\end{equation}
for each $\ell\in\left[  n\right]  $. Thus, in particular, $t_{n}%
=\operatorname*{cyc}\nolimits_{n}=\operatorname*{id}=1$ (where $1$ means the
unity of $\mathbf{k}\left[  S_{n}\right]  $). We shall refer to the $n$
elements $t_{1},t_{2},\ldots,t_{n}$ as the \emph{somewhere-to-below shuffles}.
These shuffles were studied in \cite{s2b1} (where, in particular, their
probabilistic meaning was discussed, which explains the origin of their name).

\subsection{Commutators}

If $a$ and $b$ are two elements of some ring, then $\left[  a,b\right]  $
shall denote their commutator $ab-ba$. This notation clashes with our
above-defined notation $\left[  a,b\right]  $ for the interval $\left\{
k\in\mathbb{Z}\ \mid\ a\leq k\leq b\right\}  $ (when $a$ and $b$ are two
integers), but we don't expect any confusion to arise in practice, since we
will only use the notation $\left[  a,b\right]  $ for $ab-ba$ when $a$ and $b$
are visibly elements of the ring $\mathbf{k}\left[  S_{n}\right]  $ (as
opposed to integers).

\section{Elementary computations in $S_{n}$}

In this section, we will perform some simple computations in the symmetric
group $S_{n}$. The results of these computations will later become ingredients
in some of our proofs.

\subsection{The cycles $\left(  v\Longrightarrow w\right)  $}

\begin{definition}
Let $v,w\in\left[  n\right]  $ satisfy $v\leq w$. Then, $\left(
v\Longrightarrow w\right)  $ shall denote the permutation $\operatorname*{cyc}%
\nolimits_{v,v+1,\ldots,w}$.
\end{definition}

The symbol \textquotedblleft$\Longrightarrow$\textquotedblright\ in this
notation $\left(  v\Longrightarrow w\right)  $ has nothing to do with logical
implication; instead, it is meant to summon an image of a \textquotedblleft
current\textquotedblright\ flowing from $v$ to $w$. The symbol
\textquotedblleft$\Longrightarrow$\textquotedblright\ is understood to bind
less strongly than addition or subtraction; thus, for example, the expression
\textquotedblleft$\left(  v+1\Longrightarrow w\right)  $\textquotedblright%
\ means $\left(  \left(  v+1\right)  \Longrightarrow w\right)  $.

Every $v\in\left[  n\right]  $ satisfies
\begin{equation}
\left(  v\Longrightarrow v\right)  =\operatorname*{cyc}\nolimits_{v}%
=\operatorname*{id}=1. \label{eq.p-cyc-p=1}%
\end{equation}
The following is just a little bit less obvious:

\begin{proposition}
\label{prop.cycpq.sss}Let $v,w\in\left[  n\right]  $ satisfy $v\leq w$. Then,
$\left(  v\Longrightarrow w\right)  =s_{v}s_{v+1}\cdots s_{w-1}$.
\end{proposition}

\begin{vershort}
\begin{proof}
Easy verification.
\end{proof}
\end{vershort}

\begin{verlong}
\begin{proof}
For any two distinct elements $i$ and $j$ of $\left[  n\right]  $, we let
$t_{i,j}$ denote the permutation in $S_{n}$ that swaps $i$ with $j$ while
leaving all other elements of $\left[  n\right]  $ unchanged. (This notation
has nothing to do with the somewhere-to-below shuffle $t_{\ell}$.) Note that
the permutation $t_{i,j}$ is called a \emph{transposition}.

It is clear that $t_{i,j}=\operatorname*{cyc}\nolimits_{i,j}$ for any two
distinct elements $i$ and $j$ of $\left[  n\right]  $. Thus, for each
$i\in\left[  n-1\right]  $, we have%
\begin{equation}
t_{i,i+1}=\operatorname*{cyc}\nolimits_{i,i+1}=s_{i}
\label{pf.prop.cycpq.sss.1}%
\end{equation}
(since $s_{i}$ is defined to be $\operatorname*{cyc}\nolimits_{i,i+1}$).

A known fact (see, e.g., \cite[Exercise 5.16]{detnotes}) shows that every
$k\in\left\{  1,2,\ldots,n\right\}  $ and every $k$ distinct elements
$i_{1},i_{2},\ldots,i_{k}$ of $\left[  n\right]  $ satisfy%
\[
\operatorname*{cyc}\nolimits_{i_{1},i_{2},\ldots,i_{k}}=t_{i_{1},i_{2}}\circ
t_{i_{2},i_{3}}\circ\cdots\circ t_{i_{k-1},i_{k}}.
\]
Applying this to $k=w-v+1$ and $\left(  i_{1},i_{2},\ldots,i_{k}\right)
=\left(  v,v+1,\ldots,w\right)  $, we obtain%
\begin{align*}
\operatorname*{cyc}\nolimits_{v,v+1,\ldots,w}  &  =\underbrace{t_{v,v+1}%
}_{\substack{=s_{v}\\\text{(by (\ref{pf.prop.cycpq.sss.1}),}\\\text{applied to
}i=v\text{)}}}\circ\underbrace{t_{v+1,v+2}}_{\substack{=s_{v+1}\\\text{(by
(\ref{pf.prop.cycpq.sss.1}),}\\\text{applied to }i=v+1\text{)}}}\circ
\cdots\circ\underbrace{t_{w-1,w}}_{\substack{=s_{w-1}\\\text{(by
(\ref{pf.prop.cycpq.sss.1}),}\\\text{applied to }i=w-1\text{)}}}\\
&  =s_{v}\circ s_{v+1}\circ\cdots\circ s_{w-1}=s_{v}s_{v+1}\cdots s_{w-1}%
\end{align*}
(since the product $\pi\sigma$ of two permutations $\pi,\sigma\in S_{n}$ is
precisely their composition $\pi\circ\sigma$).

However, the definition of $\left(  v\Longrightarrow w\right)  $ yields%
\[
\left(  v\Longrightarrow w\right)  =\operatorname*{cyc}\nolimits_{v,v+1,\ldots
,w}=s_{v}s_{v+1}\cdots s_{w-1}.
\]
This proves Proposition \ref{prop.cycpq.sss}.
\end{proof}
\end{verlong}

\begin{proposition}
\label{prop.cycpq.rec}Let $v,w\in\left[  n\right]  $ satisfy $v<w$. Then:

\begin{enumerate}
\item[\textbf{(a)}] We have $\left(  v\Longrightarrow w\right)  =s_{v}\left(
v+1\Longrightarrow w\right)  $.

\item[\textbf{(b)}] We have $\left(  v\Longrightarrow w\right)  =\left(
v\Longrightarrow w-1\right)  s_{w-1}$.
\end{enumerate}
\end{proposition}

\begin{vershort}

\begin{proof}
Easy verification (easiest using Proposition \ref{prop.cycpq.sss}).
\end{proof}
\end{vershort}

\begin{verlong}

\begin{proof}
We have $v<w$, so that $v\leq w-1$ (since $v$ and $w$ are integers) and thus
$v+1\leq w$. \medskip

\textbf{(a)} We have $v+1\leq w$. Thus, the permutation $\left(
v+1\Longrightarrow w\right)  $ is well-defined, and Proposition
\ref{prop.cycpq.sss} (applied to $v+1$ instead of $v$) yields
\begin{equation}
\left(  v+1\Longrightarrow w\right)  =s_{v+1}s_{\left(  v+1\right)  +1}\cdots
s_{w-1}=s_{v+1}s_{v+2}\cdots s_{w-1}. \label{pf.prop.cycpq.rec.1}%
\end{equation}
However, Proposition \ref{prop.cycpq.sss} yields
\[
\left(  v\Longrightarrow w\right)  =s_{v}s_{v+1}\cdots s_{w-1}=s_{v}%
\underbrace{\left(  s_{v+1}s_{v+2}\cdots s_{w-1}\right)  }_{\substack{=\left(
v+1\Longrightarrow w\right)  \\\text{(by (\ref{pf.prop.cycpq.rec.1}))}}%
}=s_{v}\left(  v+1\Longrightarrow w\right)  .
\]
This proves Proposition \ref{prop.cycpq.rec} \textbf{(a)}. \medskip

\textbf{(b)} We have $v\leq w-1$. Thus, the permutation $\left(
v\Longrightarrow w-1\right)  $ is well-defined, and Proposition
\ref{prop.cycpq.sss} (applied to $w-1$ instead of $w$) yields
\begin{equation}
\left(  v\Longrightarrow w-1\right)  =s_{v}s_{v+1}\cdots s_{\left(
w-1\right)  -1}=s_{v}s_{v+1}\cdots s_{w-2}. \label{pf.prop.cycpq.rec.2}%
\end{equation}
However, Proposition \ref{prop.cycpq.sss} yields
\[
\left(  v\Longrightarrow w\right)  =s_{v}s_{v+1}\cdots s_{w-1}%
=\underbrace{\left(  s_{v}s_{v+1}\cdots s_{w-2}\right)  }_{\substack{=\left(
v\Longrightarrow w-1\right)  \\\text{(by (\ref{pf.prop.cycpq.rec.2}))}%
}}s_{w-1}=\left(  v\Longrightarrow w-1\right)  s_{w-1}.
\]
This proves Proposition \ref{prop.cycpq.rec} \textbf{(b)}.
\end{proof}
\end{verlong}

\subsection{Rewriting rules for products of cycles}

Next we recall how conjugation in $S_{n}$ acts on cycles:

\begin{proposition}
\label{prop.cyc-conj}Let $\sigma\in S_{n}$. Let $i_{1},i_{2},\ldots,i_{k}$ be
$k$ distinct elements of $\left[  n\right]  $. Then,%
\begin{equation}
\sigma\operatorname*{cyc}\nolimits_{i_{1},i_{2},\ldots,i_{k}}\sigma
^{-1}=\operatorname*{cyc}\nolimits_{\sigma\left(  i_{1}\right)  ,\sigma\left(
i_{2}\right)  ,\ldots,\sigma\left(  i_{k}\right)  }.
\label{eq.prop.cyc-conj.eq}%
\end{equation}

\end{proposition}

\begin{vershort}

\begin{proof}
Well-known.
\end{proof}
\end{vershort}

\begin{verlong}

\begin{proof}
A well-known fact (see, e.g., \cite[Exercise 5.17 \textbf{(a)}]{detnotes})
says that%
\begin{equation}
\sigma\circ\operatorname*{cyc}\nolimits_{i_{1},i_{2},\ldots,i_{k}}\circ
\sigma^{-1}=\operatorname*{cyc}\nolimits_{\sigma\left(  i_{1}\right)
,\sigma\left(  i_{2}\right)  ,\ldots,\sigma\left(  i_{k}\right)  }.
\label{pf.prop.cyc-conj.1}%
\end{equation}
However, the product $\pi\sigma$ of two permutations $\pi,\sigma\in S_{n}$ is
precisely their composition $\pi\circ\sigma$. Hence,%
\[
\sigma\operatorname*{cyc}\nolimits_{i_{1},i_{2},\ldots,i_{k}}\sigma
^{-1}=\sigma\circ\operatorname*{cyc}\nolimits_{i_{1},i_{2},\ldots,i_{k}}%
\circ\sigma^{-1}=\operatorname*{cyc}\nolimits_{\sigma\left(  i_{1}\right)
,\sigma\left(  i_{2}\right)  ,\ldots,\sigma\left(  i_{k}\right)  }%
\]
(by (\ref{pf.prop.cyc-conj.1})). This proves Proposition \ref{prop.cyc-conj}.
\end{proof}
\end{verlong}

Proposition \ref{prop.cyc-conj} allows us to prove several relations between
the cycles $\left(  v\Longrightarrow w\right)  $. We shall collect a catalogue
of such relations now in order to have them at arm's reach in later proofs.

\begin{lemma}
\label{lem.jpiq1}Let $i,j,v,w\in\left[  n\right]  $ be such that $w\geq
v>j\geq i$. Then,%
\[
\left(  j+1\Longrightarrow v\right)  \left(  i\Longrightarrow w\right)
=\left(  i\Longrightarrow w\right)  \left(  j\Longrightarrow v-1\right)  .
\]

\end{lemma}

\begin{proof}
Let $\sigma:=\left(  i\Longrightarrow w\right)  $. We have $i\leq j$ (since
$j\geq i$) and $v-1\leq w-1$ (since $w\geq v$). Thus, the numbers
$j,j+1,\ldots,v-1$ all belong to the interval $\left[  i,w-1\right]  $. Hence,
the permutation $\sigma=\left(  i\Longrightarrow w\right)
=\operatorname*{cyc}\nolimits_{i,i+1,\ldots,w}$ sends these numbers to
$j+1,j+2,\ldots,v$, respectively. In other words,%
\[
\left(  \sigma\left(  j\right)  ,\sigma\left(  j+1\right)  ,\ldots
,\sigma\left(  v-1\right)  \right)  =\left(  j+1,j+2,\ldots,v\right)  .
\]
However, from $\left(  i\Longrightarrow w\right)  =\sigma$ and $\left(
j\Longrightarrow v-1\right)  =\operatorname*{cyc}\nolimits_{j,j+1,\ldots,v-1}%
$, we obtain%
\begin{align*}
&  \left(  i\Longrightarrow w\right)  \left(  j\Longrightarrow v-1\right)
\left(  i\Longrightarrow w\right)  ^{-1}\\
&  =\sigma\operatorname*{cyc}\nolimits_{j,j+1,\ldots,v-1}\sigma^{-1}%
=\operatorname*{cyc}\nolimits_{\sigma\left(  j\right)  ,\sigma\left(
j+1\right)  ,\ldots,\sigma\left(  v-1\right)  }\ \ \ \ \ \ \ \ \ \ \left(
\text{by (\ref{eq.prop.cyc-conj.eq})}\right) \\
&  =\operatorname*{cyc}\nolimits_{j+1,j+2,\ldots,v}\ \ \ \ \ \ \ \ \ \ \left(
\text{since }\left(  \sigma\left(  j\right)  ,\sigma\left(  j+1\right)
,\ldots,\sigma\left(  v-1\right)  \right)  =\left(  j+1,j+2,\ldots,v\right)
\right) \\
&  =\left(  j+1\Longrightarrow v\right)  .
\end{align*}
In other words, $\left(  i\Longrightarrow w\right)  \left(  j\Longrightarrow
v-1\right)  =\left(  j+1\Longrightarrow v\right)  \left(  i\Longrightarrow
w\right)  $. Thus, Lemma \ref{lem.jpiq1} is proved.
\end{proof}

\begin{lemma}
\label{lem.jpiq2}Let $i,v,w\in\left[  n\right]  $ be such that $v>w\geq i$.
Then,%
\[
\left(  i+1\Longrightarrow v\right)  \left(  i\Longrightarrow w\right)
=\left(  i\Longrightarrow w+1\right)  \left(  i\Longrightarrow v\right)  .
\]

\end{lemma}

\begin{proof}
We have $i<v$ (since $v>i$). Thus, Proposition \ref{prop.cycpq.rec}
\textbf{(a)} yields
\begin{equation}
\left(  i\Longrightarrow v\right)  =s_{i}\left(  i+1\Longrightarrow v\right)
. \label{pf.lem.jpiq2.1}%
\end{equation}

On the other hand, from $v>w$, we obtain $v\geq w+1$, so that $w+1\leq v\leq
n$ and therefore $w+1\in\left[  n\right]  $. Furthermore, $v\geq w+1>w\geq
i\geq i$. Thus, Lemma \ref{lem.jpiq1} (applied to $i$, $w+1$ and $v$ instead
of $j$, $v$ and $w$) yields%
\begin{align}
\left(  i+1\Longrightarrow w+1\right)  \left(  i\Longrightarrow v\right)   &
=\left(  i\Longrightarrow v\right)  \left(  i\Longrightarrow
\underbrace{\left(  w+1\right)  -1}_{=w}\right) \nonumber\\
&  =\left(  i\Longrightarrow v\right)  \left(  i\Longrightarrow w\right)  .
\label{pf.lem.jpiq2.3}%
\end{align}
However, Proposition \ref{prop.cycpq.rec} \textbf{(a)} yields $\left(
i\Longrightarrow w+1\right)  =s_{i}\left(  i+1\Longrightarrow w+1\right)  $
(since $i\leq w<w+1$). Hence,%
\begin{align*}
\underbrace{\left(  i\Longrightarrow w+1\right)  }_{=s_{i}\left(
i+1\Longrightarrow w+1\right)  }\left(  i\Longrightarrow v\right)   &
=s_{i}\underbrace{\left(  i+1\Longrightarrow w+1\right)  \left(
i\Longrightarrow v\right)  }_{\substack{=\left(  i\Longrightarrow v\right)
\left(  i\Longrightarrow w\right)  \\\text{(by (\ref{pf.lem.jpiq2.3}))}%
}}=s_{i}\underbrace{\left(  i\Longrightarrow v\right)  }_{\substack{=s_{i}%
\left(  i+1\Longrightarrow v\right)  \\\text{(by (\ref{pf.lem.jpiq2.1}))}%
}}\left(  i\Longrightarrow w\right) \\
&  =\underbrace{s_{i}s_{i}}_{\substack{=s_{i}^{2}=\operatorname*{id}%
\\\text{(by (\ref{eq.si.invol}))}}}\left(  i+1\Longrightarrow v\right)
\left(  i\Longrightarrow w\right)  =\left(  i+1\Longrightarrow v\right)
\left(  i\Longrightarrow w\right)  .
\end{align*}
This proves Lemma \ref{lem.jpiq2}.
\end{proof}

\begin{lemma}
\label{lem.suip}Let $i,u,v\in\left[  n\right]  $ be such that $i<u<v$. Then,%
\[
s_{u}\left(  i\Longrightarrow v\right)  =\left(  i\Longrightarrow v\right)
s_{u-1}.
\]

\end{lemma}

\begin{proof}
Let $\sigma:=\left(  i\Longrightarrow v\right)  $. Then, $i\leq u-1$ (since
$i<u$) and $u\leq v-1$ (since $u<v$). Therefore, the numbers $u-1$ and $u$
both belong to the interval $\left[  i,v-1\right]  $. Hence, the permutation
$\sigma=\left(  i\Longrightarrow v\right)  =\operatorname*{cyc}%
\nolimits_{i,i+1,\ldots,w}$ sends these numbers to $u$ and $u+1$,
respectively. In other words,%
\[
\sigma\left(  u-1\right)  =u\ \ \ \ \ \ \ \ \ \ \text{and}%
\ \ \ \ \ \ \ \ \ \ \sigma\left(  u\right)  =u+1.
\]
However, from $\left(  i\Longrightarrow v\right)  =\sigma$ and $s_{u-1}%
=\operatorname*{cyc}\nolimits_{u-1,u}$, we obtain%
\begin{align*}
\left(  i\Longrightarrow v\right)  s_{u-1}\left(  i\Longrightarrow v\right)
^{-1}  &  =\sigma\operatorname*{cyc}\nolimits_{u-1,u}\sigma^{-1}%
=\operatorname*{cyc}\nolimits_{\sigma\left(  u-1\right)  ,\sigma\left(
u\right)  }\ \ \ \ \ \ \ \ \ \ \left(  \text{by (\ref{eq.prop.cyc-conj.eq}%
)}\right) \\
&  =\operatorname*{cyc}\nolimits_{u,u+1}\ \ \ \ \ \ \ \ \ \ \left(
\text{since }\sigma\left(  u-1\right)  =u\text{ and }\sigma\left(  u\right)
=u+1\right) \\
&  =s_{u}.
\end{align*}
In other words, $\left(  i\Longrightarrow v\right)  s_{u-1}=s_{u}\left(
i\Longrightarrow v\right)  $. Thus, Lemma \ref{lem.suip} is proved.
\end{proof}

Finally, the following fact is easy to check:

\begin{lemma}
\label{lem.cyc-ijk}Let $i,j,k\in\left[  n\right]  $ be such that $i\leq j\leq
k$. Then,%
\[
\left(  i\Longrightarrow k\right)  =\left(  i\Longrightarrow j\right)  \left(
j\Longrightarrow k\right)  .
\]

\end{lemma}

\begin{proof}
Proposition \ref{prop.cycpq.sss} yields $\left(  i\Longrightarrow j\right)
=s_{i}s_{i+1}\cdots s_{j-1}$ and $\left(  j\Longrightarrow k\right)
=s_{j}s_{j+1}\cdots s_{k-1}$. Multiplying these equalities by each other, we
find%
\[
\left(  i\Longrightarrow j\right)  \left(  j\Longrightarrow k\right)  =\left(
s_{i}s_{i+1}\cdots s_{j-1}\right)  \left(  s_{j}s_{j+1}\cdots s_{k-1}\right)
=s_{i}s_{i+1}\cdots s_{k-1}.
\]
Comparing this with%
\[
\left(  i\Longrightarrow k\right)  =s_{i}s_{i+1}\cdots s_{k-1}%
\ \ \ \ \ \ \ \ \ \ \left(  \text{by Proposition \ref{prop.cycpq.sss}}\right)
,
\]
we obtain $\left(  i\Longrightarrow k\right)  =\left(  i\Longrightarrow
j\right)  \left(  j\Longrightarrow k\right)  $. This proves Lemma
\ref{lem.cyc-ijk}.
\end{proof}

\section{Basic properties of somewhere-to-below shuffles}

We now return to the group algebra $\mathbf{k}\left[  S_{n}\right]  $. We
begin by rewriting the definition of the somewhere-to-below shuffle $t_{\ell}$:

\begin{proposition}
\label{prop.tl-as-sum}Let $\ell\in\left[  n\right]  $. Then,%
\[
t_{\ell}=\sum_{w=\ell}^{n}\left(  \ell\Longrightarrow w\right)  .
\]

\end{proposition}

\begin{proof}
From (\ref{eq.def.tl.deftl}), we have%
\begin{align*}
t_{\ell}  &  =\operatorname*{cyc}\nolimits_{\ell}+\operatorname*{cyc}%
\nolimits_{\ell,\ell+1}+\operatorname*{cyc}\nolimits_{\ell,\ell+1,\ell
+2}+\cdots+\operatorname*{cyc}\nolimits_{\ell,\ell+1,\ldots,n}\\
&  =\sum_{w=\ell}^{n}\underbrace{\operatorname*{cyc}\nolimits_{\ell
,\ell+1,\ldots,w}}_{\substack{=\left(  \ell\Longrightarrow w\right)
\\\text{(by the definition of }\left(  \ell\Longrightarrow w\right)  \text{)}%
}}=\sum_{w=\ell}^{n}\left(  \ell\Longrightarrow w\right)  .
\end{align*}
This proves Proposition \ref{prop.tl-as-sum}.
\end{proof}

\begin{corollary}
\label{cor.tl-via-tl+1}Let $\ell\in\left[  n-1\right]  $. Then, $t_{\ell
}=1+s_{\ell}t_{\ell+1}$.
\end{corollary}

\begin{proof}
Proposition \ref{prop.tl-as-sum} yields%
\begin{align*}
t_{\ell}  &  =\sum_{w=\ell}^{n}\left(  \ell\Longrightarrow w\right)
=\underbrace{\left(  \ell\Longrightarrow\ell\right)  }_{=1}+\sum_{w=\ell
+1}^{n}\underbrace{\left(  \ell\Longrightarrow w\right)  }_{\substack{=s_{\ell
}\left(  \ell+1\Longrightarrow w\right)  \\\text{(by Proposition
\ref{prop.cycpq.rec} \textbf{(a)})}}}\\
&  =1+\sum_{w=\ell+1}^{n}s_{\ell}\left(  \ell+1\Longrightarrow w\right)
=1+s_{\ell}\sum_{w=\ell+1}^{n}\left(  \ell+1\Longrightarrow w\right)  .
\end{align*}
Comparing this with%
\[
1+s_{\ell}\underbrace{t_{\ell+1}}_{\substack{=\sum_{w=\ell+1}^{n}\left(
\ell+1\Longrightarrow w\right)  \\\text{(by Proposition \ref{prop.tl-as-sum}%
)}}}=1+s_{\ell}\sum_{w=\ell+1}^{n}\left(  \ell+1\Longrightarrow w\right)  ,
\]
we obtain $t_{\ell}=1+s_{\ell}t_{\ell+1}$, qed.
\end{proof}

We state another simple property of the $t_{\ell}$'s:

\begin{lemma}
\label{lem.commute-with-tj}Let $\ell\in\left[  n\right]  $. Let $\sigma\in
S_{n}$. Assume that $\sigma$ leaves all the elements $\ell,\ell+1,\ldots,n$
unchanged. Then, $\sigma$ commutes with $t_{\ell}$ in $\mathbf{k}\left[
S_{n}\right]  $.
\end{lemma}

\begin{proof}
The permutation $\sigma$ leaves all the elements $\ell,\ell+1,\ldots,n$
unchanged, and thus commutes with each cycle $\operatorname*{cyc}%
\nolimits_{\ell,\ell+1,\ldots,w}$ with $w\geq\ell$ (because the latter cycle
permutes only elements of $\left\{  \ell,\ell+1,\ldots,n\right\}  $). Hence,
the permutation $\sigma$ also commutes with the sum $\sum_{w=\ell}%
^{n}\operatorname*{cyc}\nolimits_{\ell,\ell+1,\ldots,w}$ of these cycles.
Since the definition of $t_{\ell}$ yields%
\[
t_{\ell}=\operatorname*{cyc}\nolimits_{\ell}+\operatorname*{cyc}%
\nolimits_{\ell,\ell+1}+\operatorname*{cyc}\nolimits_{\ell,\ell+1,\ell
+2}+\cdots+\operatorname*{cyc}\nolimits_{\ell,\ell+1,\ldots,n}=\sum_{w=\ell
}^{n}\operatorname*{cyc}\nolimits_{\ell,\ell+1,\ldots,w},
\]
we can rewrite this as follows: The permutation $\sigma$ commutes with
$t_{\ell}$. This proves Lemma \ref{lem.commute-with-tj}.
\end{proof}

Specifically, we will need only the following particular case of Lemma
\ref{lem.commute-with-tj}:

\begin{lemma}
\label{lem.commute-with-tj-specific}Let $i,k,j\in\left[  n\right]  $ be such
that $i\leq k<j$. Then,%
\begin{equation}
\left(  i\Longrightarrow k\right)  t_{j}=t_{j}\left(  i\Longrightarrow
k\right)  \label{eq.lem.commute-with-tj-specific.ab=ba}%
\end{equation}
and%
\begin{equation}
\left[  \left(  i\Longrightarrow k\right)  ,\ t_{j}\right]  =0.
\label{eq.lem.commute-with-tj-specific.comm}%
\end{equation}

\end{lemma}

\begin{proof}
The permutation $\left(  i\Longrightarrow k\right)  =\operatorname*{cyc}%
\nolimits_{i,i+1,\ldots,k}$ leaves all the elements $k+1,k+2,\ldots,n$
unchanged, and thus leaves all the elements $j,j+1,\ldots,n$ unchanged (since
the latter elements are a subset of the former elements (because $k<j$)).
Hence, Lemma \ref{lem.commute-with-tj} (applied to $\ell=j$ and $\sigma
=\left(  i\Longrightarrow k\right)  $) shows that $\left(  i\Longrightarrow
k\right)  $ commutes with $t_{j}$ in $\mathbf{k}\left[  S_{n}\right]  $. In
other words, $\left(  i\Longrightarrow k\right)  t_{j}=t_{j}\left(
i\Longrightarrow k\right)  $. This proves
(\ref{eq.lem.commute-with-tj-specific.ab=ba}).

Now, the definition of a commutator yields
\[
\left[  \left(  i\Longrightarrow k\right)  ,\ t_{j}\right]  =\left(
i\Longrightarrow k\right)  t_{j}-t_{j}\left(  i\Longrightarrow k\right)  =0
\]
(since $\left(  i\Longrightarrow k\right)  t_{j}=t_{j}\left(  i\Longrightarrow
k\right)  $). This proves (\ref{eq.lem.commute-with-tj-specific.comm}). Thus,
Lemma \ref{lem.commute-with-tj-specific} is completely proved.
\end{proof}

\section{The identities $t_{i+1}t_{i}=\left(  t_{i}-1\right)  t_{i}%
=t_{i}\left(  t_{i}-1\right)  $ and $\left[  t_{i},t_{i+1}\right]  ^{2}=0$}

\subsection{The identity $t_{i+1}t_{i}=\left(  t_{i}-1\right)  t_{i}%
=t_{i}\left(  t_{i}-1\right)  $}

We are now ready to prove the first really surprising result:

\begin{theorem}
\label{thm.ti+1ti}Let $i\in\left[  n-1\right]  $. Then,%
\begin{align}
t_{i+1}t_{i}  &  =\left(  t_{i}-1\right)  t_{i}\label{eq.thm.ti+1ti.ti-1ti}\\
&  =t_{i}\left(  t_{i}-1\right)  . \label{eq.thm.ti+1ti.titi-1}%
\end{align}

\end{theorem}

\begin{proof}
From Proposition \ref{prop.tl-as-sum}, we obtain%
\begin{align}
t_{i}  &  =\sum_{w=i}^{n}\left(  i\Longrightarrow w\right)
\label{pf.thm.ti+1ti.ti=}\\
&  =\underbrace{\left(  i\Longrightarrow i\right)  }_{=\operatorname*{id}%
=1}+\sum_{w=i+1}^{n}\left(  i\Longrightarrow w\right)
\ \ \ \ \ \ \ \ \ \ \left(
\begin{array}
[c]{c}%
\text{here, we have split off the}\\
\text{addend for }w=i\text{ from the sum}%
\end{array}
\right) \nonumber\\
&  =1+\sum_{w=i+1}^{n}\left(  i\Longrightarrow w\right)  .\nonumber
\end{align}
In other words,%
\begin{equation}
t_{i}-1=\sum_{w=i+1}^{n}\left(  i\Longrightarrow w\right)  .
\label{pf.thm.ti+1ti.ti-1=}%
\end{equation}
Moreover, (\ref{pf.thm.ti+1ti.ti=}) becomes%
\begin{equation}
t_{i}=\sum_{w=i}^{n}\left(  i\Longrightarrow w\right)  =\sum_{v=i}^{n}\left(
i\Longrightarrow v\right)  . \label{pf.thm.ti+1ti.ti=p}%
\end{equation}

Also, Proposition \ref{prop.tl-as-sum} (applied to $\ell=i+1$) yields
\begin{equation}
t_{i+1}=\sum_{w=i+1}^{n}\left(  i+1\Longrightarrow w\right)  =\sum_{v=i+1}%
^{n}\left(  i+1\Longrightarrow v\right)  .\nonumber
\end{equation}
Multiplying this equality by (\ref{pf.thm.ti+1ti.ti=}), we obtain%
\begin{align*}
t_{i+1}t_{i}  &  =\sum_{v=i+1}^{n}\left(  i+1\Longrightarrow v\right)
\cdot\sum_{w=i}^{n}\left(  i\Longrightarrow w\right) \\
&  =\sum_{v=i+1}^{n}\ \ \underbrace{\sum_{w=i}^{n}\left(  i+1\Longrightarrow
v\right)  \left(  i\Longrightarrow w\right)  }_{\substack{=\sum_{w=i}%
^{v-1}\left(  i+1\Longrightarrow v\right)  \left(  i\Longrightarrow w\right)
+\sum_{w=v}^{n}\left(  i+1\Longrightarrow v\right)  \left(  i\Longrightarrow
w\right)  \\\text{(since }i<v\leq n\text{)}}}\\
&  =\sum_{v=i+1}^{n}\left(  \sum_{w=i}^{v-1}\left(  i+1\Longrightarrow
v\right)  \left(  i\Longrightarrow w\right)  +\sum_{w=v}^{n}\left(
i+1\Longrightarrow v\right)  \left(  i\Longrightarrow w\right)  \right) \\
&  =\sum_{v=i+1}^{n}\ \ \sum_{w=i}^{v-1}\underbrace{\left(  i+1\Longrightarrow
v\right)  \left(  i\Longrightarrow w\right)  }_{\substack{=\left(
i\Longrightarrow w+1\right)  \left(  i\Longrightarrow v\right)  \\\text{(by
Lemma \ref{lem.jpiq2}}\\\text{(since }w\leq v-1<v\text{ and thus }v>w\geq
i\text{))}}}+\sum_{v=i+1}^{n}\ \ \sum_{w=v}^{n}\underbrace{\left(
i+1\Longrightarrow v\right)  \left(  i\Longrightarrow w\right)  }%
_{\substack{=\left(  i\Longrightarrow w\right)  \left(  i\Longrightarrow
v-1\right)  \\\text{(by Lemma \ref{lem.jpiq1}, applied to }j=i\\\text{(since
}w\geq v\geq i+1>i\geq i\text{))}}}\\
&  =\sum_{v=i+1}^{n}\ \ \underbrace{\sum_{w=i}^{v-1}\left(  i\Longrightarrow
w+1\right)  \left(  i\Longrightarrow v\right)  }_{\substack{=\sum_{w=i+1}%
^{v}\left(  i\Longrightarrow w\right)  \left(  i\Longrightarrow v\right)
\\\text{(here, we have substituted }w\text{ for }w+1\text{)}}%
}+\underbrace{\sum_{v=i+1}^{n}\ \ \sum_{w=v}^{n}\left(  i\Longrightarrow
w\right)  \left(  i\Longrightarrow v-1\right)  }_{\substack{=\sum_{v=i}%
^{n-1}\ \ \sum_{w=v+1}^{n}\left(  i\Longrightarrow w\right)  \left(
i\Longrightarrow v\right)  \\\text{(here, we have substituted }v\text{ for
}v-1\text{)}}}\\
&  =\sum_{v=i+1}^{n}\ \ \sum_{w=i+1}^{v}\left(  i\Longrightarrow w\right)
\left(  i\Longrightarrow v\right)  +\sum_{v=i}^{n-1}\ \ \sum_{w=v+1}%
^{n}\left(  i\Longrightarrow w\right)  \left(  i\Longrightarrow v\right)  .
\end{align*}

Comparing this with%
\begin{align*}
\left(  t_{i}-1\right)  t_{i}  &  =\sum_{w=i+1}^{n}\left(  i\Longrightarrow
w\right)  \cdot\sum_{v=i}^{n}\left(  i\Longrightarrow v\right) \\
&  \ \ \ \ \ \ \ \ \ \ \ \ \ \ \ \ \ \ \ \ \left(  \text{by multiplying the
two equalities (\ref{pf.thm.ti+1ti.ti-1=}) and (\ref{pf.thm.ti+1ti.ti=p}%
)}\right) \\
&  =\sum_{v=i}^{n}\ \ \underbrace{\sum_{w=i+1}^{n}\left(  i\Longrightarrow
w\right)  \left(  i\Longrightarrow v\right)  }_{\substack{=\sum_{w=i+1}%
^{v}\left(  i\Longrightarrow w\right)  \left(  i\Longrightarrow v\right)
+\sum_{w=v+1}^{n}\left(  i\Longrightarrow w\right)  \left(  i\Longrightarrow
v\right)  \\\text{(since }i\leq v\leq n\text{)}}}\\
&  =\sum_{v=i}^{n}\left(  \sum_{w=i+1}^{v}\left(  i\Longrightarrow w\right)
\left(  i\Longrightarrow v\right)  +\sum_{w=v+1}^{n}\left(  i\Longrightarrow
w\right)  \left(  i\Longrightarrow v\right)  \right) \\
&  =\underbrace{\sum_{v=i}^{n}\ \ \sum_{w=i+1}^{v}\left(  i\Longrightarrow
w\right)  \left(  i\Longrightarrow v\right)  }_{\substack{=\sum_{w=i+1}%
^{i}\left(  i\Longrightarrow w\right)  \left(  i\Longrightarrow i\right)
+\sum_{v=i+1}^{n}\ \ \sum_{w=i+1}^{v}\left(  i\Longrightarrow w\right)
\left(  i\Longrightarrow v\right)  \\\text{(here, we have split off the addend
for }v=i\text{ from the sum)}}}\\
&  \ \ \ \ \ \ \ \ \ \ +\underbrace{\sum_{v=i}^{n}\ \ \sum_{w=v+1}^{n}\left(
i\Longrightarrow w\right)  \left(  i\Longrightarrow v\right)  }%
_{\substack{=\sum_{w=n+1}^{n}\left(  i\Longrightarrow w\right)  \left(
i\Longrightarrow n\right)  +\sum_{v=i}^{n-1}\ \ \sum_{w=v+1}^{n}\left(
i\Longrightarrow w\right)  \left(  i\Longrightarrow v\right)  \\\text{(here,
we have split off the addend for }v=n\text{ from the sum)}}}\\
&  =\underbrace{\sum_{w=i+1}^{i}\left(  i\Longrightarrow w\right)  \left(
i\Longrightarrow i\right)  }_{=\left(  \text{empty sum}\right)  =0}%
+\sum_{v=i+1}^{n}\ \ \sum_{w=i+1}^{v}\left(  i\Longrightarrow w\right)
\left(  i\Longrightarrow v\right) \\
&  \ \ \ \ \ \ \ \ \ \ +\underbrace{\sum_{w=n+1}^{n}\left(  i\Longrightarrow
w\right)  \left(  i\Longrightarrow n\right)  }_{=\left(  \text{empty
sum}\right)  =0}+\sum_{v=i}^{n-1}\ \ \sum_{w=v+1}^{n}\left(  i\Longrightarrow
w\right)  \left(  i\Longrightarrow v\right) \\
&  =\sum_{v=i+1}^{n}\ \ \sum_{w=i+1}^{v}\left(  i\Longrightarrow w\right)
\left(  i\Longrightarrow v\right)  +\sum_{v=i}^{n-1}\ \ \sum_{w=v+1}%
^{n}\left(  i\Longrightarrow w\right)  \left(  i\Longrightarrow v\right)  ,
\end{align*}
we obtain $t_{i+1}t_{i}=\left(  t_{i}-1\right)  t_{i}$. This proves
(\ref{eq.thm.ti+1ti.ti-1ti}). From this, (\ref{eq.thm.ti+1ti.titi-1}) follows,
since $\left(  t_{i}-1\right)  t_{i}=t_{i}^{2}-t_{i}=t_{i}\left(
t_{i}-1\right)  $. Thus, Theorem \ref{thm.ti+1ti} is proved.
\end{proof}

\subsection{The identity $\left[  t_{i},t_{i+1}\right]  ^{2}=0$}

\begin{corollary}
\label{cor.titi+12=0}Let $i\in\left[  n-1\right]  $. Then,%
\begin{equation}
\left[  t_{i},t_{i+1}\right]  =t_{i}\left(  t_{i+1}-\left(  t_{i}-1\right)
\right)  \label{eq.cor.titi+12=0.comm}%
\end{equation}
and%
\begin{equation}
\left[  t_{i},t_{i+1}\right]  t_{i}=0 \label{eq.cor.titi+12=0.prod}%
\end{equation}
and%
\begin{equation}
\left[  t_{i},t_{i+1}\right]  ^{2}=0. \label{eq.cor.titi+12=0.sq}%
\end{equation}

\end{corollary}

\begin{proof}
The definition of a commutator yields%
\[
\left[  t_{i},t_{i+1}\right]  =t_{i}t_{i+1}-\underbrace{t_{i+1}t_{i}%
}_{\substack{=t_{i}\left(  t_{i}-1\right)  \\\text{(by
(\ref{eq.thm.ti+1ti.titi-1}))}}}=t_{i}t_{i+1}-t_{i}\left(  t_{i}-1\right)
=t_{i}\left(  t_{i+1}-\left(  t_{i}-1\right)  \right)  .
\]
This proves the equality (\ref{eq.cor.titi+12=0.comm}). Multiplying both sides
of this equality by $t_{i}$ on the right, we obtain%
\[
\left[  t_{i},t_{i+1}\right]  t_{i}=t_{i}\underbrace{\left(  t_{i+1}-\left(
t_{i}-1\right)  \right)  t_{i}}_{=t_{i+1}t_{i}-\left(  t_{i}-1\right)  t_{i}%
}=t_{i}\underbrace{\left(  t_{i+1}t_{i}-\left(  t_{i}-1\right)  t_{i}\right)
}_{\substack{=0\\\text{(by (\ref{eq.thm.ti+1ti.ti-1ti}))}}}=0.
\]
This proves (\ref{eq.cor.titi+12=0.prod}). Now,%
\[
\left[  t_{i},t_{i+1}\right]  ^{2}=\left[  t_{i},t_{i+1}\right]
\underbrace{\left[  t_{i},t_{i+1}\right]  }_{\substack{=t_{i}\left(
t_{i+1}-\left(  t_{i}-1\right)  \right)  \\\text{(by
(\ref{eq.cor.titi+12=0.comm}))}}}=\underbrace{\left[  t_{i},t_{i+1}\right]
t_{i}}_{\substack{=0\\\text{(by (\ref{eq.cor.titi+12=0.prod}))}}}\left(
t_{i+1}-\left(  t_{i}-1\right)  \right)  =0.
\]
This proves (\ref{eq.cor.titi+12=0.sq}). Thus, Corollary \ref{cor.titi+12=0}
is proved.
\end{proof}

\section{\label{sec.t3t1}The identities $t_{i+2}\left(  t_{i}-1\right)
=\left(  t_{i}-1\right)  \left(  t_{i+1}-1\right)  $ and $\left[
t_{i},t_{i+2}\right]  \left(  t_{i}-1\right)  =t_{i+1}\left[  t_{i}%
,t_{i+1}\right]  $}

\subsection{The identity $t_{i+2}\left(  t_{i}-1\right)  =\left(
t_{i}-1\right)  \left(  t_{i+1}-1\right)  $}

The next theorem is a \textquotedblleft next-level\textquotedblright\ analogue
of Theorem \ref{thm.ti+1ti}:

\begin{theorem}
\label{thm.ti+2ti}Let $i\in\left[  n-2\right]  $. Then,%
\[
t_{i+2}\left(  t_{i}-1\right)  =\left(  t_{i}-1\right)  \left(  t_{i+1}%
-1\right)  .
\]

\end{theorem}

\begin{proof}
[Proof of Theorem \ref{thm.ti+2ti}.]From $i\in\left[  n-2\right]  $, we obtain
$i+1\in\left[  2,n-1\right]  \subseteq\left[  n-1\right]  $. Hence,
(\ref{eq.thm.ti+1ti.titi-1}) (applied to $i+1$ instead of $i$) yields
$t_{\left(  i+1\right)  +1}t_{i+1}=t_{i+1}\left(  t_{i+1}-1\right)  $. In view
of $\left(  i+1\right)  +1=i+2$, we can rewrite this as%
\begin{equation}
t_{i+2}t_{i+1}=t_{i+1}\left(  t_{i+1}-1\right)  . \label{pf.thm.ti+2ti.1}%
\end{equation}

Furthermore, Corollary \ref{cor.tl-via-tl+1} (applied to $\ell=i$) yields
$t_{i}=1+s_{i}t_{i+1}$ (since $i\in\left[  n-2\right]  \subseteq\left[
n-1\right]  $). Hence,
\begin{equation}
t_{i}-1=s_{i}t_{i+1}. \label{pf.thm.ti+2ti.2}%
\end{equation}

The definition of $\left(  i\Longrightarrow i+1\right)  $ yields $\left(
i\Longrightarrow i+1\right)  =\operatorname*{cyc}\nolimits_{i,i+1,\ldots
,i+1}=\operatorname*{cyc}\nolimits_{i,i+1}=s_{i}$. However,
(\ref{eq.lem.commute-with-tj-specific.ab=ba}) (applied to $k=i+1$ and $j=i+2$)
yields $\left(  i\Longrightarrow i+1\right)  t_{i+2}=t_{i+2}\left(
i\Longrightarrow i+1\right)  $. In view of $\left(  i\Longrightarrow
i+1\right)  =s_{i}$, we can rewrite this as $s_{i}t_{i+2}=t_{i+2}s_{i}$. In
other words, $t_{i+2}s_{i}=s_{i}t_{i+2}$.

Now,
\[
t_{i+2}\underbrace{\left(  t_{i}-1\right)  }_{\substack{=s_{i}t_{i+1}%
\\\text{(by (\ref{pf.thm.ti+2ti.2}))}}}=\underbrace{t_{i+2}s_{i}}%
_{=s_{i}t_{i+2}}t_{i+1}=s_{i}\underbrace{t_{i+2}t_{i+1}}_{\substack{=t_{i+1}%
\left(  t_{i+1}-1\right)  \\\text{(by (\ref{pf.thm.ti+2ti.1}))}}%
}=\underbrace{s_{i}t_{i+1}}_{\substack{=t_{i}-1\\\text{(by
(\ref{pf.thm.ti+2ti.2}))}}}\left(  t_{i+1}-1\right)  =\left(  t_{i}-1\right)
\left(  t_{i+1}-1\right)  .
\]
Thus, Theorem \ref{thm.ti+2ti} is proved.
\end{proof}

\begin{remark}
\label{rmk.ti+kti}The similarity between Theorem \ref{thm.ti+1ti} and Theorem
\ref{thm.ti+2ti} might suggest that the two theorems are the first two
instances of a general identity of the form $t_{i+k}\left(  t_{i}-\ell\right)
=\left(  t_{i}-m\right)  \left(  \text{something}\right)  $ for certain
integers $\ell$ and $m$. Unfortunately, such an identity most likely does not
exist for $k=3$. Indeed, using SageMath, we have verified that for $n=6$ and
$\mathbf{k}=\mathbb{Q}$, the product $t_{4}\left(  t_{1}-\ell\right)  $ is not
a right multiple of $t_{1}-m$ for any $m\in\left\{  0,1,2,3,4,6\right\}  $ and
$\ell\in\left[  0,12\right]  $. (The restriction $m\in\left\{
0,1,2,3,4,6\right\}  $ ensures that $t_{1}-m$ is not invertible; otherwise,
the claim would be trivial and uninteresting.)

We also observe that $t_{4}t_{1}$ does not belong to the $\mathbb{Q}$-linear
span of the elements $1$, $t_{i}$ and $t_{i}t_{j}$ for $i\leq j$ when $n=5$.
This is another piece of evidence suggesting that the pattern of Theorems
\ref{thm.ti+1ti} and \ref{thm.ti+2ti} does not continue.

Generally, for $\mathbf{k}=\mathbb{Q}$, the span of all products of the form
$t_{i}t_{j}$ with $i,j\in\left[  n\right]  $ inside $\mathbb{Q}\left[
S_{n}\right]  $ seems to have dimension $n^{2}-3n+4$ (verified using SageMath
for all $n\leq14$). The discrepancy between this dimension and the naive
maximum guess $n^{2}$ is fully explained by the $n-1$ identities $t_{i+1}%
t_{i}=t_{i}^{2}-t_{i}$ from Theorem \ref{thm.ti+1ti}, the $n-2$ identities
$t_{i+2}t_{i}-t_{i+2}=t_{i}t_{i+1}-t_{i}-t_{i+1}+1$ from Theorem
\ref{thm.ti+2ti}, and the $n-1$ obvious identities $t_{i}t_{n}=t_{n}t_{i}$
that come from $t_{n}=1$ (assuming that these $3n-4$ identities are linearly
independent). This suggests that Theorem \ref{thm.ti+1ti} and Theorem
\ref{thm.ti+2ti} exhaust the interesting quadratic identities between the
$t_{i}$.
\end{remark}

\subsection{The identity $\left[  t_{i},t_{i+2}\right]  \left(  t_{i}%
-1\right)  =t_{i+1}\left[  t_{i},t_{i+1}\right]  $}

\begin{corollary}
\label{cor.titi+2ti-1}Let $i\in\left[  n-2\right]  $. Then,%
\[
\left[  t_{i},t_{i+2}\right]  \left(  t_{i}-1\right)  =t_{i+1}\left[
t_{i},t_{i+1}\right]  .
\]

\end{corollary}

\begin{proof}
From $i\in\left[  n-2\right]  $, we obtain $i+1\in\left[  2,n-1\right]
\subseteq\left[  n-1\right]  $. Thus, (\ref{eq.thm.ti+1ti.titi-1}) (applied to
$i+1$ instead of $i$) yields $t_{\left(  i+1\right)  +1}t_{i+1}=t_{i+1}\left(
t_{i+1}-1\right)  $. In view of $\left(  i+1\right)  +1=i+2$, we can rewrite
this as%
\begin{equation}
t_{i+2}t_{i+1}=t_{i+1}\left(  t_{i+1}-1\right)  . \label{pf.cor.titi+2ti-1.1}%
\end{equation}

However, we have
\begin{align}
t_{i}\underbrace{t_{i+2}\left(  t_{i}-1\right)  }_{\substack{=\left(
t_{i}-1\right)  \left(  t_{i+1}-1\right)  \\\text{(by Theorem \ref{thm.ti+2ti}%
)}}}  &  =\underbrace{t_{i}\left(  t_{i}-1\right)  }_{\substack{=t_{i+1}%
t_{i}\\\text{(by (\ref{eq.thm.ti+1ti.titi-1}))}}}\left(  t_{i+1}-1\right)
=t_{i+1}t_{i}\left(  t_{i+1}-1\right) \nonumber\\
&  =t_{i+1}t_{i}t_{i+1}-t_{i+1}t_{i} \label{pf.cor.titi+2ti-1.2}%
\end{align}
and%
\begin{align}
t_{i+2}\underbrace{t_{i}\left(  t_{i}-1\right)  }_{\substack{=t_{i+1}%
t_{i}\\\text{(by (\ref{eq.thm.ti+1ti.titi-1}))}}}  &  =\underbrace{t_{i+2}%
t_{i+1}}_{\substack{=t_{i+1}\left(  t_{i+1}-1\right)  \\\text{(by
(\ref{pf.cor.titi+2ti-1.1}))}}}t_{i}=t_{i+1}\left(  t_{i+1}-1\right)
t_{i}\nonumber\\
&  =t_{i+1}t_{i+1}t_{i}-t_{i+1}t_{i}. \label{pf.cor.titi+2ti-1.3}%
\end{align}

Now, the definition of a commutator yields $\left[  t_{i},t_{i+1}\right]
=t_{i}t_{i+1}-t_{i+1}t_{i}$ and $\left[  t_{i},t_{i+2}\right]  =t_{i}%
t_{i+2}-t_{i+2}t_{i}$. Hence,%
\begin{align*}
\underbrace{\left[  t_{i},t_{i+2}\right]  }_{=t_{i}t_{i+2}-t_{i+2}t_{i}%
}\left(  t_{i}-1\right)   &  =\left(  t_{i}t_{i+2}-t_{i+2}t_{i}\right)
\left(  t_{i}-1\right) \\
&  =t_{i}t_{i+2}\left(  t_{i}-1\right)  -t_{i+2}t_{i}\left(  t_{i}-1\right) \\
&  =\left(  t_{i+1}t_{i}t_{i+1}-t_{i+1}t_{i}\right)  -\left(  t_{i+1}%
t_{i+1}t_{i}-t_{i+1}t_{i}\right) \\
&  \ \ \ \ \ \ \ \ \ \ \ \ \ \ \ \ \ \ \ \ \left(  \text{here, we subtracted
the equality (\ref{pf.cor.titi+2ti-1.3}) from (\ref{pf.cor.titi+2ti-1.2}%
)}\right) \\
&  =t_{i+1}t_{i}t_{i+1}-t_{i+1}t_{i+1}t_{i}=t_{i+1}\underbrace{\left(
t_{i}t_{i+1}-t_{i+1}t_{i}\right)  }_{=\left[  t_{i},t_{i+1}\right]  }%
=t_{i+1}\left[  t_{i},t_{i+1}\right]  .
\end{align*}
This proves Corollary \ref{cor.titi+2ti-1}.
\end{proof}

\section{The identity $\left(  1+s_{j}\right)  \left[  t_{i},t_{j}\right]  =0$
for all $i\leq j$}

\subsection{The identity $\left(  1+s_{j}\right)  \left[  t_{j-1}%
,t_{j}\right]  =0$}

We shall next prove the following:

\begin{lemma}
\label{lem.1+si+1titi+1}Let $i\in\left[  n-2\right]  $. Then,%
\[
\left(  1+s_{i+1}\right)  \left[  t_{i},t_{i+1}\right]  =0.
\]

\end{lemma}

\begin{proof}
Set $a:=t_{i}$ and $b:=t_{i+1}$.

From (\ref{eq.si.invol}), we obtain $s_{i+1}^{2}=\operatorname*{id}=1$. Hence,
$s_{i+1}^{2}-1=0$.

The definition of a commutator yields%
\begin{align}
\left[  a-1,\ b-1\right]   &  =\left(  a-1\right)  \left(  b-1\right)
-\left(  b-1\right)  \left(  a-1\right) \nonumber\\
&  =\left(  ab-a-b+1\right)  -\left(  ba-b-a+1\right) \nonumber\\
&  =ab-ba=\left[  a,b\right]  . \label{pf.lem.1+si+1titi+1.-1}%
\end{align}

From $i\in\left[  n-2\right]  $, we obtain $i+1\in\left[  2,n-1\right]
\subseteq\left[  n-1\right]  $. Hence, Corollary \ref{cor.tl-via-tl+1}
(applied to $\ell=i+1$) yields $t_{i+1}=1+s_{i+1}t_{\left(  i+1\right)
+1}=1+s_{i+1}t_{i+2}$ (since $\left(  i+1\right)  +1=i+2$). Hence,
$b=t_{i+1}=1+s_{i+1}t_{i+2}$. Therefore, $b-1=s_{i+1}t_{i+2}$. Thus,%
\begin{equation}
s_{i+1}\underbrace{\left(  b-1\right)  }_{=s_{i+1}t_{i+2}}=\underbrace{s_{i+1}%
s_{i+1}}_{\substack{=s_{i+1}^{2}=1}}t_{i+2}=t_{i+2}.
\label{pf.lem.1+si+1titi+1.=ti+2}%
\end{equation}
However, Theorem \ref{thm.ti+2ti} yields
\[
t_{i+2}\left(  t_{i}-1\right)  =\left(  t_{i}-1\right)  \left(  t_{i+1}%
-1\right)  .
\]
In view of $a=t_{i}$ and $b=t_{i+1}$, we can rewrite this as%
\[
t_{i+2}\left(  a-1\right)  =\left(  a-1\right)  \left(  b-1\right)  .
\]
Hence,
\begin{equation}
\left(  a-1\right)  \left(  b-1\right)  =\underbrace{t_{i+2}}%
_{\substack{=s_{i+1}\left(  b-1\right)  \\\text{(by
(\ref{pf.lem.1+si+1titi+1.=ti+2}))}}}\left(  a-1\right)  =s_{i+1}\left(
b-1\right)  \left(  a-1\right)  . \label{pf.lem.1+si+1titi+1.3}%
\end{equation}

Now, (\ref{pf.lem.1+si+1titi+1.-1}) becomes%
\begin{align*}
\left[  a,b\right]   &  =\left[  a-1,\ b-1\right]  =\underbrace{\left(
a-1\right)  \left(  b-1\right)  }_{\substack{=s_{i+1}\left(  b-1\right)
\left(  a-1\right)  \\\text{(by (\ref{pf.lem.1+si+1titi+1.3}))}}}-\left(
b-1\right)  \left(  a-1\right) \\
&  =s_{i+1}\left(  b-1\right)  \left(  a-1\right)  -\left(  b-1\right)
\left(  a-1\right) \\
&  =\left(  s_{i+1}-1\right)  \left(  b-1\right)  \left(  a-1\right)  .
\end{align*}
Multiplying both sides of this equality by $1+s_{i+1}$ from the left, we
obtain%
\[
\left(  1+s_{i+1}\right)  \left[  a,b\right]  =\underbrace{\left(
1+s_{i+1}\right)  \left(  s_{i+1}-1\right)  }_{\substack{=\left(
s_{i+1}+1\right)  \left(  s_{i+1}-1\right)  \\=s_{i+1}^{2}-1=0}}\left(
b-1\right)  \left(  a-1\right)  =0.
\]
In view of $a=t_{i}$ and $b=t_{i+1}$, we can rewrite this as $\left(
1+s_{i+1}\right)  \left[  t_{i},t_{i+1}\right]  =0$. This proves Lemma
\ref{lem.1+si+1titi+1}.
\end{proof}

The following is just a restatement of Lemma \ref{lem.1+si+1titi+1}:

\begin{lemma}
\label{lem.1+sjtj-1tj}Let $j\in\left[  2,n-1\right]  $. Then,%
\[
\left(  1+s_{j}\right)  \left[  t_{j-1},t_{j}\right]  =0.
\]

\end{lemma}

\begin{proof}
We have $j-1\in\left[  n-2\right]  $ (since $j\in\left[  2,n-1\right]  $).
Hence, Lemma \ref{lem.1+si+1titi+1} (applied to $i=j-1$) yields%
\[
\left(  1+s_{\left(  j-1\right)  +1}\right)  \left[  t_{j-1},t_{\left(
j-1\right)  +1}\right]  =0.
\]
In view of $\left(  j-1\right)  +1=j$, we can rewrite this as $\left(
1+s_{j}\right)  \left[  t_{j-1},t_{j}\right]  =0$. This proves Lemma
\ref{lem.1+sjtj-1tj}.
\end{proof}

\subsection{Expressing $\left[  t_{i},t_{j}\right]  $ via $\left[
t_{j-1},t_{j}\right]  $}

The following lemma is useful for reducing questions about $\left[
t_{i},t_{j}\right]  $ to questions about $\left[  t_{j-1},t_{j}\right]  $:

\begin{lemma}
\label{lem.ti-to-tj-1}Let $i,j\in\left[  n\right]  $ satisfy $i<j$. Then:

\begin{enumerate}
\item[\textbf{(a)}] We have%
\[
\left[  t_{i},t_{j}\right]  =\left[  s_{i}s_{i+1}\cdots s_{j-1},t_{j}\right]
t_{j}.
\]

\item[\textbf{(b)}] We have%
\[
\left[  t_{i},t_{j}\right]  =\left(  s_{i}s_{i+1}\cdots s_{j-2}\right)
\left[  t_{j-1},t_{j}\right]  .
\]

\end{enumerate}
\end{lemma}

\begin{proof}
A well-known identity for commutators says that if $R$ is a ring, then any
three elements $a,b,c\in R$ satisfy%
\begin{equation}
\left[  ab,c\right]  =\left[  a,c\right]  b+a\left[  b,c\right]  .
\label{pf.lem.ti-to-tj-1.b.0}%
\end{equation}
Hence, if $R$ is a ring, then any two elements $a,b\in R$ satisfy%
\begin{align}
\left[  ab,b\right]   &  =\left[  a,b\right]  b+a\underbrace{\left[
b,b\right]  }_{\substack{=bb-bb\\=0}}\ \ \ \ \ \ \ \ \ \ \left(  \text{by
(\ref{pf.lem.ti-to-tj-1.b.0}), applied to }c=b\right) \nonumber\\
&  =\left[  a,b\right]  b+a0=\left[  a,b\right]  b.
\label{pf.lem.ti-to-tj-1.b.0abb}%
\end{align}

\textbf{(a)} Proposition \ref{prop.tl-as-sum} yields%
\begin{align*}
t_{i}  &  =\sum_{w=i}^{n}\left(  i\Longrightarrow w\right)  =\sum_{k=i}%
^{n}\left(  i\Longrightarrow k\right) \\
&  =\sum_{k=i}^{j-1}\left(  i\Longrightarrow k\right)  +\sum_{k=j}%
^{n}\underbrace{\left(  i\Longrightarrow k\right)  }_{\substack{=\left(
i\Longrightarrow j\right)  \left(  j\Longrightarrow k\right)  \\\text{(by
Lemma \ref{lem.cyc-ijk},}\\\text{since }i\leq j\leq k\text{)}}%
}\ \ \ \ \ \ \ \ \ \ \left(  \text{since }i<j\leq n\right) \\
&  =\sum_{k=i}^{j-1}\left(  i\Longrightarrow k\right)  +\underbrace{\sum
_{k=j}^{n}\left(  i\Longrightarrow j\right)  \left(  j\Longrightarrow
k\right)  }_{=\left(  i\Longrightarrow j\right)  \sum_{k=j}^{n}\left(
j\Longrightarrow k\right)  }\\
&  =\sum_{k=i}^{j-1}\left(  i\Longrightarrow k\right)  +\left(
i\Longrightarrow j\right)  \underbrace{\sum_{k=j}^{n}\left(  j\Longrightarrow
k\right)  }_{\substack{=\sum_{w=j}^{n}\left(  j\Longrightarrow w\right)
=t_{j}\\\text{(since Proposition \ref{prop.tl-as-sum}}\\\text{yields }%
t_{j}=\sum_{w=j}^{n}\left(  j\Longrightarrow w\right)  \text{)}}}\\
&  =\sum_{k=i}^{j-1}\left(  i\Longrightarrow k\right)  +\left(
i\Longrightarrow j\right)  t_{j}.
\end{align*}
Thus,%
\begin{align*}
\left[  t_{i},t_{j}\right]   &  =\left[  \sum_{k=i}^{j-1}\left(
i\Longrightarrow k\right)  +\left(  i\Longrightarrow j\right)  t_{j}%
,\ t_{j}\right] \\
&  =\sum_{k=i}^{j-1}\underbrace{\left[  \left(  i\Longrightarrow k\right)
,\ t_{j}\right]  }_{\substack{=0\\\text{(by
(\ref{eq.lem.commute-with-tj-specific.comm}) (since }i\leq k<j\text{))}%
}}+\left[  \left(  i\Longrightarrow j\right)  t_{j},\ t_{j}\right] \\
&  \ \ \ \ \ \ \ \ \ \ \ \ \ \ \ \ \ \ \ \ \left(  \text{since the commutator
}\left[  a,b\right]  \text{ is bilinear in }a\text{ and }b\right) \\
&  =\underbrace{\sum_{k=i}^{j-1}0}_{=0}+\left[  \left(  i\Longrightarrow
j\right)  t_{j},\ t_{j}\right]  =\left[  \left(  i\Longrightarrow j\right)
t_{j},\ t_{j}\right] \\
&  =\left[  \left(  i\Longrightarrow j\right)  ,t_{j}\right]  t_{j}%
\ \ \ \ \ \ \ \ \ \ \left(  \text{by (\ref{pf.lem.ti-to-tj-1.b.0abb}), applied
to }a=\left(  i\Longrightarrow j\right)  \text{ and }b=t_{j}\right) \\
&  =\left[  s_{i}s_{i+1}\cdots s_{j-1},t_{j}\right]  t_{j}%
\end{align*}
(since Proposition \ref{prop.cycpq.sss} yields $\left(  i\Longrightarrow
j\right)  =s_{i}s_{i+1}\cdots s_{j-1}$). This proves Lemma
\ref{lem.ti-to-tj-1} \textbf{(a)}. \medskip

\textbf{(b)} Set $a:=s_{i}s_{i+1}\cdots s_{j-2}$ and $b:=s_{j-1}$ and
$c:=t_{j}$. Thus,%
\begin{equation}
ab=\left(  s_{i}s_{i+1}\cdots s_{j-2}\right)  s_{j-1}=s_{i}s_{i+1}\cdots
s_{j-1}. \label{pf.lem.ti-to-tj-1.b.ab}%
\end{equation}

However, $i\leq j-1$ (since $i<j$). Hence, Proposition \ref{prop.cycpq.sss}
(applied to $v=i$ and $w=j-1$) yields $\left(  i\Longrightarrow j-1\right)
=s_{i}s_{i+1}\cdots s_{\left(  j-1\right)  -1}=s_{i}s_{i+1}\cdots s_{j-2}=a$
(since $a=s_{i}s_{i+1}\cdots s_{j-2}$).

Now, $i\leq j-1<j$. Hence, (\ref{eq.lem.commute-with-tj-specific.comm})
(applied to $k=j-1$) yields $\left[  \left(  i\Longrightarrow j-1\right)
,\ t_{j}\right]  =0$. In view of $\left(  i\Longrightarrow j-1\right)  =a$ and
$t_{j}=c$, we can rewrite this as $\left[  a,c\right]  =0$. Hence,
(\ref{pf.lem.ti-to-tj-1.b.0}) becomes%
\begin{equation}
\left[  ab,c\right]  =\underbrace{\left[  a,c\right]  }_{=0}b+a\left[
b,c\right]  =a\left[  b,c\right]  . \label{pf.lem.ti-to-tj-1.b.abc}%
\end{equation}

On the other hand, applying Lemma \ref{lem.ti-to-tj-1} \textbf{(a)} to $j-1$
instead of $j$, we obtain
\begin{equation}
\left[  t_{j-1},t_{j}\right]  =\left[  s_{j-1},t_{j}\right]  t_{j}.
\label{pf.lem.ti-to-tj-1.b.1}%
\end{equation}

However, Lemma \ref{lem.ti-to-tj-1} \textbf{(a)} yields%
\begin{align*}
\left[  t_{i},t_{j}\right]   &  =\left[  \underbrace{s_{i}s_{i+1}\cdots
s_{j-1}}_{\substack{=ab\\\text{(by (\ref{pf.lem.ti-to-tj-1.b.ab}))}%
}},\underbrace{t_{j}}_{=c}\right]  t_{j}=\underbrace{\left[  ab,c\right]
}_{\substack{=a\left[  b,c\right]  \\\text{(by (\ref{pf.lem.ti-to-tj-1.b.abc}%
))}}}t_{j}=\underbrace{a}_{=s_{i}s_{i+1}\cdots s_{j-2}}\left[  \underbrace{b}%
_{=s_{j-1}},\underbrace{c}_{=t_{j}}\right]  t_{j}\\
&  =\left(  s_{i}s_{i+1}\cdots s_{j-2}\right)  \underbrace{\left[
s_{j-1},t_{j}\right]  t_{j}}_{\substack{=\left[  t_{j-1},t_{j}\right]
\\\text{(by (\ref{pf.lem.ti-to-tj-1.b.1}))}}}=\left(  s_{i}s_{i+1}\cdots
s_{j-2}\right)  \left[  t_{j-1},t_{j}\right]  .
\end{align*}
This proves Lemma \ref{lem.ti-to-tj-1} \textbf{(b)}.
\end{proof}

\begin{corollary}
\label{cor.titjtj-1}Let $i\in\left[  n\right]  $ and $j\in\left[  2,n\right]
$ be such that $i\leq j$. Then,%
\[
\left[  t_{i},t_{j}\right]  t_{j-1}=0.
\]

\end{corollary}

\begin{proof}
If $i=j$, then this is obvious (because $i=j$ entails $\left[  t_{i}%
,t_{j}\right]  =\left[  t_{j},t_{j}\right]  =t_{j}t_{j}-t_{j}t_{j}=0$ and
therefore $\underbrace{\left[  t_{i},t_{j}\right]  }_{=0}t_{j-1}=0$). Hence,
for the rest of this proof, we WLOG assume that $i\neq j$.

Combining $i\leq j$ with $i\neq j$, we obtain $i<j$. Hence, Lemma
\ref{lem.ti-to-tj-1} \textbf{(b)} yields
\[
\left[  t_{i},t_{j}\right]  =\left(  s_{i}s_{i+1}\cdots s_{j-2}\right)
\left[  t_{j-1},t_{j}\right]  .
\]

On the other hand, $j-1\in\left[  n-1\right]  $ (since $j\in\left[
2,n\right]  $). Thus, (\ref{eq.cor.titi+12=0.prod}) (applied to $j-1$ instead
of $i$) yields $\left[  t_{j-1},t_{j-1+1}\right]  t_{j-1}=0$. In other words,
$\left[  t_{j-1},t_{j}\right]  t_{j-1}=0$ (since $j-1+1=j$). Thus,%
\[
\underbrace{\left[  t_{i},t_{j}\right]  }_{=\left(  s_{i}s_{i+1}\cdots
s_{j-2}\right)  \left[  t_{j-1},t_{j}\right]  }t_{j-1}=\left(  s_{i}%
s_{i+1}\cdots s_{j-2}\right)  \underbrace{\left[  t_{j-1},t_{j}\right]
t_{j-1}}_{=0}=0.
\]
This proves Corollary \ref{cor.titjtj-1}.
\end{proof}

\subsection{The identity $\left(  1+s_{j}\right)  \left[  t_{i},t_{j}\right]
=0$ for all $i\leq j$}

We are now ready to prove the following surprising result:

\begin{theorem}
\label{thm.1+sjtitj}Let $i,j\in\left[  n-1\right]  $ satisfy $i\leq j$. Then,
\[
\left(  1+s_{j}\right)  \left[  t_{i},t_{j}\right]  =0.
\]

\end{theorem}

\begin{proof}
If $i=j$, then this is obvious (since $i=j$ entails $\left[  t_{i}%
,t_{j}\right]  =\left[  t_{j},t_{j}\right]  =0$). Hence, we WLOG assume that
$i\neq j$. Thus, $i<j$ (since $i\leq j$).

The transpositions $s_{i},s_{i+1},\ldots,s_{j-2}$ all commute with $s_{j}$ (by
reflection locality, since the numbers $i,i+1,\ldots,j-2$ differ by more than
$1$ from $j$). Thus, their product $s_{i}s_{i+1}\cdots s_{j-2}$ commutes with
$s_{j}$ as well. In other words,%
\[
s_{j}\left(  s_{i}s_{i+1}\cdots s_{j-2}\right)  =\left(  s_{i}s_{i+1}\cdots
s_{j-2}\right)  s_{j}.
\]
Thus, in $\mathbf{k}\left[  S_{n}\right]  $, we have%
\begin{align}
\left(  1+s_{j}\right)  \left(  s_{i}s_{i+1}\cdots s_{j-2}\right)   &
=s_{i}s_{i+1}\cdots s_{j-2}+\underbrace{s_{j}\left(  s_{i}s_{i+1}\cdots
s_{j-2}\right)  }_{=\left(  s_{i}s_{i+1}\cdots s_{j-2}\right)  s_{j}%
}\nonumber\\
&  =s_{i}s_{i+1}\cdots s_{j-2}+\left(  s_{i}s_{i+1}\cdots s_{j-2}\right)
s_{j}\nonumber\\
&  =\left(  s_{i}s_{i+1}\cdots s_{j-2}\right)  \left(  1+s_{j}\right)  .
\label{pf.thm.1+sjtitj.2}%
\end{align}

However, Lemma \ref{lem.ti-to-tj-1} \textbf{(b)} yields $\left[  t_{i}%
,t_{j}\right]  =\left(  s_{i}s_{i+1}\cdots s_{j-2}\right)  \left[
t_{j-1},t_{j}\right]  $ (since $i<j$). Hence,%
\begin{align*}
\left(  1+s_{j}\right)  \underbrace{\left[  t_{i},t_{j}\right]  }_{=\left(
s_{i}s_{i+1}\cdots s_{j-2}\right)  \left[  t_{j-1},t_{j}\right]  }  &
=\underbrace{\left(  1+s_{j}\right)  \left(  s_{i}s_{i+1}\cdots s_{j-2}%
\right)  }_{\substack{=\left(  s_{i}s_{i+1}\cdots s_{j-2}\right)  \left(
1+s_{j}\right)  \\\text{(by (\ref{pf.thm.1+sjtitj.2}))}}}\left[  t_{j-1}%
,t_{j}\right] \\
&  =\left(  s_{i}s_{i+1}\cdots s_{j-2}\right)  \underbrace{\left(
1+s_{j}\right)  \left[  t_{j-1},t_{j}\right]  }_{\substack{=0\\\text{(by Lemma
\ref{lem.1+sjtj-1tj})}}}=0.
\end{align*}
This proves Theorem \ref{thm.1+sjtitj}.
\end{proof}

\begin{corollary}
\label{cor.tn-1titn-1=0}Let $n\geq2$ and $i\in\left[  n\right]  $. Then,
$t_{n-1}\left[  t_{i},t_{n-1}\right]  =0$.
\end{corollary}

\begin{proof}
This is true for $i=n$ (because $t_{n}=1$ and thus $\left[  t_{n}%
,t_{n-1}\right]  =\left[  1,t_{n-1}\right]  =1t_{n-1}-t_{n-1}1=t_{n-1}%
-t_{n-1}=0$ and therefore $t_{n-1}\underbrace{\left[  t_{n},t_{n-1}\right]
}_{=0}=0$). Hence, we WLOG assume that $i\neq n$. Therefore, $i\in\left[
n\right]  \setminus\left\{  n\right\}  =\left[  n-1\right]  $. Also,
$n-1\in\left[  n-1\right]  $ (since $n\geq2$).

The definition of $t_{n-1}$ yields $t_{n-1}=\underbrace{\operatorname*{cyc}%
\nolimits_{n-1}}_{=1}+\underbrace{\operatorname*{cyc}\nolimits_{n-1,n}%
}_{=s_{n-1}}=1+s_{n-1}$.

However, Theorem \ref{thm.1+sjtitj} (applied to $j=n-1$) yields $\left(
1+s_{n-1}\right)  \left[  t_{i},t_{n-1}\right]  =0$. In view of $t_{n-1}%
=1+s_{n-1}$, we can rewrite this as $t_{n-1}\left[  t_{i},t_{n-1}\right]  =0$.
This proves Corollary \ref{cor.tn-1titn-1=0}.
\end{proof}

\section{The identity $\left[  t_{i},t_{j}\right]  ^{\left\lceil \left(
n-j\right)  /2\right\rceil +1}=0$ for all $i,j\in\left[  n\right]  $}

\subsection{The elements $s_{k}^{+}$ and the left ideals $H_{k,j}$}

We now introduce two crucial notions for the proof of our first main theorem:

\begin{definition}
\label{def.si+}We set $\mathbf{A}:=\mathbf{k}\left[  S_{n}\right]  $.
Furthermore, for any $i\in\left[  n-1\right]  $, we set
\[
s_{i}^{+}:=s_{i}+1\in\mathbf{A}.
\]
We also set $s_{i}^{+}:=1\in\mathbf{A}$ for all integers $i\notin\left[
n-1\right]  $. Thus, $s_{i}^{+}$ is defined for all integers $i$.
\end{definition}

\begin{definition}
\label{def.Hkj}Let $k$ and $j$ be two integers. Then, we define%
\[
H_{k,j}:=\sum_{\substack{u\in\left[  j,k\right]  ;\\u\equiv
k\operatorname{mod}2}}\mathbf{A}s_{u}^{+}.
\]
This is a left ideal of $\mathbf{A}$. Note that%
\begin{equation}
H_{k,j}=0\ \ \ \ \ \ \ \ \ \ \text{whenever }k<j. \label{eq.def.Hkj.=0}%
\end{equation}

\end{definition}

\begin{example}
We have%
\begin{align*}
H_{7,3}  &  =\sum_{\substack{u\in\left[  3,7\right]  ;\\u\equiv
7\operatorname{mod}2}}\mathbf{A}s_{u}^{+}=\mathbf{A}s_{3}^{+}+\mathbf{A}%
s_{5}^{+}+\mathbf{A}s_{7}^{+}\ \ \ \ \ \ \ \ \ \ \text{and}\\
H_{7,2}  &  =\sum_{\substack{u\in\left[  2,7\right]  ;\\u\equiv
7\operatorname{mod}2}}\mathbf{A}s_{u}^{+}=\mathbf{A}s_{3}^{+}+\mathbf{A}%
s_{5}^{+}+\mathbf{A}s_{7}^{+},
\end{align*}
so that $H_{7,2}=H_{7,3}$. Similarly, $H_{7,4}=H_{7,5}=\mathbf{A}s_{5}%
^{+}+\mathbf{A}s_{7}^{+}$ and $H_{7,6}=H_{7,7}=\mathbf{A}s_{7}^{+}$.
\end{example}

Let us prove some basic properties of the left ideals $H_{k,j}$:

\begin{remark}
\label{rmk.Hnj}Let $k$ be an integer such that $k\notin\left[  n-1\right]  $.
Let $j\in\left[  n\right]  $ satisfy $j\leq k$. Then, $H_{k,j}=\mathbf{A}$.
\end{remark}

\begin{proof}
Since $k\notin\left[  n-1\right]  $, we have $s_{k}^{+}=1$ (by the definition
of $s_{k}^{+}$). Also, $k\in\left[  j,k\right]  $ (since $j\leq k$).

Recall that $H_{k,j}$ is defined as the sum $\sum_{\substack{u\in\left[
j,k\right]  ;\\u\equiv k\operatorname{mod}2}}\mathbf{A}s_{u}^{+}$. But this
sum contains the addend $\mathbf{A}s_{k}^{+}$ (since $k\in\left[  j,k\right]
$ and $k\equiv k\operatorname{mod}2$). Hence, $\sum_{\substack{u\in\left[
j,k\right]  ;\\u\equiv k\operatorname{mod}2}}\mathbf{A}s_{u}^{+}%
\supseteq\mathbf{A}\underbrace{s_{k}^{+}}_{=1}=\mathbf{A}1=\mathbf{A}$. Now,%
\[
H_{k,j}=\sum_{\substack{u\in\left[  j,k\right]  ;\\u\equiv k\operatorname{mod}%
2}}\mathbf{A}s_{u}^{+}\supseteq\mathbf{A},
\]
so that $H_{k,j}=\mathbf{A}$. This proves Remark \ref{rmk.Hnj}.
\end{proof}

\begin{lemma}
\label{lem.Hkj.some-basics}Let $k$ and $j$ be two integers. Then,
$H_{k,j}\subseteq H_{k,j-1}$.
\end{lemma}

\begin{proof}
Definition \ref{def.Hkj} yields $H_{k,j}=\sum_{\substack{u\in\left[
j,k\right]  ;\\u\equiv k\operatorname{mod}2}}\mathbf{A}s_{u}^{+}$ and
$H_{k,j-1}=\sum_{\substack{u\in\left[  j-1,k\right]  ;\\u\equiv
k\operatorname{mod}2}}\mathbf{A}s_{u}^{+}$. But clearly, any addend of the
former sum is an addend of the latter sum as well (since each $u\in\left[
j,k\right]  $ satisfies $u\in\left[  j,k\right]  \subseteq\left[
j-1,k\right]  $). Thus, the former sum is a subset of the latter. In other
words, $H_{k,j}\subseteq H_{k,j-1}$. This proves Lemma
\ref{lem.Hkj.some-basics}.
\end{proof}

\begin{lemma}
\label{lem.Hkj.some-basics-2}Let $v$, $w$ and $j$ be three integers such that
$v\leq w$ and $v\equiv w\operatorname{mod}2$. Then, $H_{v,j}\subseteq H_{w,j}$.
\end{lemma}

\begin{proof}
Definition \ref{def.Hkj} yields%
\begin{align}
H_{v,j}  &  =\sum_{\substack{u\in\left[  j,v\right]  ;\\u\equiv
v\operatorname{mod}2}}\mathbf{A}s_{u}^{+}\ \ \ \ \ \ \ \ \ \ \text{and}%
\label{pf.lem.Hkj.some-basics-2.1}\\
H_{w,j}  &  =\sum_{\substack{u\in\left[  j,w\right]  ;\\u\equiv
w\operatorname{mod}2}}\mathbf{A}s_{u}^{+}. \label{pf.lem.Hkj.some-basics-2.2}%
\end{align}

However, each $u\in\left[  j,v\right]  $ satisfying $u\equiv
v\operatorname{mod}2$ is also an element of $\left[  j,w\right]  $ (since
$u\leq v\leq w$) and satisfies $u\equiv w\operatorname{mod}2$ (since $u\equiv
v\equiv w\operatorname{mod}2$). Thus, any addend of the sum $\sum
_{\substack{u\in\left[  j,v\right]  ;\\u\equiv v\operatorname{mod}%
2}}\mathbf{A}s_{u}^{+}$ is also an addend of the sum $\sum_{\substack{u\in
\left[  j,w\right]  ;\\u\equiv w\operatorname{mod}2}}\mathbf{A}s_{u}^{+}$.
Therefore, the former sum is a subset of the latter sum. In other words,
$\sum_{\substack{u\in\left[  j,v\right]  ;\\u\equiv v\operatorname{mod}%
2}}\mathbf{A}s_{u}^{+}\subseteq\sum_{\substack{u\in\left[  j,w\right]
;\\u\equiv w\operatorname{mod}2}}\mathbf{A}s_{u}^{+}$. In view of
(\ref{pf.lem.Hkj.some-basics-2.1}) and (\ref{pf.lem.Hkj.some-basics-2.2}), we
can rewrite this as $H_{v,j}\subseteq H_{w,j}$. This proves Lemma
\ref{lem.Hkj.some-basics-2}.
\end{proof}

\begin{lemma}
\label{lem.Hkj.some-basics-3}Let $k$ and $j$ be two integers such that
$k\equiv j\operatorname{mod}2$. Then, $H_{k,j-1}=H_{k,j}$.
\end{lemma}

\begin{proof}
Definition \ref{def.Hkj} yields%
\begin{align}
H_{k,j}  &  =\sum_{\substack{u\in\left[  j,k\right]  ;\\u\equiv
k\operatorname{mod}2}}\mathbf{A}s_{u}^{+}\ \ \ \ \ \ \ \ \ \ \text{and}%
\label{pf.lem.Hkj.some-basics-3.1}\\
H_{k,j-1}  &  =\sum_{\substack{u\in\left[  j-1,k\right]  ;\\u\equiv
k\operatorname{mod}2}}\mathbf{A}s_{u}^{+}. \label{pf.lem.Hkj.some-basics-3.2}%
\end{align}

However, each element $u\in\left[  j,k\right]  $ satisfying $u\equiv
k\operatorname{mod}2$ is also an element of $\left[  j-1,k\right]  $ (since
$u\in\left[  j,k\right]  \subseteq\left[  j-1,k\right]  $). Conversely, each
element $u\in\left[  j-1,k\right]  $ satisfying $u\equiv k\operatorname{mod}2$
is also an element of $\left[  j,k\right]  $ (since otherwise, it would equal
$j-1$, so that we would have $j-1=u\equiv k\equiv j\operatorname{mod}2$, but
this would contradict $j-1\not \equiv j\operatorname{mod}2$). Therefore, the
elements $u\in\left[  j-1,k\right]  $ satisfying $u\equiv k\operatorname{mod}%
2$ are precisely the elements $u\in\left[  j,k\right]  $ satisfying $u\equiv
k\operatorname{mod}2$. In other words, the sum on the right hand side of
(\ref{pf.lem.Hkj.some-basics-3.2}) ranges over the same set as the sum on the
right hand side of (\ref{pf.lem.Hkj.some-basics-3.1}). Therefore, the right
hand sides of the equalities (\ref{pf.lem.Hkj.some-basics-3.2}) and
(\ref{pf.lem.Hkj.some-basics-3.1}) are equal. Hence, their left hand sides
must also be equal. In other words, $H_{k,j-1}=H_{k,j}$. This proves Lemma
\ref{lem.Hkj.some-basics-3}.
\end{proof}

The following easy property follows from Lemma \ref{lem.suip}:

\begin{lemma}
\label{lem.su+ip}Let $i,u,v\in\left[  n\right]  $ be such that $i<u<v$. Then,%
\[
s_{u}^{+}\left(  i\Longrightarrow v\right)  =\left(  i\Longrightarrow
v\right)  s_{u-1}^{+}.
\]

\end{lemma}

\begin{proof}
We have $u\in\left[  n-1\right]  $ (since $u<v\leq n$ and $u>i\geq1$) and thus
$s_{u}^{+}=s_{u}+1$.

Also, $u>i\geq1$, so that $u-1>0$. Hence, $u-1\in\left[  n-1\right]  $ (since
$u-1<u<v\leq n$ and $u-1>0$) and thus $s_{u-1}^{+}=s_{u-1}+1$. Hence,%
\begin{align*}
&  \underbrace{s_{u}^{+}}_{=s_{u}+1}\left(  i\Longrightarrow v\right)
-\left(  i\Longrightarrow v\right)  \underbrace{s_{u-1}^{+}}_{=s_{u-1}+1}\\
&  =\left(  s_{u}+1\right)  \left(  i\Longrightarrow v\right)  -\left(
i\Longrightarrow v\right)  \left(  s_{u-1}+1\right) \\
&  =s_{u}\left(  i\Longrightarrow v\right)  +\left(  i\Longrightarrow
v\right)  -\left(  i\Longrightarrow v\right)  s_{u-1}-\left(  i\Longrightarrow
v\right) \\
&  =\underbrace{s_{u}\left(  i\Longrightarrow v\right)  }_{\substack{=\left(
i\Longrightarrow v\right)  s_{u-1}\\\text{(by Lemma \ref{lem.suip})}}}-\left(
i\Longrightarrow v\right)  s_{u-1}=\left(  i\Longrightarrow v\right)
s_{u-1}-\left(  i\Longrightarrow v\right)  s_{u-1}=0.
\end{align*}
Thus, $s_{u}^{+}\left(  i\Longrightarrow v\right)  =\left(  i\Longrightarrow
v\right)  s_{u-1}^{+}$. This proves Lemma \ref{lem.su+ip}.
\end{proof}

Let us also derive a simple consequence from Corollary \ref{cor.tl-via-tl+1}:

\begin{lemma}
\label{lem.sj-1+0}Let $j\in\left[  2,n\right]  $. Then,%
\[
s_{j-1}^{+}\left(  t_{j}-\left(  t_{j-1}-1\right)  \right)  =0.
\]

\end{lemma}

\begin{proof}
From $j\in\left[  2,n\right]  $, we obtain $j-1\in\left[  n-1\right]  $.
Hence, the definition of $s_{j-1}^{+}$ yields $s_{j-1}^{+}=s_{j-1}%
+1=1+s_{j-1}$.

Corollary \ref{cor.tl-via-tl+1} (applied to $\ell=j-1$) yields $t_{j-1}%
=1+s_{j-1}t_{\left(  j-1\right)  +1}=1+s_{j-1}t_{j}$ (since $\left(
j-1\right)  +1=j$). Thus, $t_{j-1}-1=s_{j-1}t_{j}$, so that%
\[
t_{j}-\underbrace{\left(  t_{j-1}-1\right)  }_{=s_{j-1}t_{j}}=t_{j}%
-s_{j-1}t_{j}=\left(  1-s_{j-1}\right)  t_{j}.
\]
Thus,%
\[
\underbrace{s_{j-1}^{+}}_{=1+s_{j-1}}\underbrace{\left(  t_{j}-\left(
t_{j-1}-1\right)  \right)  }_{=\left(  1-s_{j-1}\right)  t_{j}}%
=\underbrace{\left(  1+s_{j-1}\right)  \left(  1-s_{j-1}\right)
}_{\substack{=1-s_{j-1}^{2}=0\\\text{(since (\ref{eq.si.invol}) yields
}s_{j-1}^{2}=\operatorname*{id}=1\text{)}}}t_{j}=0.
\]
This proves Lemma \ref{lem.sj-1+0}.
\end{proof}

\subsection{The fuse}

The next lemma will help us analyze the behavior of the ideals $H_{k,j}$ under
repeated multiplication by $t_{j}$'s:

\begin{lemma}
\label{lem.4}Let $j\in\left[  n\right]  $ and $k\in\left[  n+1\right]  $ be
such that $j<k$. Then:

\begin{enumerate}
\item[\textbf{(a)}] If $k\not \equiv j\operatorname{mod}2$, then $s_{k}%
^{+}t_{j}\in H_{k-1,j}$.

\item[\textbf{(b)}] If $k\equiv j\operatorname{mod}2$, then $s_{k}^{+}\left(
t_{j}-1\right)  \in H_{k-1,j}$.
\end{enumerate}
\end{lemma}

\begin{proof}
If $k=n+1$, then both parts of Lemma \ref{lem.4} hold for fairly obvious
reasons\footnote{\textit{Proof.} Assume that $k=n+1$. Then, $k-1=n\notin%
\left[  n-1\right]  $ and $j\leq k-1$ (since $j<k$), so that $H_{k-1,j}%
=\mathbf{A}$ (by Remark \ref{rmk.Hnj}, applied to $k-1$ instead of $k$). Thus,
both elements $s_{k}^{+}t_{j}$ and $s_{k}^{+}\left(  t_{j}-1\right)  $ belong
to $H_{k-1,j}$ (since they both belong to $\mathbf{A}$). Therefore, both parts
of Lemma \ref{lem.4} hold. Qed.}. Hence, for the rest of this proof, we WLOG
assume that $k\neq n+1$. Therefore, $k\in\left[  n+1\right]  \setminus\left\{
n+1\right\}  =\left[  n\right]  $, so that $k\leq n$.

Recall that $H_{k-1,j}$ is a left ideal of $\mathbf{A}$, therefore an additive
subgroup of $\mathbf{A}$.

We have $j<k$, so that $j\leq k-1$. Hence, $k-1\in\left[  j,k-1\right]  $.

The definition of $H_{k-1,j}$ yields%
\begin{equation}
H_{k-1,j}=\sum_{\substack{u\in\left[  j,k-1\right]  ;\\u\equiv
k-1\operatorname{mod}2}}\mathbf{A}s_{u}^{+}. \label{pf.lem.4.H=}%
\end{equation}
Since $\mathbf{A}s_{k-1}^{+}$ is an addend of the sum on the right hand side
here (because $k-1\in\left[  j,k-1\right]  $ and $k-1\equiv
k-1\operatorname{mod}2$), we thus conclude that $\mathbf{A}s_{k-1}%
^{+}\subseteq H_{k-1,j}$.

Proposition \ref{prop.tl-as-sum} yields%
\[
t_{j}=\sum_{w=j}^{n}\left(  j\Longrightarrow w\right)  .
\]
Hence,%
\begin{align}
s_{k}^{+}t_{j}  &  =s_{k}^{+}\sum_{w=j}^{n}\left(  j\Longrightarrow w\right)
=\sum_{w=j}^{n}s_{k}^{+}\left(  j\Longrightarrow w\right) \nonumber\\
&  =\underbrace{\sum_{w=j}^{k}s_{k}^{+}\left(  j\Longrightarrow w\right)
}_{=s_{k}^{+}\sum_{w=j}^{k}\left(  j\Longrightarrow w\right)  }+\sum
_{w=k+1}^{n}\underbrace{s_{k}^{+}\left(  j\Longrightarrow w\right)
}_{\substack{=\left(  j\Longrightarrow w\right)  s_{k-1}^{+}\\\text{(by Lemma
\ref{lem.su+ip},}\\\text{since }j<k<w\text{)}}}\ \ \ \ \ \ \ \ \ \ \left(
\text{since }j\leq k\leq n\right) \nonumber\\
&  =s_{k}^{+}\sum_{w=j}^{k}\left(  j\Longrightarrow w\right)  +\sum
_{w=k+1}^{n}\left(  j\Longrightarrow w\right)  s_{k-1}^{+}. \label{pf.lem.4.2}%
\end{align}
We now need a better understanding of the sums on the right hand side. For
this purpose, we observe that every $w\in\left[  j,n-1\right]  $ satisfies%
\begin{align}
\left(  j\Longrightarrow w\right)  +\underbrace{\left(  j\Longrightarrow
w+1\right)  }_{\substack{=\left(  j\Longrightarrow w\right)  s_{w}\\\text{(by
Proposition \ref{prop.cycpq.rec} \textbf{(b)})}}}  &  =\left(
j\Longrightarrow w\right)  +\left(  j\Longrightarrow w\right)  s_{w}%
\nonumber\\
&  =\left(  j\Longrightarrow w\right)  \underbrace{\left(  1+s_{w}\right)
}_{\substack{=s_{w}+1=s_{w}^{+}\\\text{(by the definition of }s_{w}%
^{+}\text{)}}}=\left(  j\Longrightarrow w\right)  s_{w}^{+}\nonumber\\
&  \in\mathbf{A}s_{w}^{+}. \label{pf.lem.4.twoterms}%
\end{align}

\medskip

\textbf{(a)} Assume that $k\not \equiv j\operatorname{mod}2$. Thus, the
integer $k-j$ is odd, so that $k-j+1$ is even.

Now,%
\begin{align*}
\sum_{w=j}^{k}\left(  j\Longrightarrow w\right)   &  =\left(  j\Longrightarrow
j\right)  +\left(  j\Longrightarrow j+1\right)  +\left(  j\Longrightarrow
j+2\right)  +\cdots+\left(  j\Longrightarrow k\right) \\
&  =\underbrace{\left(  \left(  j\Longrightarrow j\right)  +\left(
j\Longrightarrow j+1\right)  \right)  }_{\substack{\in\mathbf{A}s_{j}%
^{+}\\\text{(by (\ref{pf.lem.4.twoterms}))}}}\\
&  \ \ \ \ \ \ \ \ \ \ +\underbrace{\left(  \left(  j\Longrightarrow
j+2\right)  +\left(  j\Longrightarrow j+3\right)  \right)  }_{\substack{\in
\mathbf{A}s_{j+2}^{+}\\\text{(by (\ref{pf.lem.4.twoterms}))}}}\\
&  \ \ \ \ \ \ \ \ \ \ +\underbrace{\left(  \left(  j\Longrightarrow
j+4\right)  +\left(  j\Longrightarrow j+5\right)  \right)  }_{\substack{\in
\mathbf{A}s_{j+4}^{+}\\\text{(by (\ref{pf.lem.4.twoterms}))}}}\\
&  \ \ \ \ \ \ \ \ \ \ +\cdots\\
&  \ \ \ \ \ \ \ \ \ \ +\underbrace{\left(  \left(  j\Longrightarrow
k-1\right)  +\left(  j\Longrightarrow k\right)  \right)  }_{\substack{\in
\mathbf{A}s_{k-1}^{+}\\\text{(by (\ref{pf.lem.4.twoterms}))}}}\\
&  \ \ \ \ \ \ \ \ \ \ \ \ \ \ \ \ \ \ \ \ \left(
\begin{array}
[c]{c}%
\text{here, we have split our sum into pairs of}\\
\text{consecutive addends, since }k-j+1\text{ is even}%
\end{array}
\right) \\
&  \in\mathbf{A}s_{j}^{+}+\mathbf{A}s_{j+2}^{+}+\mathbf{A}s_{j+4}^{+}%
+\cdots+\mathbf{A}s_{k-1}^{+}=\sum_{\substack{u\in\left[  j,k-1\right]
;\\u\equiv k-1\operatorname{mod}2}}\mathbf{A}s_{u}^{+}\\
&  =H_{k-1,j}\ \ \ \ \ \ \ \ \ \ \left(  \text{by (\ref{pf.lem.4.H=})}\right)
.
\end{align*}

Now, (\ref{pf.lem.4.2}) becomes%
\begin{align*}
s_{k}^{+}t_{j}  &  =s_{k}^{+}\underbrace{\sum_{w=j}^{k}\left(
j\Longrightarrow w\right)  }_{\in H_{k-1,j}}+\sum_{w=k+1}^{n}%
\underbrace{\left(  j\Longrightarrow w\right)  s_{k-1}^{+}}_{\in
\mathbf{A}s_{k-1}^{+}\subseteq H_{k-1,j}}\\
&  \in s_{k}^{+}H_{k-1,j}+\sum_{w=k+1}^{n}H_{k-1,j}\subseteq H_{k-1,j}%
\ \ \ \ \ \ \ \ \ \ \left(  \text{since }H_{k-1,j}\text{ is a left ideal of
}\mathbf{A}\right)  .
\end{align*}
This proves Lemma \ref{lem.4} \textbf{(a)}. \medskip

\textbf{(b)} Assume that $k\equiv j\operatorname{mod}2$. Thus, the integer
$k-j$ is even.

Now,%
\begin{align}
\sum_{w=j+1}^{k}\left(  j\Longrightarrow w\right)   &  =\left(
j\Longrightarrow j+1\right)  +\left(  j\Longrightarrow j+2\right)  +\left(
j\Longrightarrow j+3\right)  +\cdots+\left(  j\Longrightarrow k\right)
\nonumber\\
&  =\underbrace{\left(  \left(  j\Longrightarrow j+1\right)  +\left(
j\Longrightarrow j+2\right)  \right)  }_{\substack{\in\mathbf{A}s_{j+1}%
^{+}\\\text{(by (\ref{pf.lem.4.twoterms}))}}}\nonumber\\
&  \ \ \ \ \ \ \ \ \ \ +\underbrace{\left(  \left(  j\Longrightarrow
j+3\right)  +\left(  j\Longrightarrow j+4\right)  \right)  }_{\substack{\in
\mathbf{A}s_{j+3}^{+}\\\text{(by (\ref{pf.lem.4.twoterms}))}}}\nonumber\\
&  \ \ \ \ \ \ \ \ \ \ +\underbrace{\left(  \left(  j\Longrightarrow
j+5\right)  +\left(  j\Longrightarrow j+6\right)  \right)  }_{\substack{\in
\mathbf{A}s_{j+5}^{+}\\\text{(by (\ref{pf.lem.4.twoterms}))}}}\nonumber\\
&  \ \ \ \ \ \ \ \ \ \ +\cdots\nonumber\\
&  \ \ \ \ \ \ \ \ \ \ +\underbrace{\left(  \left(  j\Longrightarrow
k-1\right)  +\left(  j\Longrightarrow k\right)  \right)  }_{\substack{\in
\mathbf{A}s_{k-1}^{+}\\\text{(by (\ref{pf.lem.4.twoterms}))}}}\nonumber\\
&  \ \ \ \ \ \ \ \ \ \ \ \ \ \ \ \ \ \ \ \ \left(
\begin{array}
[c]{c}%
\text{here, we have split our sum into pairs of}\\
\text{consecutive addends, since }k-j\text{ is even}%
\end{array}
\right) \nonumber\\
&  \in\mathbf{A}s_{j+1}^{+}+\mathbf{A}s_{j+3}^{+}+\mathbf{A}s_{j+5}^{+}%
+\cdots+\mathbf{A}s_{k-1}^{+}=\sum_{\substack{u\in\left[  j,k-1\right]
;\\u\equiv k-1\operatorname{mod}2}}\mathbf{A}s_{u}^{+}\nonumber\\
&  =H_{k-1,j}\ \ \ \ \ \ \ \ \ \ \left(  \text{by (\ref{pf.lem.4.H=})}\right)
. \label{pf.lem.4.2.b.1}%
\end{align}

But%
\begin{align*}
\sum_{w=j}^{k}\left(  j\Longrightarrow w\right)   &  =\underbrace{\left(
j\Longrightarrow j\right)  }_{=\operatorname*{id}=1}+\sum_{w=j+1}^{k}\left(
j\Longrightarrow w\right) \\
&  \ \ \ \ \ \ \ \ \ \ \ \ \ \ \ \ \ \ \ \ \left(
\begin{array}
[c]{c}%
\text{here, we have split off the}\\
\text{addend for }w=j\text{ from the sum}%
\end{array}
\right) \\
&  =1+\sum_{w=j+1}^{k}\left(  j\Longrightarrow w\right)  .
\end{align*}
Now, (\ref{pf.lem.4.2}) becomes%
\begin{align*}
s_{k}^{+}t_{j}  &  =s_{k}^{+}\underbrace{\sum_{w=j}^{k}\left(
j\Longrightarrow w\right)  }_{=1+\sum_{w=j+1}^{k}\left(  j\Longrightarrow
w\right)  }+\sum_{w=k+1}^{n}\left(  j\Longrightarrow w\right)  s_{k-1}^{+}\\
&  =s_{k}^{+}\left(  1+\sum_{w=j+1}^{k}\left(  j\Longrightarrow w\right)
\right)  +\sum_{w=k+1}^{n}\left(  j\Longrightarrow w\right)  s_{k-1}^{+}\\
&  =s_{k}^{+}+s_{k}^{+}\sum_{w=j+1}^{k}\left(  j\Longrightarrow w\right)
+\sum_{w=k+1}^{n}\left(  j\Longrightarrow w\right)  s_{k-1}^{+}.
\end{align*}
Subtracting $s_{k}^{+}$ from both sides of this equality, we find%
\begin{align*}
s_{k}^{+}t_{j}-s_{k}^{+}  &  =s_{k}^{+}\underbrace{\sum_{w=j+1}^{k}\left(
j\Longrightarrow w\right)  }_{\substack{\in H_{k-1,j}\\\text{(by
(\ref{pf.lem.4.2.b.1}))}}}+\sum_{w=k+1}^{n}\underbrace{\left(
j\Longrightarrow w\right)  s_{k-1}^{+}}_{\in\mathbf{A}s_{k-1}^{+}\subseteq
H_{k-1,j}}\\
&  \in s_{k}^{+}H_{k-1,j}+\sum_{w=k+1}^{n}H_{k-1,j}\\
&  \subseteq H_{k-1,j}\ \ \ \ \ \ \ \ \ \ \left(  \text{since }H_{k-1,j}\text{
is a left ideal of }\mathbf{A}\right)  .
\end{align*}
Hence, $s_{k}^{+}\left(  t_{j}-1\right)  =s_{k}^{+}t_{j}-s_{k}^{+}\in
H_{k-1,j}$. This proves Lemma \ref{lem.4} \textbf{(b)}.
\end{proof}

From Lemma \ref{lem.4} and Lemma \ref{lem.sj-1+0}, we can easily obtain the following:

\begin{lemma}
\label{lem.4+0}Let $j\in\left[  2,n\right]  $ and $u\in\left[  n\right]  $ be
such that $u\geq j-1$ and $u\equiv j-1\operatorname{mod}2$. Then,%
\[
s_{u}^{+}\left(  t_{j}-\left(  t_{j-1}-1\right)  \right)  \in H_{u-1,j}.
\]

\end{lemma}

\begin{proof}
If $u=j-1$, then this follows from
\begin{align*}
s_{j-1}^{+}\left(  t_{j}-\left(  t_{j-1}-1\right)  \right)   &
=0\ \ \ \ \ \ \ \ \ \ \left(  \text{by Lemma \ref{lem.sj-1+0}}\right) \\
&  \in H_{\left(  j-1\right)  -1,j}\ \ \ \ \ \ \ \ \ \ \left(  \text{since
}H_{\left(  j-1\right)  -1,j}\text{ is a left ideal}\right)  .
\end{align*}
Thus, for the rest of this proof, we WLOG assume that $u\neq j-1$. Combining
this with $u\geq j-1$, we obtain $u>j-1$. Therefore, $u\geq\left(  j-1\right)
+1=j$. Moreover, $u\neq j$ (since $u\equiv j-1\not \equiv j\operatorname{mod}%
2$). Combining this with $u\geq j$, we obtain $u>j$. Thus, $u\geq j+1$.

Also, from $j\in\left[  2,n\right]  $, we obtain $j-1\in\left[  n-1\right]
\subseteq\left[  n\right]  $. From $u>j$, we obtain $j<u$. Moreover, $u\equiv
j-1\not \equiv j\operatorname{mod}2$. Hence, Lemma \ref{lem.4} \textbf{(a)}
(applied to $k=u$) yields $s_{u}^{+}t_{j}\in H_{u-1,j}$ (since $j<u$).
Furthermore, Lemma \ref{lem.4} \textbf{(b)} (applied to $u$ and $j-1$ instead
of $k$ and $j$) yields $s_{u}^{+}\left(  t_{j-1}-1\right)  \in H_{u-1,j-1}$
(since $u\equiv j-1\operatorname{mod}2$ and $j-1<j<u$). Moreover,
$H_{u-1,j}\subseteq H_{u-1,j-1}$ (by Lemma \ref{lem.Hkj.some-basics}, applied
to $k=u-1$).

Finally, from $u\equiv j-1\operatorname{mod}2$, we obtain $u-1\equiv\left(
j-1\right)  -1=j-2\equiv j\operatorname{mod}2$. Hence, Lemma
\ref{lem.Hkj.some-basics-3} (applied to $k=u-1$) yields $H_{u-1,j-1}%
=H_{u-1,j}$.

Altogether, we now have
\[
s_{u}^{+}\left(  t_{j}-\left(  t_{j-1}-1\right)  \right)  =\underbrace{s_{u}%
^{+}t_{j}}_{\in H_{u-1,j}}-\underbrace{s_{u}^{+}\left(  t_{j-1}-1\right)
}_{\in H_{u-1,j-1}=H_{u-1,j}}\in H_{u-1,j}-H_{u-1,j}\subseteq H_{u-1,j}%
\]
(since $H_{u-1,j}$ is a left ideal of $\mathbf{A}$). This proves Lemma
\ref{lem.4+0}.
\end{proof}

Using Lemma \ref{lem.4+0} with Lemma \ref{lem.4} \textbf{(a)}, we can obtain
the following:

\begin{lemma}
\label{lem.5}Let $j\in\left[  n\right]  $ and $k\in\left[  n+1\right]  $ be
such that $1<j\leq k$ and $k\equiv j\operatorname{mod}2$. Then,%
\[
s_{k}^{+}\left[  t_{j-1},t_{j}\right]  \in H_{k-2,j}.
\]

\end{lemma}

\begin{proof}
From $j\in\left[  n\right]  $ and $1<j$, we obtain $j\in\left[  2,n\right]  $.
Thus, $j-1\in\left[  n-1\right]  $.

Hence, (\ref{eq.cor.titi+12=0.comm}) (applied to $i=j-1$) yields%
\begin{equation}
\left[  t_{j-1},t_{j}\right]  =t_{j-1}\left(  t_{j}-\left(  t_{j-1}-1\right)
\right)  . \label{pf.lem.5.comm}%
\end{equation}
Multiplying this equality by $s_{k}^{+}$ from the left, we obtain%
\begin{equation}
s_{k}^{+}\left[  t_{j-1},t_{j}\right]  =s_{k}^{+}t_{j-1}\left(  t_{j}-\left(
t_{j-1}-1\right)  \right)  . \label{pf.lem.5.scomm}%
\end{equation}

However, $k-1\equiv j-1\operatorname{mod}2$ (since $k\equiv
j\operatorname{mod}2$). Furthermore, we have $j-1\in\left[  n-1\right]
\subseteq\left[  n\right]  $ and $j-1<j\leq k$ and $k\equiv j\not \equiv
j-1\operatorname{mod}2$. Thus, Lemma \ref{lem.4} \textbf{(a)} (applied to
$j-1$ instead of $j$) yields%
\[
s_{k}^{+}t_{j-1}\in H_{k-1,j-1}=\sum_{\substack{u\in\left[  j-1,k-1\right]
;\\u\equiv k-1\operatorname{mod}2}}\mathbf{A}s_{u}^{+}%
\ \ \ \ \ \ \ \ \ \ \left(  \text{by the definition of }H_{k-1,j-1}\right)  .
\]
In other words, we can write $s_{k}^{+}t_{j-1}$ in the form%
\begin{equation}
s_{k}^{+}t_{j-1}=\sum_{\substack{u\in\left[  j-1,k-1\right]  ;\\u\equiv
k-1\operatorname{mod}2}}a_{u}s_{u}^{+}, \label{pf.lem.5.asum}%
\end{equation}
where $a_{u}\in\mathbf{A}$ is an element for each $u\in\left[  j-1,k-1\right]
$ satisfying $u\equiv k-1\operatorname{mod}2$.

Consider these elements $a_{u}$. Now, (\ref{pf.lem.5.scomm}) becomes%
\begin{align}
s_{k}^{+}\left[  t_{j-1},t_{j}\right]   &  =s_{k}^{+}t_{j-1}\left(
t_{j}-\left(  t_{j-1}-1\right)  \right) \nonumber\\
&  =\left(  \sum_{\substack{u\in\left[  j-1,k-1\right]  ;\\u\equiv
k-1\operatorname{mod}2}}a_{u}s_{u}^{+}\right)  \left(  t_{j}-\left(
t_{j-1}-1\right)  \right)  \ \ \ \ \ \ \ \ \ \ \left(  \text{by
(\ref{pf.lem.5.asum})}\right) \nonumber\\
&  =\sum_{\substack{u\in\left[  j-1,k-1\right]  ;\\u\equiv
k-1\operatorname{mod}2}}a_{u}s_{u}^{+}\left(  t_{j}-\left(  t_{j-1}-1\right)
\right)  . \label{pf.lem.5.sksum}%
\end{align}
However, every $u\in\left[  j-1,k-1\right]  $ satisfying $u\equiv
k-1\operatorname{mod}2$ satisfies%
\begin{equation}
s_{u}^{+}\left(  t_{j}-\left(  t_{j-1}-1\right)  \right)  \in H_{k-2,j}.
\label{pf.lem.5.p-lem}%
\end{equation}

[\textit{Proof of (\ref{pf.lem.5.p-lem}):} Let $u\in\left[  j-1,k-1\right]  $
be such that $u\equiv k-1\operatorname{mod}2$.

We have $j\in\left[  2,n\right]  $. Moreover, $u\in\left[  j-1,k-1\right]  $
shows that $u\geq j-1$ and $u\leq k-1\leq n$ (since $k\leq n+1$). Thus,
$u\in\left[  n\right]  $ (since $u\leq n$). Furthermore, $u\equiv k-1\equiv
j-1\operatorname{mod}2$. Thus, Lemma \ref{lem.4+0} yields $s_{u}^{+}\left(
t_{j}-\left(  t_{j-1}-1\right)  \right)  \in H_{u-1,j}$.

However, from $u\leq k-1$, we obtain $u-1\leq\left(  k-1\right)  -1=k-2$.
Moreover, from $u\equiv k-1\operatorname{mod}2$, we obtain $u-1\equiv\left(
k-1\right)  -1=k-2\operatorname{mod}2$. These two facts entail $H_{u-1,j}%
\subseteq H_{k-2,j}$ (by Lemma \ref{lem.Hkj.some-basics-2}, applied to $u-1$
and $k-2$ instead of $v$ and $w$).

Hence, $s_{u}^{+}\left(  t_{j}-\left(  t_{j-1}-1\right)  \right)  \in
H_{u-1,j}\subseteq H_{k-2,j}$. This proves (\ref{pf.lem.5.p-lem}).] \medskip

Now, (\ref{pf.lem.5.sksum}) becomes%
\[
s_{k}^{+}\left[  t_{j-1},t_{j}\right]  =\sum_{\substack{u\in\left[
j-1,k-1\right]  ;\\u\equiv k-1\operatorname{mod}2}}a_{u}\underbrace{s_{u}%
^{+}\left(  t_{j}-\left(  t_{j-1}-1\right)  \right)  }_{\substack{\in
H_{k-2,j}\\\text{(by (\ref{pf.lem.5.p-lem}))}}}\in\sum_{\substack{u\in\left[
j-1,k-1\right]  ;\\u\equiv k-1\operatorname{mod}2}}a_{u}H_{k-2,j}\subseteq
H_{k-2,j}%
\]
(since $H_{k-2,j}$ is a left ideal of $\mathbf{A}$). This proves Lemma
\ref{lem.5}.
\end{proof}

\begin{lemma}
\label{lem.6}Let $i,j\in\left[  n\right]  $ and $k\in\left[  n+1\right]  $ be
such that $i\leq j$ and $k\equiv j\operatorname{mod}2$. Then,%
\[
H_{k,j}\left[  t_{i},t_{j}\right]  \subseteq H_{k-2,j}.
\]

\end{lemma}

\begin{proof}
This is obvious for $i=j$ (since we have $\left[  t_{i},t_{j}\right]  =\left[
t_{j},t_{j}\right]  =0$ in this case). Thus, we WLOG assume that $i\neq j$.
Hence, $i<j$ (since $i\leq j$). Therefore, $1\leq i<j$.

Set $a:=s_{i}s_{i+1}\cdots s_{j-2}$. We have $i<j$. Thus, Lemma
\ref{lem.ti-to-tj-1} \textbf{(b)} yields
\[
\left[  t_{i},t_{j}\right]  =\underbrace{\left(  s_{i}s_{i+1}\cdots
s_{j-2}\right)  }_{=a}\left[  t_{j-1},t_{j}\right]  =a\left[  t_{j-1}%
,t_{j}\right]  .
\]

However, it is easy to see that%
\begin{equation}
s_{u}^{+}a=as_{u}^{+}\ \ \ \ \ \ \ \ \ \ \text{for each }u\in\left[
j,k\right]  . \label{pf.lem.6.su+a}%
\end{equation}

[\textit{Proof of (\ref{pf.lem.6.su+a}):} Let $u\in\left[  j,k\right]  $. We
must prove that $s_{u}^{+}a=as_{u}^{+}$.

If $u\notin\left[  n-1\right]  $, then $s_{u}^{+}=1$ (by the definition of
$s_{u}^{+}$), and thus this claim boils down to $1a=a1$, which is obvious.
Thus, we WLOG assume that $u\in\left[  n-1\right]  $. Hence, $s_{u}^{+}%
=s_{u}+1$.

However, from $u\in\left[  j,k\right]  $, we obtain $u\geq j$. Hence, $j\leq
u$, so that $j-2\leq u-2$. Thus, each of the integers $i,i+1,\ldots,j-2$ has a
distance larger than $1$ from $u$. Hence, each of the transpositions
$s_{i},s_{i+1},\ldots,s_{j-2}$ commutes with $s_{u}$ (by reflection locality).
Therefore, the product $s_{i}s_{i+1}\cdots s_{j-2}$ of these transpositions
also commutes with $s_{u}$. In other words, $a$ commutes with $s_{u}$ (since
$a=s_{i}s_{i+1}\cdots s_{j-2}$). In other words, $s_{u}a=as_{u}$. Now,%
\[
\underbrace{s_{u}^{+}}_{=s_{u}+1}a=\left(  s_{u}+1\right)  a=\underbrace{s_{u}%
a}_{=as_{u}}+\,a=as_{u}+a=a\underbrace{\left(  s_{u}+1\right)  }_{=s_{u}^{+}%
}=as_{u}^{+}.
\]
This proves (\ref{pf.lem.6.su+a}).] \medskip

Using (\ref{pf.lem.6.su+a}), we can easily see the following: For each
$u\in\left[  j,k\right]  $ satisfying $u\equiv k\operatorname{mod}2$, we have%
\begin{equation}
s_{u}^{+}\left[  t_{i},t_{j}\right]  \in H_{k-2,j}. \label{pf.lem.6.su+abra}%
\end{equation}

[\textit{Proof of (\ref{pf.lem.6.su+abra}):} Let $u\in\left[  j,k\right]  $ be
such that $u\equiv k\operatorname{mod}2$. From (\ref{pf.lem.6.su+a}), we
obtain $s_{u}^{+}a=as_{u}^{+}$. Hence,%
\begin{equation}
s_{u}^{+}\underbrace{\left[  t_{i},t_{j}\right]  }_{=a\left[  t_{j-1}%
,t_{j}\right]  }=\underbrace{s_{u}^{+}a}_{=as_{u}^{+}}\left[  t_{j-1}%
,t_{j}\right]  =as_{u}^{+}\left[  t_{j-1},t_{j}\right]  .
\label{pf.lem.6.su+abra.pf.1}%
\end{equation}

However, $u\in\left[  j,k\right]  \subseteq\left[  k\right]  \subseteq\left[
n+1\right]  $ and $1<j\leq u$ and $u\equiv k\equiv j\operatorname{mod}2$.
Thus, Lemma \ref{lem.5} (applied to $u$ instead of $k$) yields $s_{u}%
^{+}\left[  t_{j-1},t_{j}\right]  \in H_{u-2,j}$.

Furthermore, $u-2\leq k-2$ (since $u\leq k$) and $u-2\equiv
k-2\operatorname{mod}2$ (since $u\equiv k\operatorname{mod}2$). Hence, Lemma
\ref{lem.Hkj.some-basics-2} (applied to $v=u-2$ and $w=k-2$) yields
$H_{u-2,j}\subseteq H_{k-2,j}$.

Now, (\ref{pf.lem.6.su+abra.pf.1}) becomes%
\begin{align*}
s_{u}^{+}\left[  t_{i},t_{j}\right]   &  =a\underbrace{s_{u}^{+}\left[
t_{j-1},t_{j}\right]  }_{\substack{\in H_{u-2,j}}}\in aH_{u-2,j}\subseteq
H_{u-2,j}\ \ \ \ \ \ \ \ \ \ \left(  \text{since }H_{u-2,j}\text{ is a left
ideal}\right) \\
&  \subseteq H_{k-2,j}.
\end{align*}
This proves (\ref{pf.lem.6.su+abra}).] \medskip

Now,%
\begin{align*}
\underbrace{H_{k,j}}_{\substack{=\sum_{\substack{u\in\left[  j,k\right]
;\\u\equiv k\operatorname{mod}2}}\mathbf{A}s_{u}^{+}\\\text{(by the definition
of }H_{k,j}\text{)}}}\left[  t_{i},t_{j}\right]   &  =\left(  \sum
_{\substack{u\in\left[  j,k\right]  ;\\u\equiv k\operatorname{mod}%
2}}\mathbf{A}s_{u}^{+}\right)  \left[  t_{i},t_{j}\right]  =\sum
_{\substack{u\in\left[  j,k\right]  ;\\u\equiv k\operatorname{mod}%
2}}\mathbf{A}\underbrace{s_{u}^{+}\left[  t_{i},t_{j}\right]  }_{\substack{\in
H_{k-2,j}\\\text{(by (\ref{pf.lem.6.su+abra}))}}}\\
&  \subseteq\sum_{\substack{u\in\left[  j,k\right]  ;\\u\equiv
k\operatorname{mod}2}}\mathbf{A}H_{k-2,j}\subseteq H_{k-2,j}%
\end{align*}
(since $H_{k-2,j}$ is a left ideal). This proves Lemma \ref{lem.6}.
\end{proof}

If $i,j\in\left[  n\right]  $ and $k\in\left[  n+1\right]  $ are such that
$i\leq j$ and $k\equiv j\operatorname{mod}2$, then we can apply Lemma
\ref{lem.6} recursively, yielding%
\begin{align*}
H_{k,j}\left[  t_{i},t_{j}\right]   &  \subseteq H_{k-2,j},\\
H_{k,j}\left[  t_{i},t_{j}\right]  ^{2}  &  \subseteq H_{k-4,j},\\
H_{k,j}\left[  t_{i},t_{j}\right]  ^{3}  &  \subseteq H_{k-6,j},\\
&  \ldots.
\end{align*}
Eventually, the right hand side will be $0$, and thus we obtain $H_{k,j}%
\left[  t_{i},t_{j}\right]  ^{s}=0$ for some $s\in\mathbb{N}$. By picking $k$
appropriately (specifically, setting $k=n$ or $k=n+1$ depending on the parity
of $n-j$), we can ensure that $H_{k,j}=\mathbf{A}$, and thus this equality
$H_{k,j}\left[  t_{i},t_{j}\right]  ^{s}=0$ yields $\left[  t_{i}%
,t_{j}\right]  ^{s}=0$. Thus, Lemma \ref{lem.6} \textquotedblleft lays a
fuse\textquotedblright\ for proving the nilpotency of $\left[  t_{i}%
,t_{j}\right]  $. We shall now elaborate on this.

\subsection{Products of $\left[  t_{i},t_{j}\right]  $'s for a fixed $j$}

\begin{lemma}
\label{lem.right-bound-1}Let $j\in\left[  n\right]  $ and $m\in\mathbb{N}$.
Let $r$ be the unique element of $\left\{  n,n+1\right\}  $ that is congruent
to $j$ modulo $2$. (That is, $r=%
\begin{cases}
n, & \text{if }n\equiv j\operatorname{mod}2;\\
n+1, & \text{otherwise.}%
\end{cases}
$)

Let $i_{1},i_{2},\ldots,i_{m}$ be $m$ elements of $\left[  j\right]  $ (not
necessarily distinct). Then,
\[
\left[  t_{i_{1}},t_{j}\right]  \left[  t_{i_{2}},t_{j}\right]  \cdots\left[
t_{i_{m}},t_{j}\right]  \in H_{r-2m,j}.
\]

\end{lemma}

\begin{proof}
We induct on $m$:

\textit{Base case:} We have $r\geq n$ (by the definition of $r$), so that
$r\notin\left[  n-1\right]  $ and $j\leq r$ (since $j\leq n\leq r$). Hence,
Remark \ref{rmk.Hnj} (applied to $k=r$) yields $H_{r,j}=\mathbf{A}$. Now%
\[
\left[  t_{i_{1}},t_{j}\right]  \left[  t_{i_{2}},t_{j}\right]  \cdots\left[
t_{i_{0}},t_{j}\right]  =\left(  \text{empty product}\right)  =1\in
\mathbf{A}=H_{r,j}=H_{r-2\cdot0,j}%
\]
(since $r=r-2\cdot0$). In other words, Lemma \ref{lem.right-bound-1} is proved
for $m=0$.

\textit{Induction step:} Let $m\in\mathbb{N}$. Assume (as the induction
hypothesis) that
\begin{equation}
\left[  t_{i_{1}},t_{j}\right]  \left[  t_{i_{2}},t_{j}\right]  \cdots\left[
t_{i_{m}},t_{j}\right]  \in H_{r-2m,j} \label{pf.lem.right-bound-1.IH}%
\end{equation}
whenever $i_{1},i_{2},\ldots,i_{m}$ are $m$ elements of $\left[  j\right]  $.
We must prove that%
\begin{equation}
\left[  t_{i_{1}},t_{j}\right]  \left[  t_{i_{2}},t_{j}\right]  \cdots\left[
t_{i_{m+1}},t_{j}\right]  \in H_{r-2\left(  m+1\right)  ,j}
\label{pf.lem.right-bound-1.IG}%
\end{equation}
whenever $i_{1},i_{2},\ldots,i_{m+1}$ are $m+1$ elements of $\left[  j\right]
$.

So let $i_{1},i_{2},\ldots,i_{m+1}$ be $m+1$ elements of $\left[  j\right]  $.
We have $r-2m\equiv r\equiv j\operatorname{mod}2$ (by the definition of $r$)
and $i_{m+1}\in\left[  j\right]  \subseteq\left[  n\right]  $ and $i_{m+1}\leq
j$ (since $i_{m+1}\in\left[  j\right]  $). Hence, Lemma \ref{lem.6} (applied
to $k=r-2m$ and $i=i_{m+1}$) yields\footnote{Strictly speaking, this argument
works only if $r-2m\in\left[  n+1\right]  $ (since Lemma \ref{lem.6} requires
$k\in\left[  n+1\right]  $). However, in all remaining cases, we can get to
the same result in an even simpler way: Namely, assume that $r-2m\notin\left[
n+1\right]  $. Thus, $r-2m$ is either $\leq0$ or $>n+1$. Since $r-2m$ cannot
be $>n+1$ (because $r-2\underbrace{m}_{\geq0}\leq r\leq n+1$), we thus
conclude that $r-2m\leq0$. Hence, $r-2m\leq0<j$ and therefore $H_{r-2m,j}=0$
(by (\ref{eq.def.Hkj.=0})). Hence,%
\[
\underbrace{H_{r-2m,j}}_{=0}\left[  t_{i_{m+1}},t_{j}\right]  =0\subseteq
H_{r-2m-2,j}.
\]
}
\[
H_{r-2m,j}\left[  t_{i_{m+1}},t_{j}\right]  \subseteq H_{r-2m-2,j}.
\]
Now,%
\begin{align*}
\left[  t_{i_{1}},t_{j}\right]  \left[  t_{i_{2}},t_{j}\right]  \cdots\left[
t_{i_{m+1}},t_{j}\right]   &  =\underbrace{\left(  \left[  t_{i_{1}}%
,t_{j}\right]  \left[  t_{i_{2}},t_{j}\right]  \cdots\left[  t_{i_{m}}%
,t_{j}\right]  \right)  }_{\substack{\in H_{r-2m,j}\\\text{(by
(\ref{pf.lem.right-bound-1.IH}))}}}\cdot\left[  t_{i_{m+1}},t_{j}\right] \\
&  \in H_{r-2m,j}\left[  t_{i_{m+1}},t_{j}\right]  \subseteq H_{r-2m-2,j}%
=H_{r-2\left(  m+1\right)  ,j}%
\end{align*}
(since $r-2m-2=r-2\left(  m+1\right)  $). In other words,
(\ref{pf.lem.right-bound-1.IG}) holds. This completes the induction step.
Thus, Lemma \ref{lem.right-bound-1} is proved.
\end{proof}

We can now prove our first main result:

\begin{theorem}
\label{thm.right-bound}Let $j\in\left[  n\right]  $ and $m\in\mathbb{N}$ be
such that $2m\geq n-j+2$. Let $i_{1},i_{2},\ldots,i_{m}$ be $m$ elements of
$\left[  j\right]  $ (not necessarily distinct). Then,
\[
\left[  t_{i_{1}},t_{j}\right]  \left[  t_{i_{2}},t_{j}\right]  \cdots\left[
t_{i_{m}},t_{j}\right]  =0.
\]

\end{theorem}

\begin{proof}
Let $r$ be the element of $\left\{  n,n+1\right\}  $ defined in Lemma
\ref{lem.right-bound-1}. Then, $r\leq n+1$, so that%
\[
\underbrace{r}_{\leq n+1}-\underbrace{2m}_{\geq n-j+2}\leq\left(  n+1\right)
-\left(  n-j+2\right)  =j-1<j.
\]
Thus, $H_{r-2m,j}=0$ (by (\ref{eq.def.Hkj.=0})). But Lemma
\ref{lem.right-bound-1} yields%
\[
\left[  t_{i_{1}},t_{j}\right]  \left[  t_{i_{2}},t_{j}\right]  \cdots\left[
t_{i_{m}},t_{j}\right]  \in H_{r-2m,j}=0.
\]
In other words, $\left[  t_{i_{1}},t_{j}\right]  \left[  t_{i_{2}}%
,t_{j}\right]  \cdots\left[  t_{i_{m}},t_{j}\right]  =0$. This proves Theorem
\ref{thm.right-bound}.
\end{proof}

\subsection{The identity $\left[  t_{i},t_{j}\right]  ^{\left\lceil \left(
n-j\right)  /2\right\rceil +1}=0$ for any $i,j\in\left[  n\right]  $}

\begin{lemma}
\label{lem.right-bound-m-1}Let $i,j\in\left[  n\right]  $ and $m\in\mathbb{N}$
be such that $2m\geq n-j+2$ and $i\leq j$. Then, $\left[  t_{i},t_{j}\right]
^{m}=0$.
\end{lemma}

\begin{proof}
We have $i\in\left[  j\right]  $ (since $i\leq j$). Hence, Theorem
\ref{thm.right-bound} (applied to $i_{k}=i$) yields $\underbrace{\left[
t_{i},t_{j}\right]  \left[  t_{i},t_{j}\right]  \cdots\left[  t_{i}%
,t_{j}\right]  }_{m\text{ times}}=0$. In other words, $\left[  t_{i}%
,t_{j}\right]  ^{m}=0$. This proves Lemma \ref{lem.right-bound-m-1}.
\end{proof}

\begin{corollary}
\label{cor.right-bound-m}Let $i,j\in\left[  n\right]  $ and $m\in\mathbb{N}$
be such that $2m\geq n-j+2$. Then, $\left[  t_{i},t_{j}\right]  ^{m}=0$.
\end{corollary}

\begin{proof}
If $i\leq j$, then Corollary \ref{cor.right-bound-m} follows directly from
Lemma \ref{lem.right-bound-m-1}. Thus, we WLOG assume that we don't have
$i\leq j$. Hence, $i>j$.

Therefore, $j<i$, so that $j\leq i$. Moreover, $2m\geq n-\underbrace{j}%
_{<i}+2>n-i+2$. Hence, we can apply Lemma \ref{lem.right-bound-m-1} to $j$ and
$i$ instead of $i$ and $j$. We thus obtain $\left[  t_{j},t_{i}\right]
^{m}=0$. However, $\left[  t_{i},t_{j}\right]  =-\left[  t_{j},t_{i}\right]  $
(since any two elements $a$ and $b$ of a ring satisfy $\left[  a,b\right]
=-\left[  b,a\right]  $). Hence, $\left[  t_{i},t_{j}\right]  ^{m}=\left(
-\left[  t_{j},t_{i}\right]  \right)  ^{m}=\left(  -1\right)  ^{m}%
\underbrace{\left[  t_{j},t_{i}\right]  ^{m}}_{=0}=0$. This proves Corollary
\ref{cor.right-bound-m}.
\end{proof}

\begin{corollary}
\label{cor.right-bound}For any $x\in\mathbb{R}$, let $\left\lceil
x\right\rceil $ denote the smallest integer that is $\geq x$. Let
$i,j\in\left[  n\right]  $. Then, $\left[  t_{i},t_{j}\right]  ^{\left\lceil
\left(  n-j\right)  /2\right\rceil +1}=0$.
\end{corollary}

\begin{proof}
We have $2\left(  \underbrace{\left\lceil \left(  n-j\right)  /2\right\rceil
}_{\geq\left(  n-j\right)  /2}+1\right)  \geq2\left(  \left(  n-j\right)
/2+1\right)  =n-j+2$. Thus, Corollary \ref{cor.right-bound-m} (applied to
$m=\left\lceil \left(  n-j\right)  /2\right\rceil +1$) yields $\left[
t_{i},t_{j}\right]  ^{\left\lceil \left(  n-j\right)  /2\right\rceil +1}=0$.
This proves Corollary \ref{cor.right-bound}.
\end{proof}

\subsection{\label{subsec.ikinn}Can we lift the $i_{1},i_{2},\ldots,i_{m}%
\in\left[  j\right]  $ restriction?}

\begin{remark}
Theorem \ref{thm.right-bound} does not hold if we drop the $i_{1},i_{2}%
,\ldots,i_{m}\in\left[  j\right]  $ restriction. For instance, for $n=6$ and
$j=3$, we have%
\[
\left[  t_{1},t_{3}\right]  \left[  t_{5},t_{3}\right]  \left[  t_{4}%
,t_{3}\right]  \left[  t_{1},t_{3}\right]  \neq
0\ \ \ \ \ \ \ \ \ \ \text{despite }2\cdot4\geq n-j+2.
\]
Another counterexample is obtained for $n=4$ and $j=2$, since $\left[
t_{3},t_{2}\right]  \left[  t_{1},t_{2}\right]  \neq0$.
\end{remark}

Despite these counterexamples, the restriction can be lifted in some
particular cases. Here is a particularly simple instance:

\begin{corollary}
\label{cor.titn-1}Assume that $n\geq2$. Let $u,v\in\left[  n\right]  $. Then,
$\left[  t_{u},t_{n-1}\right]  \left[  t_{v},t_{n-1}\right]  =0$.
\end{corollary}

\begin{proof}
We are in one of the following three cases:

\textit{Case 1:} We have $u=n$.

\textit{Case 2:} We have $v=n$.

\textit{Case 3:} Neither $u$ nor $v$ equals $n$.

Let us first consider Case 1. In this case, we have $u=n$. Hence, $t_{u}%
=t_{n}=1$ and thus $\left[  t_{u},t_{n-1}\right]  =\left[  1,t_{n-1}\right]
=0$ (since $\left[  1,x\right]  =0$ for each $x$). Hence, $\underbrace{\left[
t_{u},t_{n-1}\right]  }_{=0}\left[  t_{v},t_{n-1}\right]  =0$. Thus, Corollary
\ref{cor.titn-1} is proved in Case 1.

A similar argument proves Corollary \ref{cor.titn-1} in Case 2.

Let us now consider Case 3. In this case, neither $u$ nor $v$ equals $n$. In
other words, $u$ and $v$ are both $\neq n$. Thus, $u$ and $v$ are elements of
$\left[  n\right]  \setminus\left\{  n\right\}  =\left[  n-1\right]  $. Hence,
Theorem \ref{thm.right-bound} (applied to $j=n-1$ and $m=2$ and $\left(
i_{1},i_{2},\ldots,i_{m}\right)  =\left(  u,v\right)  $) yields $\left[
t_{u},t_{n-1}\right]  \left[  t_{v},t_{n-1}\right]  =0$ (since $2\cdot
2=4\geq3=n-\left(  n-1\right)  +2$). Thus, Corollary \ref{cor.titn-1} is
proved in Case 3.

We have now proved Corollary \ref{cor.titn-1} in all three Cases 1, 2 and 3.
\end{proof}

\begin{proposition}
\label{prop.tn-2-three}Assume that $n\geq3$. Then:

\begin{enumerate}
\item[\textbf{(a)}] We have $\left[  t_{i},t_{n-2}\right]  \left[
s_{n-1},s_{n-2}\right]  =0$ for all $i\in\left[  n-2\right]  $.

\item[\textbf{(b)}] We have $\left[  t_{i},t_{n-2}\right]  \left[
t_{n-1},t_{n-2}\right]  =0$ for all $i\in\left[  n\right]  $.

\item[\textbf{(c)}] We have $\left[  t_{u},t_{n-2}\right]  \left[
t_{v},t_{n-2}\right]  \left[  t_{w},t_{n-2}\right]  =0$ for all $u,v,w\in
\left[  n\right]  $.
\end{enumerate}
\end{proposition}

\begin{proof}
[Proof sketch.]\textbf{(a)} This is easily checked for $i=n-3$ and for
$i=n-2$.\ \ \ \ \footnote{Indeed, the case of $i=n-2$ is obvious (since
$\left[  t_{n-2},t_{n-2}\right]  =0$). The case of $i=n-3$ requires some
calculations, which can be made simpler by checking that $\left[
t_{n-3},t_{n-2}\right]  $ is an element $a\in\mathbf{k}\left[  S_{n}\right]  $
satisfying $a=as_{n-2}=as_{n-1}$. (Explicitly, $\left[  t_{n-3},t_{n-2}%
\right]  =\left(  1-s_{n-2}\right)  s_{n-3}b$, where $b$ is the sum of all six
permutations in $S_{n}$ that fix each of $1,2,\ldots,n-3$.)} In all other
cases, Lemma \ref{lem.ti-to-tj-1} \textbf{(b)} lets us rewrite $\left[
t_{i},t_{n-2}\right]  $ as $\left(  s_{i}s_{i+1}\cdots s_{n-4}\right)  \left[
t_{n-3},t_{n-2}\right]  $, and thus it remains to prove that $\left[
t_{n-3},t_{n-2}\right]  \left[  s_{n-1},s_{n-2}\right]  =0$, which is exactly
the $i=n-3$ case. Thus, Proposition \ref{prop.tn-2-three} \textbf{(a)} is proved.

\textbf{(b)} This is easily checked for $i=n-1$ and for $i=n$. In all other
cases, we have $i\in\left[  n-2\right]  $, and an easy computation shows that
$\left[  t_{n-1},t_{n-2}\right]  =\left[  s_{n-1},s_{n-2}\right]  \left(
1+s_{n-1}\right)  $, so that the claim follows from Proposition
\ref{prop.tn-2-three} \textbf{(a)}. Thus, Proposition \ref{prop.tn-2-three}
\textbf{(b)} is proved.

\textbf{(c)} Let $u,v,w\in\left[  n\right]  $. We must prove that $\left[
t_{u},t_{n-2}\right]  \left[  t_{v},t_{n-2}\right]  \left[  t_{w}%
,t_{n-2}\right]  =0$. If any of $u,v,w$ equals $n$, then this is clear (since
$t_{n}=1$ and thus $\left[  t_{n},t_{n-2}\right]  =\left[  1,t_{n-2}\right]
=0$). Thus, WLOG assume that $u,v,w\in\left[  n-1\right]  $.

If $v=n-1$, then $\left[  t_{u},t_{n-2}\right]  \left[  t_{v},t_{n-2}\right]
=\left[  t_{u},t_{n-2}\right]  \left[  t_{n-1},t_{n-2}\right]  =0$ (by
Proposition \ref{prop.tn-2-three} \textbf{(b)}), so that our claim holds.
Likewise, our claim can be shown if $w=n-1$. Thus, WLOG assume that neither
$v$ nor $w$ equals $n-1$. Hence, $v,w\in\left[  n-2\right]  $. Therefore,
Theorem \ref{thm.right-bound} shows that $\left[  t_{u},t_{n-2}\right]
\left[  t_{v},t_{n-2}\right]  =0$, which yields our claim again. This proves
Proposition \ref{prop.tn-2-three} \textbf{(c)}.
\end{proof}

\section{The identity $\left[  t_{i},t_{j}\right]  ^{j-i+1}=0$ for all $i\leq
j$}

We now approach the proof of another remarkable theorem: the identity $\left[
t_{i},t_{j}\right]  ^{j-i+1}=0$, which holds for all $i,j\in\left[  n\right]
$ satisfying $i\leq j$. Some more work must be done before we can prove this.

\subsection{The elements $\mu_{i,j}$ for $i\in\left[  j-1\right]  $}

We first introduce a family of elements of the group algebra $\mathbf{k}%
\left[  S_{n}\right]  $.

\begin{definition}
Set $\mathbf{A}=\mathbf{k}\left[  S_{n}\right]  $.
\end{definition}

\begin{definition}
\label{def.muij}Let $j\in\left[  n\right]  $, and let $i\in\left[  j-1\right]
$. Then, $j-1\geq1$ (since $i\in\left[  j-1\right]  $ entails $1\leq i\leq
j-1$), so that $j-1\in\left[  n\right]  $. Hence, the elements $\left(
i\Longrightarrow j-1\right)  \in S_{n}$ and $t_{j-1}\in\mathbf{k}\left[
S_{n}\right]  $ are well-defined.

Now, we define an element%
\[
\mu_{i,j}:=\left(  i\Longrightarrow j-1\right)  t_{j-1}\in\mathbf{A}.
\]

\end{definition}

\begin{lemma}
\label{lem.com-mu}Let $j\in\left[  n\right]  $, and let $i\in\left[
j-1\right]  $. Then,%
\begin{align}
\left[  t_{i},t_{j}\right]   &  =\left(  i\Longrightarrow j-1\right)  \left[
t_{j-1},t_{j}\right] \label{eq.lem.com-mu.1}\\
&  =\mu_{i,j}\left(  t_{j}-t_{j-1}+1\right)  . \label{eq.lem.com-mu.2}%
\end{align}

\end{lemma}

\begin{proof}
From $i\in\left[  j-1\right]  $, we obtain $1\leq i\leq j-1$, so that
$j-1\geq1$. Thus, $j-1\in\left[  n-1\right]  $ (since $j-1<j\leq n$). Hence,
(\ref{eq.cor.titi+12=0.comm}) (applied to $j-1$ instead of $i$) yields
\[
\left[  t_{j-1},t_{j-1+1}\right]  =t_{j-1}\left(  t_{j-1+1}-\left(
t_{j-1}-1\right)  \right)  .
\]
Since $j-1+1=j$, we can rewrite this as%
\begin{equation}
\left[  t_{j-1},t_{j}\right]  =t_{j-1}\left(  t_{j}-\left(  t_{j-1}-1\right)
\right)  . \label{pf.lem.com-mu.0}%
\end{equation}

We have $i\leq j-1$. Hence, Proposition \ref{prop.cycpq.sss} (applied to $v=i$
and $w=j-1$) yields%
\begin{equation}
\left(  i\Longrightarrow j-1\right)  =s_{i}s_{i+1}\cdots s_{\left(
j-1\right)  -1}=s_{i}s_{i+1}\cdots s_{j-2}. \label{pf.lem.com-mu.1}%
\end{equation}
However, $i\leq j-1<j$. Thus, Lemma \ref{lem.ti-to-tj-1} \textbf{(b)} yields%
\[
\left[  t_{i},t_{j}\right]  =\underbrace{\left(  s_{i}s_{i+1}\cdots
s_{j-2}\right)  }_{\substack{=\left(  i\Longrightarrow j-1\right)  \\\text{(by
(\ref{pf.lem.com-mu.1}))}}}\left[  t_{j-1},t_{j}\right]  =\left(
i\Longrightarrow j-1\right)  \left[  t_{j-1},t_{j}\right]  .
\]
This proves (\ref{eq.lem.com-mu.1}). Furthermore,%
\begin{align*}
\left[  t_{i},t_{j}\right]   &  =\left(  i\Longrightarrow j-1\right)
\underbrace{\left[  t_{j-1},t_{j}\right]  }_{\substack{=t_{j-1}\left(
t_{j}-\left(  t_{j-1}-1\right)  \right)  \\\text{(by (\ref{pf.lem.com-mu.0}%
))}}}=\underbrace{\left(  i\Longrightarrow j-1\right)  t_{j-1}}%
_{\substack{=\mu_{i,j}\\\text{(by the definition of }\mu_{i,j}\text{)}%
}}\underbrace{\left(  t_{j}-\left(  t_{j-1}-1\right)  \right)  }%
_{=t_{j}-t_{j-1}+1}\\
&  =\mu_{i,j}\left(  t_{j}-t_{j-1}+1\right)  .
\end{align*}
This proves (\ref{eq.lem.com-mu.2}). Thus, Lemma \ref{lem.com-mu} is proved.
\end{proof}

\begin{lemma}
\label{lem.abc=cab}Let $R$ be a ring. Let $a,b,c\in R$ be three elements
satisfying $ca=ac$ and $cb=bc$. Then,%
\[
c\left[  a,b\right]  =\left[  a,b\right]  c.
\]

\end{lemma}

\begin{proof}
The definition of a commutator yields $\left[  a,b\right]  =ab-ba$. Thus,%
\begin{align*}
c\underbrace{\left[  a,b\right]  }_{=ab-ba}  & =c\left(  ab-ba\right)
=\underbrace{ca}_{=ac}b-\underbrace{cb}_{=bc}a=a\underbrace{cb}_{=bc}%
-\,b\underbrace{ca}_{=ac}\\
& =abc-bac=\underbrace{\left(  ab-ba\right)  }_{=\left[  a,b\right]
}c=\left[  a,b\right]  c.
\end{align*}
This proves Lemma \ref{lem.abc=cab}.
\end{proof}

\begin{lemma}
\label{lem.move-mu}Let $i,j,k\in\left[  n\right]  $ be such that $i\leq
k<j-1$. Then,%
\[
\left[  t_{i},t_{j}\right]  \mu_{k,j}=\mu_{k+1,j}\left[  t_{i},t_{j-1}\right]
.
\]

\end{lemma}

\begin{proof}
We have $j-1\geq j-1>k\geq i$. Thus, Lemma \ref{lem.jpiq1} (applied to $k$,
$j-1$ and $j-1$ instead of $j$, $v$ and $w$) yields%
\begin{align}
\left(  k+1\Longrightarrow j-1\right)  \left(  i\Longrightarrow j-1\right)
&  =\left(  i\Longrightarrow j-1\right)  \left(  k\Longrightarrow
\underbrace{\left(  j-1\right)  -1}_{=j-2}\right)  \nonumber\\
&  =\left(  i\Longrightarrow j-1\right)  \left(  k\Longrightarrow j-2\right)
.\label{pf.lem.move-mu.1}%
\end{align}

From $i<j-1$, we obtain $i\leq j-2$ and thus $i\in\left[  j-2\right]
\subseteq\left[  j-1\right]  $. Likewise, $k\in\left[  j-1\right]  $ (since
$k<j-1$).

Furthermore, Proposition \ref{prop.cycpq.rec} \textbf{(b)} (applied to $v=k$
and $w=j-1$) yields%
\begin{align}
\left(  k\Longrightarrow j-1\right)   &  =\left(  k\Longrightarrow\left(
j-1\right)  -1\right)  s_{\left(  j-1\right)  -1}\ \ \ \ \ \ \ \ \ \ \left(
\text{since }k<j-1\right)  \nonumber\\
&  =\left(  k\Longrightarrow j-2\right)  s_{j-2}\label{pf.lem.move-mu.2k}%
\end{align}
(since $\left(  j-1\right)  -1=j-2$). The same argument (applied to $i$
instead of $k$) yields%
\begin{equation}
\left(  i\Longrightarrow j-1\right)  =\left(  i\Longrightarrow j-2\right)
s_{j-2}\label{pf.lem.move-mu.2i}%
\end{equation}
(since $i<j-1$).

We have $k\leq j-2$ (since $k<j-1$) and $j-2<j$. Thus,
(\ref{eq.lem.commute-with-tj-specific.ab=ba}) (applied to $k$ and $j-2$
instead of $i$ and $k$) yields%
\begin{equation}
\left(  k\Longrightarrow j-2\right)  t_{j}=t_{j}\left(  k\Longrightarrow
j-2\right)  .\label{pf.lem.move-mu.3j}%
\end{equation}

Furthermore, we have $k\leq j-2$ and $j-2<j-1$. Thus,
(\ref{eq.lem.commute-with-tj-specific.ab=ba}) (applied to $k$, $j-2$ and $j-1$
instead of $i$, $k$ and $j$) yields
\begin{equation}
\left(  k\Longrightarrow j-2\right)  t_{j-1}=t_{j-1}\left(  k\Longrightarrow
j-2\right)  .\label{pf.lem.move-mu.3j-1}%
\end{equation}
The same argument (but using $i$ instead of $k$) shows that%
\begin{equation}
\left(  i\Longrightarrow j-2\right)  t_{j-1}=t_{j-1}\left(  i\Longrightarrow
j-2\right)  \label{pf.lem.move-mu.3j-1i}%
\end{equation}
(since $i\leq j-2$).

From (\ref{pf.lem.move-mu.3j}) and (\ref{pf.lem.move-mu.3j-1}), we obtain%
\begin{equation}
\left(  k\Longrightarrow j-2\right)  \left[  t_{j-1},t_{j}\right]  =\left[
t_{j-1},t_{j}\right]  \left(  k\Longrightarrow j-2\right)
\label{pf.lem.move-mu.3}%
\end{equation}
(by Lemma \ref{lem.abc=cab}, applied to $R=\mathbf{A}$, $a=t_{j-1}$, $b=t_{j}$
and $c=\left(  k\Longrightarrow j-2\right)  $).

Now, the definition of $\mu_{k,j}$ yields $\mu_{k,j}=\left(  k\Longrightarrow
j-1\right)  t_{j-1}$. Hence,%
\begin{align}
&  \underbrace{\left[  t_{i},t_{j}\right]  }_{\substack{=\left(
i\Longrightarrow j-1\right)  \left[  t_{j-1},t_{j}\right]  \\\text{(by
(\ref{eq.lem.com-mu.1}))}}}\ \ \underbrace{\mu_{k,j}}_{=\left(
k\Longrightarrow j-1\right)  t_{j-1}}\nonumber\\
&  =\left(  i\Longrightarrow j-1\right)  \left[  t_{j-1},t_{j}\right]
\underbrace{\left(  k\Longrightarrow j-1\right)  }_{\substack{=\left(
k\Longrightarrow j-2\right)  s_{j-2}\\\text{(by (\ref{pf.lem.move-mu.2k}))}%
}}t_{j-1}\nonumber\\
&  =\left(  i\Longrightarrow j-1\right)  \underbrace{\left[  t_{j-1}%
,t_{j}\right]  \left(  k\Longrightarrow j-2\right)  }_{\substack{=\left(
k\Longrightarrow j-2\right)  \left[  t_{j-1},t_{j}\right]  \\\text{(by
(\ref{pf.lem.move-mu.3}))}}}s_{j-2}t_{j-1}\nonumber\\
&  =\underbrace{\left(  i\Longrightarrow j-1\right)  \left(  k\Longrightarrow
j-2\right)  }_{\substack{=\left(  k+1\Longrightarrow j-1\right)  \left(
i\Longrightarrow j-1\right)  \\\text{(by (\ref{pf.lem.move-mu.1}))}}}\left[
t_{j-1},t_{j}\right]  s_{j-2}t_{j-1}\nonumber\\
&  =\left(  k+1\Longrightarrow j-1\right)  \underbrace{\left(
i\Longrightarrow j-1\right)  }_{\substack{=\left(  i\Longrightarrow
j-2\right)  s_{j-2}\\\text{(by (\ref{pf.lem.move-mu.2i}))}}}\left[
t_{j-1},t_{j}\right]  s_{j-2}t_{j-1}\nonumber\\
&  =\left(  k+1\Longrightarrow j-1\right)  \left(  i\Longrightarrow
j-2\right)  s_{j-2}\left[  t_{j-1},t_{j}\right]  s_{j-2}t_{j-1}%
.\label{pf.lem.move-mu.5}%
\end{align}

Next, we shall simplify the product $s_{j-2}\left[  t_{j-1},t_{j}\right]
s_{j-2}t_{j-1}$ on the right hand side.

Lemma \ref{lem.ti-to-tj-1} \textbf{(b)} (applied to $j-2$ instead of $i$)
yields that%
\begin{align}
\left[  t_{j-2},t_{j}\right]   &  =\underbrace{\left(  s_{j-2}s_{\left(
j-2\right)  +1}\cdots s_{j-2}\right)  }_{=s_{j-2}}\left[  t_{j-1}%
,t_{j}\right]  \ \ \ \ \ \ \ \ \ \ \left(  \text{since }j-2<j\right)
\nonumber\\
&  =s_{j-2}\left[  t_{j-1},t_{j}\right]  .\label{pf.lem.move-mu.6}%
\end{align}

Furthermore, $i\leq j-2$, so that $j-2\geq i\geq1$. Combining this with
$j-2\leq n-2$ (since $j\leq n$), we obtain $j-2\in\left[  n-2\right]
\subseteq\left[  n-1\right]  $. Hence, Corollary \ref{cor.tl-via-tl+1}
(applied to $\ell=j-2$) yields $t_{j-2}=1+s_{j-2}\underbrace{t_{\left(
j-2\right)  +1}}_{=t_{j-1}}=1+s_{j-2}t_{j-1}$. Hence,
\begin{equation}
t_{j-2}-1=s_{j-2}t_{j-1}.\label{pf.lem.move-mu.7}%
\end{equation}
Multiplying the equalities (\ref{pf.lem.move-mu.6}) and
(\ref{pf.lem.move-mu.7}) together, we obtain%
\begin{equation}
\left[  t_{j-2},t_{j}\right]  \left(  t_{j-2}-1\right)  =s_{j-2}\left[
t_{j-1},t_{j}\right]  s_{j-2}t_{j-1}.\label{pf.lem.move-mu.6x7}%
\end{equation}

On the other hand, Corollary \ref{cor.titi+2ti-1} (applied to $i=j-2$) yields
\[
\left[  t_{j-2},t_{j-2+2}\right]  \left(  t_{j-2}-1\right)  =t_{j-2+1}\left[
t_{j-2},t_{j-2+1}\right]  \ \ \ \ \ \ \ \ \ \ \left(  \text{since }%
j-2\in\left[  n-2\right]  \right)  .
\]
In view of $j-2+2=j$ and $j-2+1=j-1$, we can rewrite this as
\[
\left[  t_{j-2},t_{j}\right]  \left(  t_{j-2}-1\right)  =t_{j-1}\left[
t_{j-2},t_{j-1}\right]  .
\]
Comparing this with (\ref{pf.lem.move-mu.6x7}), we obtain%
\[
s_{j-2}\left[  t_{j-1},t_{j}\right]  s_{j-2}t_{j-1}=t_{j-1}\left[
t_{j-2},t_{j-1}\right]  .
\]
Hence, (\ref{pf.lem.move-mu.5}) becomes%
\begin{align}
\left[  t_{i},t_{j}\right]  \mu_{k,j}  &  =\left(  k+1\Longrightarrow
j-1\right)  \left(  i\Longrightarrow j-2\right)  \underbrace{s_{j-2}\left[
t_{j-1},t_{j}\right]  s_{j-2}t_{j-1}}_{=t_{j-1}\left[  t_{j-2},t_{j-1}\right]
}\nonumber\\
&  =\left(  k+1\Longrightarrow j-1\right)  \underbrace{\left(
i\Longrightarrow j-2\right)  t_{j-1}}_{\substack{=t_{j-1}\left(
i\Longrightarrow j-2\right)  \\\text{(by (\ref{pf.lem.move-mu.3j-1i}))}%
}}\left[  t_{j-2},t_{j-1}\right] \nonumber\\
&  =\left(  k+1\Longrightarrow j-1\right)  t_{j-1}\left(  i\Longrightarrow
j-2\right)  \left[  t_{j-2},t_{j-1}\right]  . \label{pf.lem.move-mu.9}%
\end{align}

But $k\leq j-2$, so that $k+1\leq j-1$. Thus, $k+1\in\left[  j-1\right]  $.
Hence, the definition of $\mu_{k+1,j}$ yields%
\begin{equation}
\mu_{k+1,j}=\left(  k+1\Longrightarrow j-1\right)  t_{j-1}.
\label{pf.lem.move-mu.10}%
\end{equation}
Furthermore, $i\leq j-2=\left(  j-1\right)  -1$, so that $i\in\left[  \left(
j-1\right)  -1\right]  $. Hence, (\ref{eq.lem.com-mu.1}) (applied to $j-1$
instead of $j$) yields%
\begin{align}
\left[  t_{i},t_{j-1}\right]   &  =\left(  i\Longrightarrow\left(  j-1\right)
-1\right)  \left[  t_{\left(  j-1\right)  -1},t_{j-1}\right] \nonumber\\
&  =\left(  i\Longrightarrow j-2\right)  \left[  t_{j-2},t_{j-1}\right]
\label{pf.lem.move-mu.11}%
\end{align}
(since $\left(  j-1\right)  -1=j-2$). Multiplying the equalities
(\ref{pf.lem.move-mu.10}) and (\ref{pf.lem.move-mu.11}), we obtain%
\[
\mu_{k+1,j}\left[  t_{i},t_{j-1}\right]  =\left(  k+1\Longrightarrow
j-1\right)  t_{j-1}\left(  i\Longrightarrow j-2\right)  \left[  t_{j-2}%
,t_{j-1}\right]  .
\]
Comparing this with (\ref{pf.lem.move-mu.9}), we obtain $\left[  t_{i}%
,t_{j}\right]  \mu_{k,j}=\mu_{k+1,j}\left[  t_{i},t_{j-1}\right]  $. This
proves Lemma \ref{lem.move-mu}.
\end{proof}

We can combine Lemma \ref{lem.com-mu} and Lemma \ref{lem.move-mu} into a
single result:

\begin{lemma}
\label{lem.increase-mu}Let $j\in\left[  n\right]  $, and let $i\in\left[
j\right]  $ and $k\in\left[  j-1\right]  $. Then, we have%
\[
\left(  \left[  t_{i},t_{j}\right]  \mu_{k,j}=0\right)  \text{ or }\left(
\left[  t_{i},t_{j}\right]  \mu_{k,j}\in\mu_{\ell,j}\mathbf{A}\text{ for some
}\ell\in\left[  k+1,j-1\right]  \right)  .
\]

\end{lemma}

\begin{proof}
If $i=j$, then this holds for obvious reasons\footnote{\textit{Proof.} Assume
that $i=j$. Then, $\left[  t_{i},t_{j}\right]  =\left[  t_{j},t_{j}\right]
=0$ (since $\left[  a,a\right]  =0$ for any element $a$ of any ring). Hence,
$\underbrace{\left[  t_{i},t_{j}\right]  }_{=0}\mu_{k,j}=0$. Therefore, we
clearly have $\left(  \left[  t_{i},t_{j}\right]  \mu_{k,j}=0\right)  $ or
$\left(  \left[  t_{i},t_{j}\right]  \mu_{k,j}\in\mu_{\ell,j}\mathbf{A}\text{
for some }\ell\in\left[  k+1,j-1\right]  \right)  $. Thus, Lemma
\ref{lem.increase-mu} is proved under the assumption that $i=j$.}. Hence, for
the rest of this proof, we WLOG assume that $i\neq j$.

We have $i\leq j$ (since $i\in\left[  j\right]  $). Combining this with $i\neq
j$, we obtain $i<j$, so that $i\leq j-1$. In other words, $i\in\left[
j-1\right]  $.

From $k\in\left[  j-1\right]  $, we obtain $1\leq k\leq j-1$, so that
$j-1\geq1$ and therefore $j\geq2$. Hence, $j\in\left[  2,n\right]  $.

We are in one of the following three cases:

\textit{Case 1:} We have $k\geq j-1$.

\textit{Case 2:} We have $i>k$.

\textit{Case 3:} We have neither $k\geq j-1$ nor $i>k$.

Let us first consider Case 1. In this case, we have $k\geq j-1$. Combining
this with $k\leq j-1$, we obtain $k=j-1$.

The definition of $\mu_{k,j}$ yields
\[
\mu_{k,j}=\left(  \underbrace{k}_{=j-1}\Longrightarrow j-1\right)
t_{j-1}=\underbrace{\left(  j-1\Longrightarrow j-1\right)  }%
_{\substack{=1\\\text{(by (\ref{eq.p-cyc-p=1}))}}}t_{j-1}=t_{j-1}.
\]
Hence,
\[
\left[  t_{i},t_{j}\right]  \underbrace{\mu_{k,j}}_{=t_{j-1}}=\left[
t_{i},t_{j}\right]  t_{j-1}=0\ \ \ \ \ \ \ \ \ \ \left(  \text{by Corollary
\ref{cor.titjtj-1}}\right)  .
\]
Thus, we have $\left(  \left[  t_{i},t_{j}\right]  \mu_{k,j}=0\right)  $ or
$\left(  \left[  t_{i},t_{j}\right]  \mu_{k,j}\in\mu_{\ell,j}\mathbf{A}\text{
for some }\ell\in\left[  k+1,j-1\right]  \right)  $. This proves Lemma
\ref{lem.increase-mu} in Case 1.

Let us next consider Case 2. In this case, we have $i>k$. Hence, $i\geq k+1$.
Combined with $i\leq j-1$, this entails $i\in\left[  k+1,j-1\right]  $.
Furthermore, (\ref{eq.lem.com-mu.2}) shows that
\[
\left[  t_{i},t_{j}\right]  =\mu_{i,j}\underbrace{\left(  t_{j}-t_{j-1}%
+1\right)  }_{\in\mathbf{A}}\in\mu_{i,j}\mathbf{A}.
\]

We now know that $i\in\left[  k+1,j-1\right]  $ and $\left[  t_{i}%
,t_{j}\right]  \in\mu_{i,j}\mathbf{A}$. Therefore, $\left[  t_{i}%
,t_{j}\right]  \mu_{k,j}\in\mu_{\ell,j}\mathbf{A}$ for some $\ell\in\left[
k+1,j-1\right]  $ (namely, for $\ell=i$). Thus, we have $\left(  \left[
t_{i},t_{j}\right]  \mu_{k,j}=0\right)  $ or $\left(  \left[  t_{i}%
,t_{j}\right]  \mu_{k,j}\in\mu_{\ell,j}\mathbf{A}\text{ for some }\ell
\in\left[  k+1,j-1\right]  \right)  $. This proves Lemma \ref{lem.increase-mu}
in Case 2.

Finally, let us consider Case 3. In this case, we have neither $k\geq j-1$ nor
$i>k$. In other words, we have $k<j-1$ and $i\leq k$. Thus, $i\leq k<j-1$.
Hence, Lemma \ref{lem.move-mu} yields%
\[
\left[  t_{i},t_{j}\right]  \mu_{k,j}=\mu_{k+1,j}\underbrace{\left[
t_{i},t_{j-1}\right]  }_{\in\mathbf{A}}\in\mu_{k+1,j}\mathbf{A}.
\]
Furthermore, $k<j-1$, so that $k\leq\left(  j-1\right)  -1$. In other words,
$k+1\leq j-1$. Hence, $k+1\in\left[  k+1,j-1\right]  $.

We now know that $k+1\in\left[  k+1,j-1\right]  $ and $\left[  t_{i}%
,t_{j}\right]  \mu_{k,j}\in\mu_{k+1,j}\mathbf{A}$. Hence, $\left[  t_{i}%
,t_{j}\right]  \mu_{k,j}\in\mu_{\ell,j}\mathbf{A}$ for some $\ell\in\left[
k+1,j-1\right]  $ (namely, for $\ell=k+1$). Thus, we have $\left(  \left[
t_{i},t_{j}\right]  \mu_{k,j}=0\right)  $ or $\left(  \left[  t_{i}%
,t_{j}\right]  \mu_{k,j}\in\mu_{\ell,j}\mathbf{A}\text{ for some }\ell
\in\left[  k+1,j-1\right]  \right)  $. This proves Lemma \ref{lem.increase-mu}
in Case 3.

We have now proved Lemma \ref{lem.increase-mu} in each of the three Cases 1, 2
and 3. Hence, this lemma is proved in all situations.
\end{proof}

\subsection{Products of $\left[  t_{i},t_{j}\right]  $'s for a fixed $j$
redux}

For the sake of convenience, we shall restate Lemma \ref{lem.increase-mu} in a
simpler form. To this purpose, we extend Definition \ref{def.muij} somewhat:

\begin{definition}
\label{def.muij2}Let $j\in\left[  n\right]  $, and let $i$ be a positive
integer. In Definition \ref{def.muij2}, we have defined $\mu_{i,j}$ whenever
$i\in\left[  j-1\right]  $. We now set
\[
\mu_{i,j}:=0\in\mathbf{A}\ \ \ \ \ \ \ \ \ \ \text{whenever }i\notin\left[
j-1\right]  .
\]
Thus, $\mu_{i,j}$ is defined for all positive integers $i$ (not just for
$i\in\left[  j-1\right]  $). For example, $\mu_{j,j}=0$ (since $j\notin\left[
j-1\right]  $).
\end{definition}

Using this extended meaning of $\mu_{i,j}$, we can rewrite Lemma
\ref{lem.increase-mu} as follows:

\begin{lemma}
\label{lem.increase-mu2}Let $j\in\left[  n\right]  $, and let $i\in\left[
j\right]  $. Let $k$ be a positive integer. Then,
\[
\left[  t_{i},t_{j}\right]  \mu_{k,j}\in\mu_{\ell,j}\mathbf{A}\text{ for some
integer }\ell\geq k+1.
\]

\end{lemma}

\begin{proof}
If $k\geq j$, then this holds for obvious reasons\footnote{\textit{Proof.}
Assume that $k\geq j$. Thus, $k\geq j>j-1$, so that $k\notin\left[
j-1\right]  $ and therefore $\mu_{k,j}=0$ (by Definition \ref{def.muij2}).
Hence,%
\[
\left[  t_{i},t_{j}\right]  \underbrace{\mu_{k,j}}_{=0}=0=\mu_{k+1,j}%
\cdot\underbrace{0}_{\in\mathbf{A}}\in\mu_{k+1,j}\mathbf{A}.
\]
Hence, $\left[  t_{i},t_{j}\right]  \mu_{k,j}\in\mu_{\ell,j}\mathbf{A}$ for
some integer $\ell\geq k+1$ (namely, for $\ell=k+1$). Thus, Lemma
\ref{lem.increase-mu2} is proved under the assumption that $k\geq j$.}. Hence,
for the rest of this proof, we WLOG assume that $k<j$. Thus, $k\in\left[
j-1\right]  $ (since $k$ is a positive integer). Therefore, Lemma
\ref{lem.increase-mu} yields that we have%
\[
\left(  \left[  t_{i},t_{j}\right]  \mu_{k,j}=0\right)  \text{ or }\left(
\left[  t_{i},t_{j}\right]  \mu_{k,j}\in\mu_{\ell,j}\mathbf{A}\text{ for some
}\ell\in\left[  k+1,j-1\right]  \right)  .
\]
In other words, we are in one of the following cases:

\textit{Case 1:} We have $\left[  t_{i},t_{j}\right]  \mu_{k,j}=0$.

\textit{Case 2:} We have $\left[  t_{i},t_{j}\right]  \mu_{k,j}\in\mu_{\ell
,j}\mathbf{A}$ for some $\ell\in\left[  k+1,j-1\right]  $.

Let us first consider Case 1. In this case, we have $\left[  t_{i}%
,t_{j}\right]  \mu_{k,j}=0$. Hence, $\left[  t_{i},t_{j}\right]  \mu
_{k,j}=0=\mu_{k+1,j}\cdot\underbrace{0}_{\in\mathbf{A}}\in\mu_{k+1,j}%
\mathbf{A}$. Hence, $\left[  t_{i},t_{j}\right]  \mu_{k,j}\in\mu_{\ell
,j}\mathbf{A}$ for some integer $\ell\geq k+1$ (namely, for $\ell=k+1$). Thus,
Lemma \ref{lem.increase-mu2} is proved in Case 1.

Let us now consider Case 2. In this case, we have $\left[  t_{i},t_{j}\right]
\mu_{k,j}\in\mu_{\ell,j}\mathbf{A}$ for some $\ell\in\left[  k+1,j-1\right]
$. Hence, we have $\left[  t_{i},t_{j}\right]  \mu_{k,j}\in\mu_{\ell
,j}\mathbf{A}$ for some integer $\ell\geq k+1$ (because any $\ell\in\left[
k+1,j-1\right]  $ is an integer $\geq k+1$). Thus, Lemma
\ref{lem.increase-mu2} is proved in Case 2.

We have now proved Lemma \ref{lem.increase-mu2} in both Cases 1 and 2. Hence,
Lemma \ref{lem.increase-mu2} is proved in all situations.
\end{proof}

The next lemma is similar to Lemma \ref{lem.right-bound-1}, and will play a
similar role:

\begin{lemma}
\label{lem.left-bound-1}Let $j\in\left[  n\right]  $. Let $k$ be a positive
integer, and let $m\in\mathbb{N}$. Let $i_{1},i_{2},\ldots,i_{m}$ be $m$
elements of $\left[  j\right]  $ (not necessarily distinct). Then,%
\[
\left[  t_{i_{m}},t_{j}\right]  \left[  t_{i_{m-1}},t_{j}\right]
\cdots\left[  t_{i_{1}},t_{j}\right]  \mu_{k,j}\in\mu_{\ell,j}\mathbf{A}\text{
for some integer }\ell\geq k+m.
\]

\end{lemma}

\begin{proof}
We shall show that for each $v\in\left\{  0,1,\ldots,m\right\}  $, we have
\begin{equation}
\left[  t_{i_{v}},t_{j}\right]  \left[  t_{i_{v-1}},t_{j}\right]
\cdots\left[  t_{i_{1}},t_{j}\right]  \mu_{k,j}\in\mu_{\ell,j}\mathbf{A}\text{
for some integer }\ell\geq k+v. \label{pf.lem.left-bound-1.1}%
\end{equation}

In fact, we shall prove (\ref{pf.lem.left-bound-1.1}) by induction on $v$:

\textit{Base case:} Let us check that (\ref{pf.lem.left-bound-1.1}) holds for
$v=0$. Indeed,
\[
\underbrace{\left[  t_{i_{0}},t_{j}\right]  \left[  t_{i_{0-1}},t_{j}\right]
\cdots\left[  t_{i_{1}},t_{j}\right]  }_{=\left(  \text{empty product}\right)
=1}\mu_{k,j}=\mu_{k,j}=\mu_{k,j}\underbrace{1}_{\in\mathbf{A}}\in\mu
_{k,j}\mathbf{A}.
\]
Thus, $\left[  t_{i_{0}},t_{j}\right]  \left[  t_{i_{0-1}},t_{j}\right]
\cdots\left[  t_{i_{1}},t_{j}\right]  \mu_{k,j}\in\mu_{\ell,j}\mathbf{A}$ for
some integer $\ell\geq k+0$ (namely, for $\ell=k$). In other words,
(\ref{pf.lem.left-bound-1.1}) holds for $v=0$. This completes the base case.

\textit{Induction step:} Let $v\in\left\{  0,1,\ldots,m-1\right\}  $. Assume
(as the induction hypothesis) that (\ref{pf.lem.left-bound-1.1}) holds for
$v$. We must prove that (\ref{pf.lem.left-bound-1.1}) holds for $v+1$ instead
of $v$. In other words, we must prove that%
\[
\left[  t_{i_{v+1}},t_{j}\right]  \left[  t_{i_{v}},t_{j}\right]
\cdots\left[  t_{i_{1}},t_{j}\right]  \mu_{k,j}\in\mu_{\ell,j}\mathbf{A}\text{
for some integer }\ell\geq k+\left(  v+1\right)  .
\]

Our induction hypothesis says that (\ref{pf.lem.left-bound-1.1}) holds for
$v$. In other words, it says that%
\[
\left[  t_{i_{v}},t_{j}\right]  \left[  t_{i_{v-1}},t_{j}\right]
\cdots\left[  t_{i_{1}},t_{j}\right]  \mu_{k,j}\in\mu_{\ell,j}\mathbf{A}\text{
for some integer }\ell\geq k+v.
\]
Let us denote this integer $\ell$ by $w$. Thus, $w\geq k+v$ is an integer and
satisfies%
\begin{equation}
\left[  t_{i_{v}},t_{j}\right]  \left[  t_{i_{v-1}},t_{j}\right]
\cdots\left[  t_{i_{1}},t_{j}\right]  \mu_{k,j}\in\mu_{w,j}\mathbf{A}.
\label{pf.lem.left-bound-1.IS.3}%
\end{equation}

However, $w\geq k+v\geq k$, so that $w$ is a positive integer. Also,
$i_{v+1}\in\left[  j\right]  $. Thus, Lemma \ref{lem.increase-mu2} (applied to
$i_{v+1}$ and $w$ instead of $i$ and $k$) yields that
\begin{equation}
\left[  t_{i_{v+1}},t_{j}\right]  \mu_{w,j}\in\mu_{\ell,j}\mathbf{A}\text{ for
some integer }\ell\geq w+1. \label{pf.lem.left-bound-1.IS.4}%
\end{equation}
Consider this $\ell$. Thus, $\ell\geq\underbrace{w}_{\geq k+v}+\,1\geq
k+v+1=k+\left(  v+1\right)  $. Furthermore,%
\begin{align*}
\underbrace{\left[  t_{i_{v+1}},t_{j}\right]  \left[  t_{i_{v}},t_{j}\right]
\cdots\left[  t_{i_{1}},t_{j}\right]  }_{=\left[  t_{i_{v+1}},t_{j}\right]
\cdot\left(  \left[  t_{i_{v}},t_{j}\right]  \left[  t_{i_{v-1}},t_{j}\right]
\cdots\left[  t_{i_{1}},t_{j}\right]  \right)  }\mu_{k,j}  &  =\left[
t_{i_{v+1}},t_{j}\right]  \cdot\underbrace{\left(  \left[  t_{i_{v}}%
,t_{j}\right]  \left[  t_{i_{v-1}},t_{j}\right]  \cdots\left[  t_{i_{1}}%
,t_{j}\right]  \right)  \mu_{k,j}}_{\substack{\in\mu_{w,j}\mathbf{A}%
\\\text{(by (\ref{pf.lem.left-bound-1.IS.3}))}}}\\
&  \in\underbrace{\left[  t_{i_{v+1}},t_{j}\right]  \mu_{w,j}}_{\substack{\in
\mu_{\ell,j}\mathbf{A}\\\text{(by (\ref{pf.lem.left-bound-1.IS.4}))}%
}}\mathbf{A}\subseteq\mu_{\ell,j}\underbrace{\mathbf{AA}}_{\subseteq
\mathbf{A}}\subseteq\mu_{\ell,j}\mathbf{A}.
\end{align*}

Thus, we have found an integer $\ell\geq k+\left(  v+1\right)  $ that
satisfies \newline$\left[  t_{i_{v+1}},t_{j}\right]  \left[  t_{i_{v}}%
,t_{j}\right]  \cdots\left[  t_{i_{1}},t_{j}\right]  \mu_{k,j}\in\mu_{\ell
,j}\mathbf{A}$. Hence, we have shown that%
\[
\left[  t_{i_{v+1}},t_{j}\right]  \left[  t_{i_{v}},t_{j}\right]
\cdots\left[  t_{i_{1}},t_{j}\right]  \mu_{k,j}\in\mu_{\ell,j}\mathbf{A}\text{
for some integer }\ell\geq k+\left(  v+1\right)  .
\]
In other words, (\ref{pf.lem.left-bound-1.1}) holds for $v+1$ instead of $v$.
This completes the induction step. Thus, (\ref{pf.lem.left-bound-1.1}) is
proved by induction on $v$.

Therefore, we can apply (\ref{pf.lem.left-bound-1.1}) to $v=m$. We obtain%
\[
\left[  t_{i_{m}},t_{j}\right]  \left[  t_{i_{m-1}},t_{j}\right]
\cdots\left[  t_{i_{1}},t_{j}\right]  \mu_{k,j}\in\mu_{\ell,j}\mathbf{A}\text{
for some integer }\ell\geq k+m.
\]
This proves Lemma \ref{lem.left-bound-1}.
\end{proof}

Now, we can show our second main result:

\begin{theorem}
\label{thm.left-bound}Let $j\in\left[  n\right]  $, and let $m$ be a positive
integer. Let $k_{1},k_{2},\ldots,k_{m}$ be any $m$ elements of $\left[
j\right]  $ (not necessarily distinct) satisfying $m\geq j-k_{m}+1$. Then,
\[
\left[  t_{k_{1}},t_{j}\right]  \left[  t_{k_{2}},t_{j}\right]  \cdots\left[
t_{k_{m}},t_{j}\right]  =0.
\]

\end{theorem}

\begin{proof}
If $k_{m}=j$, then this claim is obvious\footnote{\textit{Proof.} Assume that
$k_{m}=j$. Thus, $\left[  t_{k_{m}},t_{j}\right]  =\left[  t_{j},t_{j}\right]
=0$ (since $\left[  a,a\right]  =0$ for any element $a$ of any ring). In other
words, the last factor of the product $\left[  t_{k_{1}},t_{j}\right]  \left[
t_{k_{2}},t_{j}\right]  \cdots\left[  t_{k_{m}},t_{j}\right]  $ is $0$. Thus,
this whole product must equal $0$. In other words, $\left[  t_{k_{1}}%
,t_{j}\right]  \left[  t_{k_{2}},t_{j}\right]  \cdots\left[  t_{k_{m}}%
,t_{j}\right]  =0$. This proves Theorem \ref{thm.left-bound} under the
assumption that $k_{m}=j$.}. Hence, for the rest of this proof, we WLOG assume
that $k_{m}\neq j$. Combining this with $k_{m}\leq j$ (since $k_{m}\in\left[
j\right]  $), we obtain $k_{m}<j$. Hence, $k_{m}\in\left[  j-1\right]  $.
Therefore, (\ref{eq.lem.com-mu.2}) (applied to $i=k_{m}$) yields
\begin{equation}
\left[  t_{k_{m}},t_{j}\right]  =\mu_{k_{m},j}\left(  t_{j}-t_{j-1}+1\right)
. \label{pf.thm.left-bound.1}%
\end{equation}

Now, we have $m-1\in\mathbb{N}$ (since $m$ is a positive integer). Let us
define an $\left(  m-1\right)  $-tuple $\left(  i_{1},i_{2},\ldots
,i_{m-1}\right)  $ of elements of $\left[  j\right]  $ by%
\[
\left(  i_{1},i_{2},\ldots,i_{m-1}\right)  :=\left(  k_{m-1},k_{m-2}%
,\ldots,k_{1}\right)
\]
(that is, $i_{v}:=k_{m-v}$ for each $v\in\left[  m-1\right]  $). Then,
$i_{1},i_{2},\ldots,i_{m-1}$ are $m-1$ elements of $\left[  j\right]  $.
Hence, Lemma \ref{lem.left-bound-1} (applied to $m-1$ and $k_{m}$ instead of
$m$ and $k$) yields
\[
\left[  t_{i_{m-1}},t_{j}\right]  \left[  t_{i_{\left(  m-1\right)  -1}}%
,t_{j}\right]  \cdots\left[  t_{i_{1}},t_{j}\right]  \mu_{k_{m},j}\in\mu
_{\ell,j}\mathbf{A}\text{ for some integer }\ell\geq k_{m}+\left(  m-1\right)
.
\]
Consider this $\ell$. We have%
\[
\ell\geq k_{m}+\left(  m-1\right)  =k_{m}+\underbrace{m}_{\geq j-k_{m}%
+1}-\,1\geq k_{m}+j-k_{m}+1-1=j>j-1,
\]
so that $\ell\notin\left[  j-1\right]  $. Therefore, $\mu_{\ell,j}=0$ (by
Definition \ref{def.muij2}). Hence,%
\[
\left[  t_{i_{m-1}},t_{j}\right]  \left[  t_{i_{\left(  m-1\right)  -1}}%
,t_{j}\right]  \cdots\left[  t_{i_{1}},t_{j}\right]  \mu_{k_{m},j}%
\in\underbrace{\mu_{\ell,j}}_{=0}\mathbf{A}=0\mathbf{A}=0.
\]
In other words,%
\begin{equation}
\left[  t_{i_{m-1}},t_{j}\right]  \left[  t_{i_{\left(  m-1\right)  -1}}%
,t_{j}\right]  \cdots\left[  t_{i_{1}},t_{j}\right]  \mu_{k_{m},j}=0.
\label{pf.thm.left-bound.4}%
\end{equation}
However, from the equality $\left(  i_{1},i_{2},\ldots,i_{m-1}\right)
=\left(  k_{m-1},k_{m-2},\ldots,k_{1}\right)  $, we immediately obtain
$\left(  i_{m-1},i_{\left(  m-1\right)  -1},\ldots,i_{1}\right)  =\left(
k_{1},k_{2},\ldots,k_{m-1}\right)  $. Therefore,%
\[
\left[  t_{i_{m-1}},t_{j}\right]  \left[  t_{i_{\left(  m-1\right)  -1}}%
,t_{j}\right]  \cdots\left[  t_{i_{1}},t_{j}\right]  =\left[  t_{k_{1}}%
,t_{j}\right]  \left[  t_{k_{2}},t_{j}\right]  \cdots\left[  t_{k_{m-1}}%
,t_{j}\right]  .
\]
Thus, we can rewrite (\ref{pf.thm.left-bound.4}) as%
\begin{equation}
\left[  t_{k_{1}},t_{j}\right]  \left[  t_{k_{2}},t_{j}\right]  \cdots\left[
t_{k_{m-1}},t_{j}\right]  \mu_{k_{m},j}=0. \label{pf.thm.left-bound.6}%
\end{equation}
Now,%
\begin{align*}
\left[  t_{k_{1}},t_{j}\right]  \left[  t_{k_{2}},t_{j}\right]  \cdots\left[
t_{k_{m}},t_{j}\right]   &  =\left[  t_{k_{1}},t_{j}\right]  \left[  t_{k_{2}%
},t_{j}\right]  \cdots\left[  t_{k_{m-1}},t_{j}\right]  \underbrace{\left[
t_{k_{m}},t_{j}\right]  }_{\substack{=\mu_{k_{m},j}\left(  t_{j}%
-t_{j-1}+1\right)  \\\text{(by (\ref{pf.thm.left-bound.1}))}}}\\
&  =\underbrace{\left[  t_{k_{1}},t_{j}\right]  \left[  t_{k_{2}}%
,t_{j}\right]  \cdots\left[  t_{k_{m-1}},t_{j}\right]  \mu_{k_{m},j}%
}_{\substack{=0\\\text{(by (\ref{pf.thm.left-bound.6}))}}}\left(
t_{j}-t_{j-1}+1\right)  =0.
\end{align*}
This proves Theorem \ref{thm.left-bound}.
\end{proof}

\subsection{The identity $\left[  t_{i},t_{j}\right]  ^{j-i+1}=0$ for all
$i\leq j$}

As a particular case of Theorem \ref{thm.left-bound}, we obtain the following:

\begin{corollary}
\label{cor.left-bound}Let $i,j\in\left[  n\right]  $ be such that $i\leq j$.
Then, $\left[  t_{i},t_{j}\right]  ^{j-i+1}=0$.
\end{corollary}

\begin{proof}
We have $j-i\geq0$ (since $i\leq j$) and thus $j-i+1\geq1$. Hence, $j-i+1$ is
a positive integer. Moreover, $i$ is an element of $\left[  j\right]  $ (since
$i\leq j$) and we have $j-i+1\geq j-i+1$. Hence, Theorem \ref{thm.left-bound}
(applied to $m=j-i+1$ and $k_{r}=i$) yields $\underbrace{\left[  t_{i}%
,t_{j}\right]  \left[  t_{i},t_{j}\right]  \cdots\left[  t_{i},t_{j}\right]
}_{j-i+1\text{ times}}=0$. Thus, $\left[  t_{i},t_{j}\right]  ^{j-i+1}%
=\underbrace{\left[  t_{i},t_{j}\right]  \left[  t_{i},t_{j}\right]
\cdots\left[  t_{i},t_{j}\right]  }_{j-i+1\text{ times}}=0$. This proves
Corollary \ref{cor.left-bound}.
\end{proof}

\section{\label{sec.further}Further directions}

\subsection{More identities?}

A few other properties of somewhere-to-below shuffles can be shown. For
example, the proofs of the following two propositions are left to the reader:

\begin{proposition}
We have $t_{i}=\sum_{k=i}^{j-1}s_{i}s_{i+1}\cdots s_{k-1}+s_{i}s_{i+1}\cdots
s_{j-1}t_{j}$ for any $1\leq i<j\leq n$.
\end{proposition}

\begin{proposition}
Let $i,j\in\left[  n-1\right]  $ be such that $i\leq j$. Then, $\left[
t_{i},t_{j}\right]  =\left[  s_{i}s_{i+1}\cdots s_{j-1},\ s_{j}\right]
t_{j+1}t_{j}$.
\end{proposition}

\begin{proposition}
Set $B_{i}:=\prod_{k=0}^{i-1}\left(  t_{1}-k\right)  $ for each $i\in\left[
0,n\right]  $. Then, $B_{i}=t_{i}B_{i-1}$ for each $i\in\left[  n\right]  $.
\end{proposition}

We wonder to what extent the identities that hold for $t_{1},t_{2}%
,\ldots,t_{n}$ can be described. For instance, we can ask:

\begin{question}
\ \ 

\begin{enumerate}
\item[\textbf{(a)}] What are generators and relations for the $\mathbb{Q}%
$-algebra $\mathbb{Q}\left[  t_{1},t_{2},\ldots,t_{n}\right]  $ for a given
$n\in\mathbb{N}$ ?

\item[\textbf{(b)}] Fix $k\in\mathbb{N}$. What identities hold for
$t_{1},t_{2},\ldots,t_{k}$ for \textbf{all} $n$ ? Is there a single algebra
that \textquotedblleft governs\textquotedblright\ the relations between
$t_{1},t_{2},t_{3},\ldots$ that hold independently of $n$ ?

\item[\textbf{(c)}] If a relation between $t_{1},t_{2},\ldots,t_{k}$ holds for
all sufficiently high $n\geq k$, must it then hold for all $n\geq k$ ?
\end{enumerate}
\end{question}

We suspect that these questions are hard to answer, as we saw in Remark
\ref{rmk.ti+kti} that even the quadratic relations between $t_{1},t_{2}%
,\ldots,t_{n}$ exhibit some rather finicky behavior. The dimension of
$\mathbb{Q}\left[  t_{1},t_{2},\ldots,t_{n}\right]  $ as a $\mathbb{Q}$-vector
space does not seem to follow a simple rule either (see (\ref{eq.dimQ}) for
the first few values), although there appear to be some patterns in how this
dimension is generated\footnote{Namely, for all $n\leq8$, we have verified
that the algebra $\mathbb{Q}\left[  t_{1},t_{2},\ldots,t_{n}\right]  $ is
generated by products of $m$ somewhere-to-below shuffles with $m\in\left\{
0,1,\ldots,n-1\right\}  $, and moreover, only one such product for $m=n-1$ is
needed.}.

Another question, which we have already touched upon in Subsection
\ref{subsec.ikinn}, is the following:

\begin{question}
Fix $j\in\left[  n\right]  $. What is the smallest $h\in\mathbb{N}$ such that
we have $\left[  t_{i_{1}},t_{j}\right]  \left[  t_{i_{2}},t_{j}\right]
\cdots\left[  t_{i_{h}},t_{j}\right]  =0$ for all $i_{1},i_{2},\ldots,i_{h}%
\in\left[  n\right]  $ (as opposed to holding only for $i_{1},i_{2}%
,\ldots,i_{h}\in\left[  j\right]  $ )?
\end{question}

\subsection{\label{subsec.further.opti}Optimal exponents?}

Corollary \ref{cor.right-bound} and Corollary \ref{cor.left-bound} give two
different answers to the question \textquotedblleft what powers of $\left[
t_{i},t_{j}\right]  $ are $0$?\textquotedblright. One might dare to ask for
the \textbf{smallest} such power (more precisely, the smallest such exponent).
In other words:

\begin{question}
Given $i,j\in\left[  n\right]  $, what is the smallest $m\in\mathbb{N}$ such
that $\left[  t_{i},t_{j}\right]  ^{m}=0$ ? (We assume $\mathbf{k}=\mathbb{Z}$
here to avoid small-characteristic cancellations.)
\end{question}

We conjecture that this smallest $m$ is $\min\left\{  j-i+1,\ \left\lceil
\left(  n-j\right)  /2\right\rceil +1\right\}  $ whenever $i<j$ (so that
whichever of Corollary \ref{cor.right-bound} and Corollary
\ref{cor.left-bound} gives the better bound actually gives the optimal bound).
Using SageMath, this conjecture has been verified for all $n\leq12$.

\subsection{Generalizing to the Hecke algebra}

The \emph{type-A Hecke algebra} (also known as the \emph{type-A Iwahori-Hecke
algebra}) is a deformation of the group algebra $\mathbf{k}\left[
S_{n}\right]  $ involving a new parameter $q\in\mathbf{k}$. It is commonly
denoted by $\mathcal{H}=\mathcal{H}_{q}\left(  S_{n}\right)  $; it has a basis
$\left(  T_{w}\right)  _{w\in S_{n}}$ indexed by the permutations $w\in S_{n}%
$, but its multiplication is more complicated than composing the indexing
permutations. We refer to \cite{Mathas99} for the definition and a deep study
of this algebra. We can define the $q$-deformed somewhere-to-below shuffles
$t_{1}^{\mathcal{H}},t_{2}^{\mathcal{H}},\ldots,t_{n}^{\mathcal{H}}$ by%
\[
t_{\ell}^{\mathcal{H}}:=T_{\operatorname*{cyc}\nolimits_{\ell}}%
+T_{\operatorname*{cyc}\nolimits_{\ell,\ell+1}}+T_{\operatorname*{cyc}%
\nolimits_{\ell,\ell+1,\ell+2}}+\cdots+T_{\operatorname*{cyc}\nolimits_{\ell
,\ell+1,\ldots,n}}\in\mathcal{H}.
\]
Surprisingly, it seems that many of the properties of the original
somewhere-to-below $t_{1},t_{2},\ldots,t_{n}$ still hold for these
deformations. In particular:

\begin{conjecture}
Corollary \ref{cor.left-bound} and Corollary \ref{cor.right-bound} both seem
to hold in $\mathcal{H}$ when the $t_{\ell}$ are replaced by the $t_{\ell
}^{\mathcal{H}}$.
\end{conjecture}

This generalization is not automatic. Our above proofs do not directly apply
to $\mathcal{H}$, as (for example) Lemma \ref{lem.jpiq2} does not generalize
to $\mathcal{H}$. The $\mathcal{H}$-generalization of Theorem \ref{thm.ti+1ti}
appears to be%
\begin{equation}
qt_{i+1}^{\mathcal{H}}t_{i}^{\mathcal{H}}=\left(  t_{i}^{\mathcal{H}%
}-1\right)  t_{i}^{\mathcal{H}}=t_{i}^{\mathcal{H}}\left(  t_{i}^{\mathcal{H}%
}-1\right)  \label{eq.hecke.ti+1ti}%
\end{equation}
(verified using SageMath for all $n\leq11$). (The $q$ on the left hand side is
necessary; the product $t_{i+1}^{\mathcal{H}}t_{i}^{\mathcal{H}}$ is not a
$\mathbb{Z}$-linear combination of $1$, $t_{i}^{\mathcal{H}}$ and $\left(
t_{i}^{\mathcal{H}}\right)  ^{2}$ when $q=0$.) Our proof of Theorem
\ref{thm.ti+1ti} does not seem to adapt to (\ref{eq.hecke.ti+1ti}), and while
we suspect that proving (\ref{eq.hecke.ti+1ti}) won't be too difficult, it is
merely the first step.

\subsection{One-sided cycle shuffles}

We return to $\mathbf{k}\left[  S_{n}\right]  $.

The $\mathbf{k}$-linear combinations $\lambda_{1}t_{1}+\lambda_{2}t_{2}%
+\cdots+\lambda_{n}t_{n}$ (with $\lambda_{1},\lambda_{2},\ldots,\lambda_{n}%
\in\mathbf{k}$) of the somewhere-to-below shuffles are called the
\emph{one-sided cycle shuffles}. They have been studied in \cite{s2b1}. Again,
the main result of \cite{s2b1} entails that their commutators are nilpotent,
but we can ask \textquotedblleft how nilpotent?\textquotedblright.

This question remains wide open, not least due to its computational complexity
(even the $n=6$ case brings SageMath to its limits). All that I can say with
surety is that the commutators of one-sided cycle shuffles don't vanish as
quickly (under taking powers) as the $\left[  t_{i},t_{j}\right]  $'s.

\begin{example}
For instance, let us set $n=6$ and choose arbitrary $a,b,c,d,e,a^{\prime
},b^{\prime},c^{\prime},d^{\prime},e^{\prime}\in\mathbf{k}$, and then
introduce the elements%
\begin{align*}
u  &  :=at_{1}+bt_{2}+ct_{3}+dt_{4}+et_{5}\ \ \ \ \ \ \ \ \ \ \text{and}\\
u^{\prime}  &  :=a^{\prime}t_{1}+b^{\prime}t_{2}+c^{\prime}t_{3}+d^{\prime
}t_{4}+e^{\prime}t_{5}%
\end{align*}
(two completely generic one-sided shuffles, except that omit $t_{6}$ terms
since $t_{6}=1$ does not influence the commutator). Then, 10 minutes of
torturing SageMath reveals that $\left[  u,u^{\prime}\right]  ^{6}=0$, but
$\left[  u,u^{\prime}\right]  ^{5}$ is generally nonzero.
\end{example}

Even this example is misleadingly well-behaved. For $n=7$, it is not hard to
find two one-sided cycle shuffles $u,u^{\prime}$ such that $\left[
u,u^{\prime}\right]  ^{n}\neq0$.

\begin{question}
For each given $n$, what is the smallest (or at least a reasonably small)
$m\in\mathbb{N}$ such that every two one-sided cycle shuffles $u,u^{\prime}$
satisfy $\left[  u,u^{\prime}\right]  ^{m}=0$ ?
\end{question}

\bibliographystyle{halpha-abbrv}
\bibliography{biblio.bib}

\end{document}